\documentclass[oneside,english,onecolumn]{amsart}
\usepackage[T1]{fontenc}
\usepackage[latin9]{inputenc}
\usepackage{xcolor}
\usepackage{float}
\usepackage{amstext}
\usepackage{amsthm}
\usepackage{amssymb}
\usepackage{graphicx}
\usepackage{esint}

\makeatletter

\providecommand{\tabularnewline}{\\}

\numberwithin{equation}{section}
\numberwithin{figure}{section}
\theoremstyle{plain}
\newtheorem{thm}{\protect\theoremname}
\theoremstyle{definition}
\newtheorem{defn}[thm]{\protect\definitionname}
\theoremstyle{plain}
\newtheorem{lem}[thm]{\protect\lemmaname}
\theoremstyle{plain}
\newtheorem{prop}[thm]{\protect\propositionname}
\theoremstyle{plain}
\newtheorem{cor}[thm]{\protect\corollaryname}

\usepackage{mathrsfs}

\usepackage[foot]{amsaddr}

\usepackage{enumitem}
\setlist[itemize]{leftmargin=10pt}

\usepackage{quiver}

\usepackage{colortbl}
\definecolor{lightgray}{rgb}{0.9,0.9,0.9}
\definecolor{lightred}{rgb}{1,0.8,0.8}
\definecolor{lightgreen}{rgb}{0.6,1,0.6}
\definecolor{lightyellow}{rgb}{1,1,0.5}
\definecolor{lightgrey}{rgb}{0.8,0.8,0.8}

\allowdisplaybreaks[1]

\makeatother

\usepackage{babel}
\providecommand{\corollaryname}{Corollary}
\providecommand{\definitionname}{Definition}
\providecommand{\lemmaname}{Lemma}
\providecommand{\propositionname}{Proposition}
\providecommand{\theoremname}{Theorem}

\begin{document}
\title[\resizebox{4.3in}{!}{A Characterization of Entropy as a Universal Monoidal Natural Transformation}]{A Characterization of Entropy as a Universal Monoidal Natural Transformation}
\author{Cheuk Ting Li}
\address{Department of Information Engineering, The Chinese University of Hong
Kong, Hong Kong SAR of China}
\email{ctli@ie.cuhk.edu.hk}
\begin{abstract}
We show that the essential properties of entropy (monotonicity, additivity
and subadditivity) are consequences of entropy being a monoidal natural
transformation from the under category functor $-/\mathsf{LProb}_{\rho}$
(where $\mathsf{LProb}_{\rho}$ is category of $\rho$-th-power-summable
probability distributions, $0<\rho<1$) to $\Delta_{\mathbb{R}}$.
Moreover, the Shannon entropy can be characterized as the universal
monoidal natural transformation from $-/\mathsf{LProb}_{\rho}$ to
the category of integrally closed partially ordered abelian groups
(a reflective subcategory of the lax-slice 2-category over $\mathsf{MonCat}_{\ell}$
in the 2-category of monoidal categories), providing a succinct characterization
of Shannon entropy as a reflection arrow. We can likewise define entropy
for every monoidal category with a monoidal structure on its under
categories (e.g. the category of finite abelian groups, the category
of finite inhabited sets, the category of finite dimensional vector
spaces, and the augmented simplex category) via the reflection arrow.
This implies that all these entropies over different categories are
components of a single natural transformation (the unit of the idempotent
monad), allowing us to connect these entropies in a natural manner.
We also provide a universal characterization of the conditional Shannon
entropy based on the chain rule which, unlike the characterization
of information loss by Baez, Fritz and Leinster, does not require
any continuity assumption.

\end{abstract}

\maketitle
\medskip{}

\tableofcontents{}

\medskip{}

\section{Introduction}

Consider the category $\mathsf{FinProb}$ of finite probability spaces,
with morphisms given by measure-preserving functions between probability
spaces, and tensor product given by the product probability space
\cite{franz2002stochastic,baez2011characterization,baez2011entropy}.
For a probability space $P\in\mathsf{FinProb}$, the under category
$P/\mathsf{FinProb}$, consisting of measure-preserving functions
from $P$ to another probability space, can be regarded as the ``category
of random variables defined over $P$'', with a categorical product
that corresponds to the joint random variable of two random variables
\cite{baez2011entropy}. It was observed by Baez, Fritz and Leinster
\cite{baez2011entropy} that the Shannon entropy $H_{1}:\mathsf{FinProb}\to[0,\infty)$
can be regarded as a strong monoidal functor (where $[0,\infty)$
is the monoidal posetal category, where there is a morphism $a\to b$
for $a,b\in[0,\infty)$ if $a\ge b$, with tensor product given by
addition), which corresponds to the \emph{monotonicity property} of
entropy: 
\begin{equation}
H_{1}(X)\ge H_{1}(Y)\label{eq:monotone}
\end{equation}
for discrete random variables $X,Y$ where $Y$ is a function of $X$,
due to $H_{1}$ being a functor; and the \emph{additivity property}
of entropy: 
\begin{equation}
H_{1}(X)+H_{1}(Y)=H_{1}(X,Y)\label{eq:additive}
\end{equation}
for any independent discrete random variables $X,Y$, due to $H_{1}$
being strongly monoidal. Also, the functor mapping a random variable
$x\in P/\mathsf{FinProb}$ to its Shannon entropy is a lax monoidal
functor \cite{baez2011entropy}, which corresponds to the \emph{subadditivity
property} of entropy: 
\begin{equation}
H_{1}(X)+H_{1}(Y)\ge H_{1}(X,Y)\label{eq:subadditive}
\end{equation}
for any jointly-distributed discrete random variables $X,Y$. This
way, three important properties of entropy can be stated as properties
of strong/lax monoidal functors. Note that these properties are also
satisfied by the Hartley entropy \cite{hartley1928transmission} $H_{0}(X)=\log|\mathrm{supp}(X)|$,
where $\mathrm{supp}(X)$ is the support set of $X$.

The purpose of this paper is to demonstrate that these three properties
are consequences of the fact that the natural transformation $h:-/\mathsf{FinProb}\Rightarrow\Delta_{\mathbb{R}}$
(where $-/\mathsf{FinProb}:\mathsf{FinProb}^{\mathrm{op}}\to\mathsf{MonCat}_{\ell}$
maps $P\in\mathsf{FinProb}$ to $P/\mathsf{FinProb}\in\mathsf{MonCat}_{\ell}$,
and $\Delta_{\mathbb{R}}:\mathsf{FinProb}^{\mathrm{op}}\to\mathsf{MonCat}_{\ell}$
is the constant functor mapping everything to the monoidal posetal
category $\mathbb{R}\in\mathsf{MonCat}_{\ell}$), with components
$h_{P}(X)=H_{1}(X)$ for random variable $X:P\to Q$, is a monoidal
natural transformation (here we assume $\Delta_{\mathbb{R}}$ is a
lax monoidal functor with coherence map given by the tensor product
of $\mathbb{R}$, i.e., addition over real numbers). We have the following
correspondence between the properties of entropy and the properties
of $h$ being a monoidal natural transformation:
\begin{align*}
\text{Monotonicity}\; & \leftrightarrow\;\text{\ensuremath{h_{P}} is a functor},\\
\text{Subadditivity}\; & \leftrightarrow\;\text{\ensuremath{h_{P}} is lax monoidal},\\
\text{Additivity}\; & \leftrightarrow\;\text{\ensuremath{h} is a monoidal natural transformation}.
\end{align*}

In this paper, we prove the following succinct characterization of
the pairing of the Hartley entropy and the Shannon entropy, given
by $\tilde{h}:-/\mathsf{FinProb}\Rightarrow\Delta_{(\log\mathbb{Q}_{>0})\times\mathbb{R}}$
with components $\tilde{h}_{P}(X)=(H_{0}(X),H_{1}(X))$. This can
be considered as an abstraction and generalization of the characterizations
by Acz{\'e}l, Forte and Ng \cite{aczel1974shannon}. Write $\mathsf{IcOrdAb}$
for the category of integrally closed partially ordered (i.c.p.o.)
abelian group \cite{clifford1940partially,glass1999partially}.\footnote{A partially ordered abelian group is \emph{integrally closed} if $nx\le y$
for all $n\in\mathbb{Z}_{>0}$ implies $x\le0$ \cite{clifford1940partially,glass1999partially}.} We have:

\medskip{}

\textit{The pairing of the Hartley entropy and the Shannon entropy
$\tilde{h}$ is the universal monoidal natural transformation in the
form $\gamma:\,-/\mathsf{FinProb}\Rightarrow\Delta_{\mathsf{W}}:\,\mathsf{FinProb}^{\mathrm{op}}\to\mathsf{MonCat}_{\ell}$
where $\mathsf{W}\in\mathsf{IcOrdAb}$.}\footnote{More explicitly, for every such $\gamma$, there exists a unique morphism
$F:(\log\mathbb{Q}_{>0})\times\mathbb{R}\to\mathsf{W}$ in $\mathsf{IcOrdAb}$
such that $\gamma=\Delta_{F}\tilde{h}$.}\textit{}

\medskip{}

Furthermore, if we consider $\mathsf{LProb}_{\rho}$ ($0<\rho<1$),
the category of $\rho$-th-power-summable discrete probability distributions
(containing probability mass functions $P$ over finite or countable
sets satisfying $\sum_{x}(P(x))^{\rho}<\infty$) instead of $\mathsf{FinProb}$,
then the Hartley entropy is no longer defined, and the Shannon entropy
alone is the universal monoidal natural transformation. This provides
a succinct characterization of Shannon entropy:

\medskip{}

\textit{The Shannon entropy $h$ is the universal monoidal natural
transformation in the form $\gamma:\,-/\mathsf{LProb}_{\rho}\Rightarrow\Delta_{\mathsf{W}}:\,\mathsf{LProb}_{\rho}^{\mathrm{op}}\to\mathsf{MonCat}_{\ell}$
where $\mathsf{W}\in\mathsf{IcOrdAb}$.}

\medskip{}

Intuitively, Shannon entropy is the universal ``monoidal cocone''
from $-/\mathsf{LProb}_{\rho}$ to an i.c.p.o. abelian group. Unlike
previous characterizations of Shannon entropy \cite{shannon1948mathematical,faddeev1956concept,chaundy1960functional,aczel1974shannon,aczel1975measures,baez2011characterization,muller2016generalization},
this characterization does not involve $\mathbb{R}$ outside of the
definition of $\mathsf{LProb}_{\rho}$. Even though this characterization
does not require $\mathsf{W}$ to be $\mathbb{R}$ or linearly ordered,
the structure of $\mathsf{LProb}_{\rho}$ ensures that every monoidal
natural transformation $\gamma:\,-/\mathsf{LProb}_{\rho}\Rightarrow\Delta_{\mathsf{W}}$
must order all probability distributions linearly according to their
entropies, and the structure of $\mathbb{R}$ is ``extracted'' from
$\mathsf{LProb}_{\rho}$. 

A consequence is that the Shannon entropy $H_{1}$ is the universal
function in the form $H:\mathsf{LProb}_{\rho}\to\mathsf{W}$, where
$\mathsf{W}$ is an i.c.p.o. abelian group, satisfying \eqref{eq:monotone},
\eqref{eq:additive} and \eqref{eq:subadditive}.\footnote{Universality means that every such $H:\mathsf{LProb}_{\rho}\to\mathsf{W}$
can be factorized uniquely as $H=fH_{1}$ where $f:\mathbb{R}\to\mathsf{W}$
is an order-preserving group homomorphism.}

Moreover, the characterization of entropy can be applied not only
to $\mathsf{FinProb}$ and $\mathsf{LProb}_{\rho}$, but also to any
monoidal category with a monoidal structure over its under categories,
generalizing the concept of entropy to all these categories. Define
\emph{the category of monoidal strictly-indexed monoidal (MonSiMon)
categories}, written as $\mathsf{MonSiMonCat}$, as the lax-slice
2-category over $\mathsf{MonCat}_{\ell}$ in the 2-category of monoidal
categories. Each object of $\mathsf{MonSiMonCat}$ is a monoidal category
$\mathsf{C}$, where each object $P$ is associated with a monoidal
category $N(P)$ via a lax monoidal functor $N:\mathsf{C}^{\mathrm{op}}\to\mathsf{MonCat}_{\ell}$.
Examples include $\mathsf{FinProb}$ (where each $P\in\mathsf{FinProb}$
is associated with the under category $N(P)=P/\mathsf{FinProb}$),
$\mathsf{LProb}_{\rho}$, the category of finite abelian groups, the
category of finite sets, the category of finite dimensional vector
spaces, the category of Gaussian distributions, and the augmented
simplex category. 

We can define entropies over the aforementioned MonSiMon categories,
i.e., those categories are the domains of the entropy functors. The
codomains of the entropy functors (e.g., $\mathbb{R}$ for Shannon
entropy) are MonSiMon categories as well. The entropy functors are
morphisms in $\mathsf{MonSiMonCat}$. Moreover, in this paper, the
entropy can be given as the universal reflection morphism from $\mathsf{FinProb}$
(or another MonSiMon category $\mathsf{C}$ under consideration) to
the reflective subcategory of allowed codomains in $\mathsf{MonSiMonCat}$
(e.g., cancellative commutative monoids, i.c.p.o. abelian groups,
i.c.p.o. vector spaces). For example, the reflection morphism from
$-/\mathsf{LProb}_{\rho}$ to i.c.p.o. abelian groups is the Shannon
entropy, whereas the reflection morphism from $-/\mathsf{FinProb}$
to i.c.p.o. abelian groups is the pairing  of the Shannon and Hartley
entropies.

Furthermore, the universal entropy of $\mathsf{C}\in\mathsf{MonSiMonCat}$
can be given as the component $\phi_{\mathsf{C}}:\mathsf{C}\to T(\mathsf{C})$,
where $\phi:\mathrm{id}_{\mathsf{MonSiMonCat}}\Rightarrow T$ is the
unit of the idempotent monad $T:\mathsf{MonSiMonCat}\to\mathsf{MonSiMonCat}$
sending a MonSiMon category to its reflection in the reflective subcategory
of allowed codomains, which we call the \emph{universal entropy monad}.
The natural transformation $\phi$ allows us to connect the entropies
of various categories  as components of the same natural transformation.
For example, passing the forgetful functor $\mathsf{FinProb}\to\mathrm{Epi}(\mathsf{FinISet})$
(where $\mathrm{Epi}(\mathsf{FinISet})$ is the category of finite
inhabited sets with surjective functions as morphisms, and the forgetful
functor maps a distribution to its support) to the monad $T$ gives
the functor $\pi_{1}:(\log\mathbb{Q}_{>0})\times\mathbb{R}\to\log\mathbb{Q}_{>0}$
that only retains the Hartley entropy $H_{0}(P)$ among the pair $(H_{0}(P),H_{1}(P))$,
corresponding to the fact that forgetting the probabilities will make
us lose the ability to calculate the Shannon entropy. On the other
hand, passing the inclusion functor $\mathsf{FinProb}\to\mathsf{LProb}_{\rho}$
to the monad $T$ will give $\pi_{2}:(\log\mathbb{Q}_{>0})\times\mathbb{R}\to\mathbb{R}$
that loses the Hartley entropy since it cannot be defined for distributions
with infinite support, and only retain the Shannon entropy. We will
also discuss some insights provided by the naturality of $\phi$,
such as whether the differential entropy or the information dimension
\cite{renyi1959dimension} is a ``nicer'' notion of entropy for
Gaussian random variables, and whether information can be treated
as commodities and vice versa. See Figure \ref{fig:reflect} for an
illustration.

\begin{figure}
\begin{centering}
\includegraphics[scale=1.1]{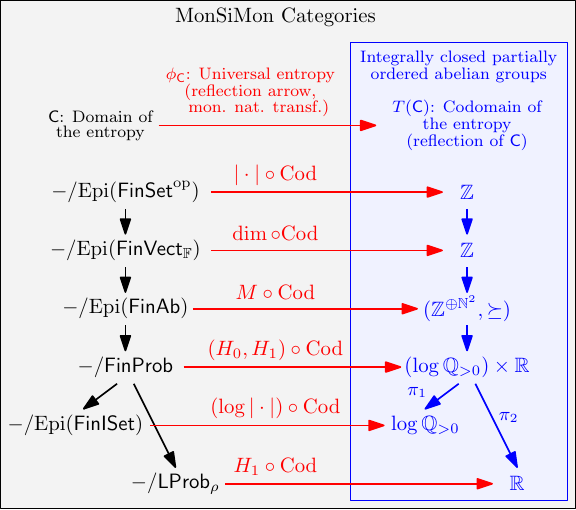}
\par\end{centering}
\caption{\label{fig:reflect}The universal entropy of a MonSiMon category $\mathsf{C}\in\mathsf{MonSiMonCat}$
is the MonSiMon functor $\phi_{\mathsf{C}}:\mathsf{C}\to T(\mathsf{C})$
from $\mathsf{C}$ to its reflection $T(\mathsf{C})$ in the reflective
subcategory of integrally closed partially ordered (i.c.p.o.) abelian
groups, where $T:\mathsf{MonSiMonCat}\to\mathsf{MonSiMonCat}$ is
the idempotent monad (the \emph{universal entropy monad}) sending
a category to its reflection, and $\phi:\mathrm{id}_{\mathsf{MonSiMonCat}}\Rightarrow T$
is the monad unit. We will see that the component $\phi_{\mathsf{C}}$
is basically a monoidal natural transformation to a constant functor.
We can then use the natural transformation $\phi$ to naturally connect
the entropies over various domains, e.g., the (epimorphism wide subcategories
of) the opposite category of $\mathsf{FinSet}$ (with entropy given
by cardinality), the category of finite-dimensional vector spaces
over finite field $\mathbb{F}$ (entropy given by dimension), the
category of finite abelian group (entropy given in Section \ref{subsec:ab}),
the category of finite probability spaces (entropy given by the pairing
of the Hartley entropy $H_{0}$ and the Shannon entropy $H_{1}$),
the category of finite inhabited sets (entropy given by log cardinality)
and the category of $\rho$-th-power-summable discrete probability
spaces ($0<\rho<1$, entropy given by Shannon entropy). Refer to Section
\ref{sec:global} for the descriptions of all the arrows (functors)
in the above diagram (which is a commutative diagram in $\mathsf{MonSiMonCat}$).}
\end{figure}

We also provide a universal characterization of the conditional Shannon
entropy based on the chain rule
\[
H_{1}(X|Z)=H_{1}(X|Y)+H_{1}(Y|Z),
\]
for discrete random variables $X,Y,Z$ where $Z$ is a function of
$Y$, which in turn is a function of $X$, via a construction called
\emph{discretely-indexed monoidal strictly-indexed monoidal category}.
The conditional Shannon entropy can be characterized as a universal
functor  to an i.c.p.o. abelian group. In contrast, for the unconditional
case, the universal functor from $\mathsf{FinProb}$ is the pairing
of the Shannon and Hartley entropies. This provides an alternative
characterization of conditional Shannon entropy or information loss
\cite{baez2011characterization}. Unlike \cite{baez2011characterization},
here the characterization does not require the continuity axiom in
\cite{baez2011characterization}. Note that the continuity axiom is
not stated purely in terms of the category $\mathsf{FinProb}$, but
requires endowing a topology on $\mathsf{FinProb}$. Here we show
that such topological structure is unnecessary for the characterization
of Shannon entropy. Shannon entropy arises naturally from the structure
of $\mathsf{FinProb}$ itself.

\medskip{}

\subsection{Previous Works}

Since the first axiomatic characterization of entropy by Shannon \cite{shannon1948mathematical},
many other axiomatic characterizations of entropy have been proposed
(e.g., \cite{faddeev1956concept,chaundy1960functional,aczel1975measures,muller2016generalization}).
Interested readers are referred to \cite{csiszar2008axiomatic} for
a comprehensive survey. An elegant characterization was given by Acz{\'e}l,
Forte and Ng \cite{aczel1974shannon}, who showed that any function
mapping finite probability distributions to nonnegative real numbers
that satisfies the additivity property \eqref{eq:additive} and the
subadditivity property \eqref{eq:subadditive}, as well as being invariant
under permutation of labels and the addition of a zero mass,\footnote{This means that the entropy of the probability mass function $P:A\to[0,1]$
is the same as the entropy of the probability mass function $Q:B\to[0,1]$
if there is an injective function $f:A\to B$ satisfying that $Q$
is the pushforward measure of $P$ along $f$, or equivalently, $Q(f(x))=P(x)$
for all $x\in A$.} must be a nonnegative linear combination of the Shannon and Hartley
entropies.  Note that \cite{aczel1974shannon} also showed that if
we impose an additional requirement that the function must be continuous,
then it must be a nonnegative multiple of the Shannon entropy. In
comparison, both characterizations of Shannon entropy in this paper
(the entropy of $\mathsf{LProb}_{\rho}$ in Section \ref{sec:lrho},
and the conditional entropy of $\mathsf{FinProb}$ or $\mathsf{LProb}_{\rho}$
in Section \ref{sec:conditional}) do not require continuity. We show
that monotonicity \eqref{eq:monotone}, additivity \eqref{eq:additive}
and subadditivity \eqref{eq:subadditive} (as well as being invariant
under the addition of a zero mass) are sufficient to characterize
Shannon entropy (up to a multiplicative constant), if we consider
$\mathsf{LProb}_{\rho}$ instead of $\mathsf{FinProb}$. The work
by M\"{u}ller and Pastena \cite{muller2016generalization} relates
the Shannon and Hartley entropies with the concept of majorization
among discrete distributions \cite{marshall1979inequalities}, which
will be an important tool in the proofs in this paper.

The categorical treatment of probability theory dates back to Lawvere
\cite{lawvere1962category} and Giry \cite{giry1982categorical}.
It is natural to study whether entropy can be characterized in a category-theoretic
manner. In \cite{baez2011characterization}, the conditional Shannon
entropy, which sends the morphism $f:P\to Q$ in $\mathsf{FinProb}$
to the conditional Shannon entropy $H_{1}(P)-H_{1}(Q)$, is characterized
as the only function $F$ (up to multiplication by a nonnegative number)
sending morphisms in $\mathsf{FinProb}$ to $[0,\infty)$ that obeys:
\begin{itemize}
\item (functoriality / chain rule) $F(gf)=F(f)+F(g)$ for morphisms $P\stackrel{f}{\to}Q\stackrel{g}{\to}R$; 
\item (convex linearity) $F(\lambda f\oplus(1-\lambda)g)=\lambda F(f)+(1-\lambda)F(g)$
for $\lambda\in[0,1]$, $f:P_{1}\to Q_{1}$, $g:P_{2}\to Q_{2}$,
where $\oplus$ denotes the convex mixture of measure-preserving functions
\cite{baez2011characterization}; and 
\item (continuity) $F(f)$ (where $f:P\to Q$) is continuous with respect
to the finite probability measure $P$. 
\end{itemize}
This characterization, while being simple to state, requires a nontrivial
amount of work to be translated into purely category-theoretic language.
 The convex linearity axiom requires a ``convex combination''\footnote{This functor is not a convex combination in the usual sense since
$\lambda P\oplus(1-\lambda)P$ is not isomorphic to $P$ for $P\in\mathsf{FinProb}$.} functor $\lambda(-)\oplus(1-\lambda)(-):\,\mathsf{FinProb}\times\mathsf{FinProb}\to\mathsf{FinProb}$
for every $\lambda\in[0,1]$ (which can be defined using operads \cite{leinster2017categorical}),
and the continuity axiom requires endowing a topology on $\mathsf{FinProb}$.
In comparison, the characterization of conditional Shannon entropy
in this paper (Section \ref{sec:conditional}) only requires the properties
of the over categories $\mathsf{FinProb}/Q$ and their under categories
$f/(\mathsf{FinProb}/Q)$, and is based on familiar notions in category
theory (over/under categories and monoidal categories). 

For other related works, the characterization in \cite{baez2011characterization}
was extended to stochastic maps in \cite{fullwood2021information}.
A categorical characterization of the relative entropy was given in
\cite{baez2014bayesian}. It was extended to the relative entropy
on standard Borel spaces in \cite{gagne2018categorical}. Characterizations
of entropy via operads were studied in \cite{leinster2021entropy,bradley2021entropy}.
Entropy over Markov categories \cite{fritz2020synthetic} has been
studied in \cite{perrone2022markov}.

\begin{table}[H]
\begin{centering}
\makebox[\textwidth][c]{{\renewcommand*{\arraystretch}{1.38}%
\begin{tabular}{ccc}
\textbf{Category} & \textbf{Objects} & \textbf{Morphisms}\tabularnewline
\hline 
$*\;\;\lefteqn{\text{(Terminal category)}}$  & $\;\;$One object $\bullet$ & Identity $\mathrm{id}_{\bullet}$\tabularnewline
\hline 
$\Delta_{+}\;\lefteqn{\text{(Aug. simplex cat.)}}$ & $\;\;$$[n]$ for $n\ge-1$ & Order-preserving maps\tabularnewline
\hline 
$\mathsf{FinISet}$ & Finite inhabited (nonempty) sets & Functions\tabularnewline
\hline 
$\mathsf{FinVect}_{\mathbb{F}}$ & Finite-dimensional vector spaces over the field $\mathbb{F}$ & Linear maps\tabularnewline
\hline 
$\mathsf{FinAb}$ & Finite abelian groups & Homomorphisms\tabularnewline
\hline 
$\mathsf{Mon}$ & Monoids & Homomorphisms\tabularnewline
\hline 
$\mathsf{Mon}(\mathsf{C})$ & Monoids in category $\mathsf{C}$ & $\!\!\!$Morphisms respecting monoids$\!\!\!$\tabularnewline
\hline 
$\mathsf{OrdMon}$ & Ordered monoids (as monoidal posetal categories) & Order-preserving hom.\tabularnewline
\hline 
$\mathsf{OrdCMon}$ & Ordered commutative monoids (as mon. pos. cat.) & Order-preserving hom.\tabularnewline
\hline 
$\mathsf{COrdCMon}$ & Cancellative ord. comm. monoids (Sec. \ref{sec:subcats}) & Order-preserving hom.\tabularnewline
\hline 
$\mathsf{OrdAb}$ & Ordered abelian groups (Sec. \ref{sec:subcats}) & Order-preserving hom.\tabularnewline
\hline 
$\mathsf{OrdVect}_{\mathbb{Q}}$ & Ordered vector spaces over $\mathbb{Q}$ (Sec. \ref{sec:subcats}) & Order-preserving linear maps\tabularnewline
\hline 
$\mathsf{IcOrdCMon}$ & $\!\!\!\!$Integrally closed ord. comm. monoid (Sec. \ref{sec:subcats})$\!\!\!\!$ & Order-preserving hom.\tabularnewline
\hline 
$\mathsf{IcOrdAb}$ & $\!\!\!\!$Integrally closed ord. abelian group (Sec. \ref{sec:subcats})$\!\!\!\!$ & Order-preserving hom.\tabularnewline
\hline 
$\mathsf{IcOrdVect}_{\mathbb{Q}}$ & Integrally closed ord. vect. spaces (Sec. \ref{sec:subcats}) & Order-preserving linear maps\tabularnewline
\hline 
$\mathsf{MonCat}_{\ell}$ & Small monoidal categories & Lax monoidal functors\tabularnewline
\hline 
$\mathsf{MonCat}_{s}$ & Small monoidal categories & Strong monoidal functors\tabularnewline
\hline 
$\mathsf{FinProb}$ & Finite probability spaces & Measure-preserving maps\tabularnewline
\hline 
$\mathsf{LProb}_{\rho}$ & $\!\!\!\!\!$$\rho$-th-power-summable prob. mass functions (Sec.
\ref{sec:lrho})$\!\!\!\!\!$ & Measure-preserving maps\tabularnewline
\hline 
$\mathsf{HProb}$ & Finite entropy prob. mass functions (Sec. \ref{sec:lrho}) & Measure-preserving maps\tabularnewline
\hline 
$\mathsf{MonSiMonCat}$ & $\!\!\!\!$Monoidal strictly-indexed monoidal cat. (Sec. \ref{sec:gcmcat})$\!\!\!\!$ & MonSiMon functors\tabularnewline
\hline 
$\!\!\!\mathsf{DisiMonSiMonCat}\!\!\!\!$ & DisiMonSiMon categories (Sec. \ref{sec:conditional}) & DisiMonSiMon functors\tabularnewline
\hline 
\end{tabular}}}
\par\end{centering}
\medskip{}

\medskip{}

\caption{Table of categories used in this paper.}
\end{table}

\begin{table}[H]
\begin{centering}
{\renewcommand*{\arraystretch}{1.38}%
\begin{tabular}{cc}
\textbf{Symbol} & \textbf{Description}\tabularnewline
\hline 
$\mathbb{N}$ & Set of positive integers $\{1,2,\ldots\}$\tabularnewline
\hline 
$\mathrm{hom}_{\mathsf{C}}(A,B)$ & Hom-set containing morphisms in the form $A\to B$ in $\mathsf{C}$\tabularnewline
\hline 
$!$ & Unique morphism to a terminal object\tabularnewline
\hline 
$\mathrm{Epi}(\mathsf{C})$ & The wide subcategory of $\mathsf{C}$ containing the epimorphisms
of $\mathsf{C}$\tabularnewline
\hline 
$\mathrm{Dis}$ & Comonad $\mathsf{Cat}\to\mathsf{Cat}$ sending a category to its discrete
subcategory\tabularnewline
\hline 
$\Delta_{X}$ & The constant functor that maps any object to $X$\tabularnewline
\hline 
$H_{0}(P)$ & Hartley entropy $H_{0}(P)=\log|\{x:\,P(x)>0\}|$\tabularnewline
\hline 
$H_{1}(P)$ & Shannon entropy $H_{1}(P)=-\sum_{x}P(x)\log P(x)$\tabularnewline
\hline 
\end{tabular}}
\par\end{centering}
\medskip{}

\medskip{}

\caption{Table of notations.}
\end{table}

\section{Preliminaries\label{sec:pre}}

We first review the basic properties of $\mathsf{FinProb}$ in \cite{franz2002stochastic,baez2011entropy}.
The category $\mathsf{FinProb}$ has finite probability spaces, or
equivalently, strictly positive probability mass functions (i.e.,
$P:\mathrm{supp}(P)\to[0,1]$ with $P(x)>0$ for $x\in\mathrm{supp}(P)$,
and $\sum_{x}P(x)=1$)\footnote{We note that \cite{baez2011characterization,baez2011entropy} does
not require $P$ to be strictly positive. This is a minor difference
since we can always ignore elements in a probability space with zero
probability, and would not affect the universal entropy of $\mathsf{FinProb}$
derived in Theorem \ref{thm:finprob_ic}. A nice consequence of requiring
$P$ to be strictly positive is that all morphisms in $\mathsf{FinProb}$
are epimorphisms. Also, an issue with allowing non-strictly-positive
probability measures is that two almost surely equal measure-preserving
functions might not be considered identical.} as objects, and a morphism from the probability space $P$ to $Q$
is a measure-preserving function (i.e., $f:\mathrm{supp}(P)\to\mathrm{supp}(Q)$
with $Q(y)=\sum_{x\in f^{-1}(y)}P(x)$) \cite{franz2002stochastic,baez2011entropy,baez2011characterization}.
It is a monoidal category with tensor product given by the product
distribution, i.e., $P\otimes Q$ is given by $R:\mathrm{supp}(P)\times\mathrm{supp}(Q)\to[0,1]$,
$R(x,y)=P(x)Q(y)$, and the tensor unit is the degenerate probability
space where $\mathrm{supp}(P)$ is a singleton.

For $P\in\mathsf{FinProb}$, the under category $P/\mathsf{FinProb}$,
containing objects that are morphisms in $\mathsf{FinProb}$ in the
form $f:P\to Q$, can be regarded as the category of random variables
defined over the probability space $P$. It was noted in \cite{baez2011entropy}
that $P/\mathsf{FinProb}$ has a categorical product, where the product
$f_{1}\times_{P}f_{2}$ between $f_{1}:P\to Q_{1}$ and $f_{2}:P\to Q_{2}$
is given by $f:P\to Q$ where $\mathrm{supp}(Q)=\mathrm{im}((f_{1},f_{2}))$
(the image of the pairing $(f_{1},f_{2}):P\to Q_{1}\times Q_{2}$),
$Q(y_{1},y_{2})=\sum_{x\in(f_{1},f_{2})^{-1}(y_{1},y_{2})}P(x)$,
and $f(x)=(f_{1}(x),f_{2}(x))$, which represents the joint random
variable of $f_{1}$ and $f_{2}$.

Consider the monoidal posetal category $[0,\infty)$, where there
is a morphism $X\to Y$ for $X,Y\in[0,\infty)$ if $X\ge Y$, with
tensor product given by addition. Consider the Shannon entropy $H_{1}:\mathsf{FinProb}\to[0,\infty)$
given by \cite{shannon1948mathematical}
\[
H_{1}(P)=-\sum_{x\in\mathrm{supp}(P)}P(x)\log P(x).
\]
It is a functor since a measure-preserving function maps a finer probability
space to a coarser probability space, which must have a smaller (or
equal) entropy. It is strongly monoidal ($H_{1}(P\otimes Q)=H_{1}(P)+H_{2}(Q)$)
due to the additive property of entropy \eqref{eq:additive} \cite{baez2011entropy}.
Consider the under category $P/\mathsf{FinProb}$ equipped with the
codomain functor $\mathrm{Cod}_{P}:P/\mathsf{FinProb}\to\mathsf{FinProb}$
mapping $f:P\to Q$ to $Q$. The composition $H_{1}\mathrm{Cod}_{P}:P/\mathsf{FinProb}\to[0,\infty)$
is a lax monoidal functor (with respect to the categorical product
over $P/\mathsf{FinProb}$) \cite{baez2011entropy}, since we have
\[
H_{1}\big(\mathrm{Cod}_{P}(f)\big)+H_{1}\big(\mathrm{Cod}_{P}(g)\big)\ge H_{1}\big(\mathrm{Cod}_{P}(f\times_{P}g)\big)
\]
due to the subadditive property of entropy \eqref{eq:subadditive}.

In this paper, unlike many previous approaches, we define entropy
over random variables instead of probability distributions. We will
see later that this leads to a simple characterization of entropy.

\medskip{}

\section{Monoidal Strictly-Indexed Monoidal Categories\label{sec:gcmcat}}

We have seen in the previous section that $\mathsf{FinProb}$ is not
only a monoidal category, but its under categories $P/\mathsf{FinProb}$
are also monoidal. We can have an ``under category functor'' $-/\mathsf{FinProb}:\,\mathsf{FinProb}^{\mathrm{op}}\to\mathsf{MonCat}_{\ell}$,
which maps the morphism $g:P\to Q$ to the precomposition functor
$Q/\mathsf{FinProb}\to P/\mathsf{FinProb}$, $(Q\stackrel{f}{\to}R)\mapsto(P\stackrel{fg}{\to}R)$,
which is a (strong) monoidal functor corresponding to an ``extension
of probability space'', i.e., it maps a random variable $f:Q\to R$
over the probability space $Q$ to the ``same'' random variable
$fg:P\to R$ over a finer probability space $P$. Therefore, $-/\mathsf{FinProb}$
can be treated as an indexed monoidal category (recall that an $\mathsf{S}$-indexed
monoidal category \cite{corradini1993categorical,hofstra2006descent,shulman2013enriched}
is a pseudofunctor $\mathsf{S}^{\mathrm{op}}\to\mathsf{MonCat}_{s}$,
though we use $\mathsf{MonCat}_{\ell}$ instead in this paper). Moreover,
$-/\mathsf{FinProb}$ is actually a strict functor, so $-/\mathsf{FinProb}$
is a ``strictly-indexed monoidal category''.

Treat $\mathsf{MonCat}_{\ell}$ as a monoidal category with a tensor
product $\times$ given by product category. We can check that the
under category functor $-/\mathsf{FinProb}:\,\mathsf{FinProb}^{\mathrm{op}}\to\mathsf{MonCat}_{\ell}$
is a lax monoidal functor, where the coherence map 
\begin{equation}
\zeta_{P,Q}:(P/\mathsf{FinProb})\times(Q/\mathsf{FinProb})\to(P\otimes Q)/\mathsf{FinProb}\label{eq:under_coherence}
\end{equation}
sends a pair of random variables $X:P\to A$, $Y:Q\to B$ to the joint
random variable $X\otimes Y:P\otimes Q\to A\otimes B$ over the product
probability space $P\otimes Q$. Note that $\zeta_{P,Q}$ itself is
a lax monoidal functor as well.\footnote{The coherence map of $\zeta_{P,Q}$ is $\varpi_{(X_{1},Y_{1}),(X_{2},Y_{2})}:\zeta_{P,Q}(X_{1},Y_{1})\otimes_{P\otimes Q}\zeta_{P,Q}(X_{2},Y_{2})\to\zeta_{P,Q}(X_{1}\otimes_{P}X_{2},Y_{1}\otimes_{Q}Y_{2})$
for $X_{1},X_{2}\in P/\mathsf{FinProb}$, $Y_{1},Y_{2}\in Q/\mathsf{FinProb}$,
where $\otimes_{P}=\times_{P}$ is the tensor product in $P/\mathsf{FinProb}$.
Since $\zeta_{P,Q}(X_{1},Y_{1})\otimes_{P\otimes Q}\zeta_{P,Q}(X_{2},Y_{2})$
is the joint random variable ``$((X_{1},Y_{1}),(X_{2},Y_{2}))$''
over the probability space $P\otimes Q$, and $\zeta_{P,Q}(X_{1}\otimes_{P}X_{2},Y_{1}\otimes_{Q}Y_{2})$
is the joint random variable ``$((X_{1},X_{2}),(Y_{1},Y_{2}))$'',
$\varpi$ simply reorders $((X_{1},Y_{1}),(X_{2},Y_{2}))$ to $((X_{1},X_{2}),(Y_{1},Y_{2}))$.} Hence, $-/\mathsf{FinProb}$ is a ``monoidal strictly-indexed monoidal
category'' (recall that a monoidal indexed category \cite{moeller2018monoidal}
is a lax monoidal pseudofunctor in the form $\mathsf{S}^{\mathrm{op}}\to\mathsf{Cat}$
where $\mathsf{S}$ is a monoidal category).

In this section, we define the concept of \emph{monoidal strictly-indexed
monoidal category}. It is basically a straightforward combination
of the notions of indexed monoidal categories \cite{corradini1993categorical,hofstra2006descent,shulman2013enriched}
and monoidal indexed categories \cite{moeller2018monoidal}, where
we focus on (strict) functors instead of pseudofunctors. It can be
defined succinctly below as an object in a lax-slice 2-category \cite{johnson2021two}.

\medskip{}

\begin{defn}
[Monoidal strictly-indexed monoidal categories] \label{def:gcmcat_succinct}Assume
$\mathsf{MonCat}_{\ell}$ contains only small categories with respect
to a Grothendieck universe. Let $\mathbb{U}$ be a Grothendieck universe
large enough such that $\mathbb{U}\mathsf{MonCat}_{\ell}$, the 2-category
of $\mathbb{U}$-small monoidal categories with lax monoidal functors
as 1-cells and monoidal natural transformations as 2-cells, contains
$\mathsf{MonCat}_{\ell}$ as an object. Then the \emph{category of
monoidal strictly-indexed monoidal (MonSiMon) categories} $\mathsf{MonSiMonCat}$
is given as the lax-slice 2-category 
\[
\mathsf{MonSiMonCat}=\mathbb{U}\mathsf{MonCat}_{\ell}/\!/\mathsf{MonCat}_{\ell}.
\]
\end{defn}

\medskip{}

\medskip{}

In short, \emph{$\mathsf{MonSiMonCat}$ is the lax-slice 2-category
over $\mathsf{MonCat}_{\ell}$ in the 2-category of monoidal categories}.\footnote{Written out, \emph{$\mathsf{MonSiMonCat}$ is the lax-slice 2-category
over the category of monoidal categories with lax monoidal functors
as morphisms, in the 2-category of monoidal categories with lax monoidal
functors as 1-cells and monoidal natural transformations as 2-cells},
which could be a tongue twister.} We now unpack the above definition. An object of $\mathsf{MonSiMonCat}$,
which we call a \emph{monoidal strictly-indexed monoidal (MonSiMon)
category}, is a pair $(\mathsf{C},N)$, where $\mathsf{C}\in\mathbb{U}\mathsf{MonCat}_{\ell}$,
and $N:\mathsf{C}^{\mathrm{op}}\to\mathsf{MonCat}_{\ell}$ is a lax
monoidal functor. A morphism in $\mathsf{MonSiMonCat}$ (which we
call a \emph{MonSiMon functor}) from $(\mathsf{C},N)$ to $(\mathsf{D},M)$
is a pair $(F,\gamma)$, where $F:\mathsf{C}\to\mathsf{D}$ is a lax
monoidal functor, and $\gamma:N\Rightarrow MF^{\mathrm{op}}$ is a
monoidal natural transformation. Composition of MonSiMon functors
$(\mathsf{C}_{1},N_{1})\stackrel{(F_{1},\gamma_{1})}{\to}(\mathsf{C}_{2},N_{2})\stackrel{(F_{2},\gamma_{2})}{\to}(\mathsf{C}_{3},N_{3})$
is given by $(F_{2}F_{1},\,\gamma_{2}F_{1}^{\mathrm{op}}\circ\gamma_{1})$.
\[\begin{tikzcd}
	{\mathsf{C}_1^{\mathrm{op}}} && {\mathsf{C}_2^{\mathrm{op}}} && {\mathsf{C}_3^{\mathrm{op}}} \\
	\\
	&& {\mathsf{MonCat}_\ell}
	\arrow["{F_1^{\mathrm{op}}}", from=1-1, to=1-3]
	\arrow[""{name=0, anchor=center, inner sep=0}, "{N_1}"', from=1-1, to=3-3]
	\arrow[""{name=1, anchor=center, inner sep=0}, "{N_2}"{description}, from=1-3, to=3-3]
	\arrow["{F_2^{\mathrm{op}}}", from=1-3, to=1-5]
	\arrow["{N_3}", from=1-5, to=3-3]
	\arrow["{\gamma_1}", shorten <=6pt, shorten >=6pt, Rightarrow, from=0, to=1-3]
	\arrow["{\gamma_2}", shorten <=12pt, shorten >=12pt, Rightarrow, from=1, to=1-5]
\end{tikzcd}\]

Refer to \cite{baez2004some} for the definition of monoidal natural
transformations. For the convenience of the readers, we include the
coherence conditions for $\gamma:N\Rightarrow MF^{\mathrm{op}}$ being
a monoidal natural transformation here. In addition of being a natural
transformation, the following diagrams in $\mathsf{MonCat}_{\ell}$
must also commute for every $P,Q\in\mathsf{C}$:
\[\begin{tikzcd}
	{N(P)\times N(Q)} && {MF(P)\times MF(Q)} && {*} \\
	{N(P\otimes_{\mathsf{C}} Q)} && {MF(P\otimes_{\mathsf{C}} Q)} && {N(I_{\mathsf{C}} )} & {MF(I_{\mathsf{C}} )}
	\arrow["{\zeta_{P,Q}}"', from=1-1, to=2-1]
	\arrow["{\gamma_P \times \gamma_Q}", from=1-1, to=1-3]
	\arrow["{\eta_{F(P), F(Q)}}", from=1-3, to=2-3]
	\arrow["{\gamma_{P\otimes_{\mathsf{C}} Q}}"', from=2-1, to=2-3]
	\arrow["\epsilon"', from=1-5, to=2-5]
	\arrow["\delta", from=1-5, to=2-6]
	\arrow["{\gamma_{I_{\mathsf{C}}}}"', from=2-5, to=2-6]
\end{tikzcd}\]where $N(P)\times N(Q)$ denotes the product monoidal category, $I_{\mathsf{C}}$
is the tensor unit of the tensor product $\otimes_{\mathsf{C}}$ of
$\mathsf{C}$, and $\zeta_{P,Q},\epsilon$ are the coherence maps
of the lax monoidal functor $N$, and $\eta_{F(P),F(Q)},\delta$ are
those of $MF^{\mathrm{op}}$.

Note that $\mathsf{MonSiMonCat}$ is a 2-category, where a 2-cell
$(F,\gamma)\Rightarrow(G,\delta):\,(\mathsf{C},N)\to(\mathsf{D},M)$
is a monoidal natural transformation $\kappa:F\Rightarrow G$ such
that $(M\kappa)\circ\gamma=\delta$, which comes from the definition
of lax-slice 2-category \cite{johnson2021two}. In this paper, we
will usually treat $\mathsf{MonSiMonCat}$ as a 1-category and ignore
these 2-cells.

\medskip{}

One class of examples of MonSiMon categories is $(\mathsf{C},N)$
where $\mathsf{C}$ is a monoidal category, and $N=-/\mathsf{C}:\,\mathsf{C}^{\mathrm{op}}\to\mathsf{MonCat}_{\ell}$
is the under category functor, which is a lax monoidal functor with
coherence map $\zeta:N(-)\times N(-)\Rightarrow N(-\otimes-)$ given
by the tensor product in $\mathsf{C}$, i.e., $\zeta_{P,Q}(X,Y)=X\otimes Y$
for $X:P\to A$, $Y:Q\to B$,\footnote{For morphisms, $\zeta_{P,Q}(f,g)=f\otimes g$ for $X:P\to A$, $X':P\to A'$,
$f:A\to A'$ with $X'=fX$ (i.e., $f$ is a morphism in $N(P)$),
and $Y:Q\to B$, $Y':Q\to B'$, $g:B\to B'$ with $Y'=gY$.} and $\epsilon:*\to N(I_{\mathsf{C}})$ mapping the unique object
to the tensor unit of $N(I_{\mathsf{C}})$. For $(\mathsf{C},N)$
to be a MonSiMon category, we also need $\zeta_{P,Q}$ to be a lax
monoidal functor satisfying the coherence conditions, though they
are usually straightforward to check for specific examples of $\mathsf{C}$.
We will write such MonSiMon category as $(\mathsf{C},\,-/\mathsf{C})$,
or even simply $-/\mathsf{C}$ if it is clear from the context.

For example, consider the MonSiMon category $(\mathsf{FinProb},\,-/\mathsf{FinProb})$.
The functor $N=-/\mathsf{FinProb}$ sends a finite probability space
$P$ to the monoidal category $N(P)$ of random variables defined
over $P$, where objects are random variables defined over $P$ (measure-preserving
functions from $P$ to another finite probability space), morphisms
are measurable functions between random variables, and the tensor
product is the joint random variable. The functor $N$ can be regarded
as a $\mathsf{MonCat}_{\ell}$-valued presheaf on $\mathsf{FinProb}$.
The functor $N$ is lax monoidal, with a coherence map $\zeta_{P,Q}$
sending a pair of random variables $(X,Y)$ (over the space $P$ and
the space $Q$, respectively) to the joint random variable $(X,Y)$
over the product space $P\otimes Q$. See \eqref{eq:under_coherence}.

We remark that the usual notations for indexed category \cite{pare2006indexed}
$\mathbb{C}:\mathsf{C}^{\mathrm{op}}\to\mathsf{Cat}$ is to write
$\mathbb{C}^{X}=\mathbb{C}(X)$ for objects $X\in\mathsf{C}$, and
$f^{*}=\mathbb{C}(f)$ for morphisms $f$ in $\mathsf{C}$. The archetypal
example of indexed category is the self-indexing, where $\mathsf{C}$
is a category with pullbacks, $\mathbb{C}^{X}=\mathsf{C}/X$ is the
over category, and $f^{*}$ is the base change functor given by pullback.
We avoid these notations since $N(X)$ in this paper is often the
under category (not the over category), and $N(f)$ is often the precomposition
functor (not the base change), so writing $f^{*}$ (which usually
denotes the base change) would be confusing. We use the standard functor
notations and write $N(X)$ and $N(f)$.

\medskip{}

\section{Reflective Subcategories of $\mathsf{MonSiMonCat}$\label{sec:subcats}}

A MonSiMon category $(\mathsf{C},N)$ can be thought as a collection
of ``categories of random variables over the probability space $P$''
given by $N(P)$, linked together by the contravariant functor $N:\mathsf{C}^{\mathrm{op}}\to\mathsf{MonCat}_{\ell}$.
Our goal is to define an ``entropy'' over the MonSiMon category
$(\mathsf{C},N)$ as a MonSiMon functor from $(\mathsf{C},N)$ to
a suitable codomain, which is another MonSiMon category satisfying
some desirable properties. An entropy sends any random variable over
any probability space to a fixed space (e.g. $\mathbb{R}$). The space
where the entropy of a random variable $X\in N(P)$ resides should
not depend on the probability space $P$. Hence, the codomain of the
entropy, as a MonSiMon category, should be in the form $(*,\,F:*\to\mathsf{MonCat}_{\ell})$
(where $*$ is the terminal monoidal category with one object $\bullet$),
which is a collection that contains only one ``space'' $F(\bullet)$.
Every MonSiMon functor $(!,\gamma)$ from $(\mathsf{C},N)$ to $(*,\,F:*\to\mathsf{MonCat}_{\ell})$
would map every random variable $X\in N(P)$ to an object $\gamma_{P}(X)\in F(\bullet)$
in the same space $F(\bullet)$. If our goal is only to recover the
Shannon entropy, then we can simply take $F(\bullet)=\mathbb{R}$.
Nevertheless, not every MonSiMon category has a universal entropy
to $\mathbb{R}$ (we will see later that $-/\mathsf{LProb}_{\rho}$
for $0<\rho<1$ has a universal entropy to $\mathbb{R}$, but $-/\mathsf{FinProb}$
only has a universal entropy to $(\log\mathbb{Q}_{>0})\times\mathbb{R}$).
To allow the universal entropy to be defined for every MonSiMon category,
we want to find a ``category of codomains'' of entropy that is more
general than $\mathbb{R}$.

In this paper, we will take the ``category of codomains'' to be
$\mathsf{OrdCMon}$ or its reflective subcategories, where $\mathsf{OrdCMon}$
is the category of ordered commutative monoids with order-preserving
monoid homomorphisms as morphisms \cite{blyth2005lattices,fritz2017resource}.
 Each ordered commutative monoid $\mathsf{M}$ in $\mathsf{OrdCMon}$
can be treated as a MonSiMon category $(*,N)$, where $N(\bullet)\in\mathsf{MonCat}_{\ell}$
is simply taken to be $\mathsf{M}$ treated as a symmetric monoidal
skeletal posetal category. This defines an embedding functor $R:\mathsf{OrdCMon}\to\mathsf{MonSiMonCat}$.
The next section will provide justifications for restricting attention
to $\mathsf{OrdCMon}$.

\medskip{}

\subsection{$\mathsf{OrdCMon}$ as a Reasonable ``Category of Codomains''}

In this subsection, we explain why $\mathsf{OrdCMon}$ is a reasonable
choice of the codomain of an entropy. This subsection is not essential,
and readers may opt to skip this subsection.

Write $[*,\mathsf{MonCat}_{\ell}]$ for the category of lax monoidal
functors $*\to\mathsf{MonCat}_{\ell}$ (where $*$ is the terminal
monoidal category with one object $\bullet$), with morphisms being
monoidal natural transformations. It is known that this is simply
the category of monoids in $\mathsf{MonCat}_{\ell}$ \cite{porst2008categories,nlab:microcosm_principle}:
\[
[*,\mathsf{MonCat}_{\ell}]\cong\mathsf{Mon}(\mathsf{MonCat}_{\ell}).
\]
Recall that $\mathsf{Mon}(\mathsf{MonCat}_{\ell})$ is the category
of monoid objects in $\mathsf{MonCat}_{\ell}$, with morphisms being
morphisms in $\mathsf{MonCat}_{\ell}$ that respects the monoid structure.
A lax monoidal functor $F:*\to\mathsf{MonCat}_{\ell}$ have coherence
maps $\mu_{\bullet,\bullet}:F(\bullet)\times F(\bullet)\to F(\bullet)$
and $\nu:*\to F(\bullet)$, which can be treated as the multiplication
and unit morphisms of $F(\bullet)$ respectively, making $(F(\bullet),\mu_{\bullet,\bullet},\nu)$
a monoid in $\mathsf{MonCat}_{\ell}$. Also, each object in $\mathsf{Mon}(\mathsf{MonCat}_{\ell})$
is a duoidal category \cite{aguiar2010monoidal,batanin2012centers}
(a category with two tensor products satisfying some coherence conditions),
where one product $\mu_{\bullet,\bullet}$ of $F(\bullet)$ is strict
monoidal, and the other product is the tensor product of $F(\bullet)\in\mathsf{MonCat}_{\ell}$.
We can then embed $\mathsf{Mon}(\mathsf{MonCat}_{\ell})$ into $\mathsf{MonSiMonCat}$
via the functor $E:\mathsf{Mon}(\mathsf{MonCat}_{\ell})\to\mathsf{MonSiMonCat}$
defined as the composition
\[
\mathsf{Mon}(\mathsf{MonCat}_{\ell})\cong[*,\mathsf{MonCat}_{\ell}]\hookrightarrow\mathsf{MonSiMonCat},
\]
where ``$\hookrightarrow$'' is the obvious embedding. It follows
from $[*,\mathsf{MonCat}_{\ell}]\cong\mathsf{Mon}(\mathsf{MonCat}_{\ell})$
that $E$ is fully faithful. We can regard $\mathsf{Mon}(\mathsf{MonCat}_{\ell})$
as a full subcategory of $\mathsf{MonSiMonCat}$.

 Since the purpose of entropy is to compare the amount of information
in different objects, we will assume that an entropy takes values
over a monoidal poset, or a monoidal skeletal posetal category.\footnote{A skeletal posetal category is a category where the homset between
two objects contains at most one morphism, and it is skeletal (isomorphisms
are identity morphisms). Hence, every strong monoidal functor between
monoidal skeletal posetal categories must be strict.} Equivalently, this is $\mathsf{OrdMon}$, the category of (partially)
ordered monoids with order-preserving monoid homomorphisms as morphisms
\cite{blyth2005lattices}. Recall that an ordered monoid is a monoid
equipped with a partial order $\ge$ which is order-preserving, i.e.,
$x\ge y$ and $z\ge w$ implies $x+z\ge y+w$. Note that each element
of $\mathsf{OrdMon}$ is regarded as a monoidal skeletal posetal category,
not a one-object category (which is sometimes how a monoid is treated
as a category). Hence, we will consider $[*,\mathsf{OrdMon}]\cong\mathsf{Mon}(\mathsf{OrdMon})$
instead of $[*,\mathsf{MonCat}_{\ell}]\cong\mathsf{Mon}(\mathsf{MonCat}_{\ell})$.

By the standard Eckmann-Hilton argument \cite{eckmann1962group},
we have 
\[
[*,\mathsf{OrdMon}]\cong\mathsf{Mon}(\mathsf{OrdMon})\cong\mathsf{OrdCMon},
\]
where $\mathsf{OrdCMon}$ is the category of (partially) ordered commutative
monoids with order-preserving monoid homomorphisms as morphisms \cite{blyth2005lattices,fritz2017resource}.
Again we emphasize that an object of $\mathsf{OrdCMon}$ should be
regarded as a symmetric monoidal skeletal posetal category, not a
one-object category. For the direction $\mathsf{OrdCMon}\to\mathsf{Mon}(\mathsf{OrdMon})\cong[*,\mathsf{OrdMon}]$,
we can map a symmetric monoidal posetal category $\mathsf{C}$ to
the monoid $(\mathsf{C},\otimes_{\mathsf{C}},I_{\mathsf{C}})$ in
$\mathsf{OrdMon}$, or the lax monoidal functor $F:*\to\mathsf{OrdMon}$
in $[*,\mathsf{OrdMon}]$ with $F(\bullet)=\mathsf{C}$ and coherence
maps $\otimes_{\mathsf{C}}:F(\bullet)\times F(\bullet)\to F(\bullet)$
and $I_{\mathsf{C}}:*\to F(\bullet)$ in $\mathsf{OrdMon}$. Another
benefit of considering $\mathsf{Mon}(\mathsf{OrdMon})$ is that now
we do not have to consider two different products (as in duoidal categories
in $\mathsf{Mon}(\mathsf{MonCat}_{\ell})$), since the two products
must coincide and must be commutative \cite{eckmann1962group}.\footnote{We briefly repeat the argument in \cite{eckmann1962group} here. For
a monoid $(\mathsf{C},\odot,J)$ in $\mathsf{OrdMon}$, whether the
product and identity of $\mathsf{C}$ are $\otimes$ and $I$ respectively,
we must have $J=I$ since the unit $*\stackrel{J}{\to}\mathsf{C}$
of the monoid must map to $I$. Since $\odot$ is a morphism in $\mathsf{OrdMon}$,
we have $(A\odot B)\otimes(C\odot D)=(A\otimes C)\odot(B\otimes D)$.
Hence $A\otimes B=(A\odot I)\otimes(I\odot B)=(A\otimes I)\odot(I\otimes B)=A\odot B$.
Also, $A\otimes B=(I\odot A)\otimes(B\odot I)=(I\otimes B)\odot(A\otimes I)=B\odot A$,
so $\otimes$ and $\odot$ coincide and are both commutative.}

Hence, we can embed $\mathsf{OrdCMon}$ into $\mathsf{MonSiMonCat}$
via $R:\mathsf{OrdCMon}\to\mathsf{MonSiMonCat}$ defined as the composition
\[
\mathsf{OrdCMon}\cong\mathsf{Mon}(\mathsf{OrdMon})\hookrightarrow\mathsf{Mon}(\mathsf{MonCat}_{\ell})\stackrel{E}{\to}\mathsf{MonSiMonCat}.
\]

To summarize, we can deduce that the ``category of codomains'' must
be $\mathsf{OrdCMon}$ based on the sole assumption that an entropy
takes values over a poset. The poset is automatically an ordered monoid
due to the definition of MonSiMon categories, and it is automatically
commutative by the Eckmann-Hilton argument \cite{eckmann1962group}.

\medskip{}

\subsection{$\mathsf{OrdCMon}$ as a Reflective Subcategory}

We can write the left adjoint of $R$ explicitly, exhibiting $\mathsf{OrdCMon}$
as a reflective subcategory of $\mathsf{MonSiMonCat}$.\medskip{}

\begin{lem}
\label{lem:lgcm}$R:\mathsf{OrdCMon}\to\mathsf{MonSiMonCat}$ is fully
faithful, and has a left adjoint $L:\mathsf{MonSiMonCat}\to\mathsf{OrdCMon}$.
Hence, $\mathsf{OrdCMon}$ is a reflective subcategory of $\mathsf{MonSiMonCat}$.
\end{lem}

\begin{proof}
Since $E$ is fully faithful, and $\mathsf{OrdCMon}\hookrightarrow\mathsf{Mon}(\mathsf{MonCat}_{\ell})$
is fully faithful,\footnote{Consider $\mathsf{C},\mathsf{D}\in\mathsf{OrdCMon}$ and morphism
$F:\mathsf{C}\to\mathsf{D}$ in $\mathsf{Mon}(\mathsf{MonCat}_{\ell})$.
$F$ is order-preserving since it is also a morphism in $\mathsf{MonCat}_{\ell}$,
and $F(X\otimes Y)=F(X)\otimes F(Y)$ since $F$ is a morphism in
a category of monoids. Hence $F$ corresponds to a unique morphism
in $\mathsf{OrdCMon}$.} $R$ is fully faithful as well. We now construct $L$. Consider a
MonSiMon category $(\mathsf{C},N)$, and let the coherence maps of
the lax monoidal functor $N:\mathsf{C}^{\mathrm{op}}\to\mathsf{MonCat}_{\ell}$
be $\zeta:N(-)\times N(-)\Rightarrow N(-\otimes-)$ and $\epsilon:*\to N(I_{\mathsf{C}})$.
Consider the disjoint union $\coprod_{P\in\mathsf{C}}\mathrm{Ob}(N(P))$
where $\mathrm{Ob}(N(P))$ is the set of objects of $N(P)$. Let $\mathsf{N}$
be the commutative monoid generated by the set $\coprod_{P\in\mathsf{C}}\mathrm{Ob}(N(P))$,
with product denoted as $+$ and identity denoted as $I_{\mathsf{N}}$.
Consider a binary relation $\mathcal{R}$ over $\mathsf{N}$. We want
$\mathcal{R}$ to turn $\mathsf{N}$ into a monoidal posetal category,
which we will take as $L((\mathsf{C},N))$. In order to have a MonSiMon
functor $(*,\gamma)$ from $(\mathsf{C},N)$ to $R(\mathsf{N})$ (where
$\gamma:N\Rightarrow\Delta_{\mathsf{N}}$, and $\gamma_{P}$ maps
$X\in N(P)$ to itself in $\mathsf{N}$), we need $\mathcal{R}$ to
include the following pairs:
\begin{itemize}
\item $(X,Y)$ for $X,Y\in N(P)$, $P\in\mathsf{C}$ whenever there is a
morphism $X\to Y$ in $N(P)$ (needed for the functoriality of $\gamma_{P}:N(P)\to\mathsf{M}$);
\item $(X+Y,\,X\otimes_{N(P)}Y)$ for $X,Y\in N(P)$, $P\in\mathsf{C}$
(needed for $\gamma_{P}$ being a lax monoidal functor);
\item $(I_{\mathsf{N}},\,I_{N(P)})$ for $P\in\mathsf{C}$ (needed for $\gamma_{P}$
being a lax monoidal functor);
\item $(X,\,N(f)(X))$ and $(N(f)(X),\,X)$ for $f:P\to Q$ in $\mathsf{C}$,
$X\in N(Q)$ (needed for the naturality of $\gamma$);
\item $(X+Y,\,\zeta_{P,Q}(X,Y))$ and $(\zeta_{P,Q}(X,Y),\,X+Y)$ for $P,Q\in\mathsf{C}$,
$X\in N(P)$, $Y\in N(Q)$ (needed for $\gamma$ being a monoidal
natural transformation);
\item $(I_{\mathsf{N}},\,\epsilon(\bullet))$ and $(\epsilon(\bullet),\,I_{\mathsf{N}})$
(needed for $\gamma$ being a monoidal natural transformation).
\end{itemize}
Consider the ``monoidal closure'' $\mathrm{Mon}(\mathcal{R})$ of
$\mathcal{R}$, which contains $(X_{1}+\cdots+X_{n},\,Y_{1}+\cdots+Y_{n})$
for $X_{i},Y_{i}\in\mathsf{C}$ if $(X_{i},Y_{i})\in\mathcal{R}$
for $i=1,\ldots,n$. Consider the transitive closure $\mathrm{Tra}(\mathrm{Mon}(\mathcal{R}))$,
which contains $(X,Y)$ if $(X,Z_{1}),(Z_{1},Z_{2}),\ldots,(Z_{n-1},Z_{n}),(Z_{n},Y)\in\mathrm{Mon}(\mathcal{R})$
for some $Z_{1},\ldots,Z_{n}\in\mathsf{C}$. Note that $\mathrm{Mon}(\mathrm{Tra}(\mathrm{Mon}(\mathcal{R})))=\mathrm{Tra}(\mathrm{Mon}(\mathcal{R}))$,
and $\mathrm{Tra}(\mathrm{Mon}(\mathcal{R}))$ is the smallest monoidal
preorder containing $\mathcal{R}$. We can then turn $\mathsf{N}$
into a preordered commutative monoid using $\mathrm{Tra}(\mathrm{Mon}(\mathcal{R}))$.
Let $\bar{\mathsf{N}}$ be the symmetric monoidal posetal category
(or equivalently, ordered commutative monoid) where objects are equivalent
classes in the preorder of $\mathsf{N}$, tensor product is given
by $+$, and there is a morphism $X\to Y$ if $X\ge Y$ in the preorder.
We then take $L((\mathsf{C},N))=\bar{\mathsf{N}}$. Consider the MonSiMon
functor $(*,\bar{\gamma})$ from $(\mathsf{C},N)$ to $R(\bar{\mathsf{N}})$
where $\bar{\gamma}_{P}$ maps $X\in N(P)$ to its equivalence class
in $\bar{\mathsf{N}}$. It is clear from the construction that $\bar{\mathsf{N}}$
is the universal symmetric monoidal posetal category such that there
is a MonSiMon functor from $(\mathsf{C},N)$ to $R(\bar{\mathsf{N}})$,
and $(*,\bar{\gamma})$ is the universal functor from $(\mathsf{C},N)$
to $R$. The result follows.

\end{proof}

\medskip{}

\subsection{Other ``Categories of Codomains'' of Entropy\label{subsec:cat_of_codomains}}

Since $\mathsf{OrdCMon}$ is a reflective subcategory of $\mathsf{MonSiMonCat}$,
every reflective subcategory of $\mathsf{OrdCMon}$ is also a reflective
subcategory of $\mathsf{MonSiMonCat}$, and would be a suitable choice
for the ``category of codomains'' of an entropy. We have the following
diagram of reflective subcategories:
\[\begin{tikzcd}
	{\mathsf{IcOrdVect}_{\mathbb{Q}}} & {\mathsf{IcOrdAb}} & {\mathsf{IcOrdCMon}} \\
	{\mathsf{OrdVect}_{\mathbb{Q}}} & {\mathsf{OrdAb}} & {\mathsf{COrdCMon}} & {\mathsf{OrdCMon}} \\
	&&& {\mathsf{MonSiMonCat}}
	\arrow[hook, from=2-4, to=3-4]
	\arrow[hook, from=1-3, to=2-3]
	\arrow[hook, from=2-3, to=2-4]
	\arrow[hook, from=1-1, to=1-2]
	\arrow[hook, from=1-2, to=1-3]
	\arrow[hook, from=2-2, to=2-3]
	\arrow[hook, from=2-1, to=2-2]
	\arrow[hook, from=1-2, to=2-2]
	\arrow[hook, from=1-1, to=2-1]
\end{tikzcd}\]which will be explained in the following paragraphs.

\medskip{}

\begin{defn}
[Cancellative ordered commutative monoids \cite{kehayopulu1998separative,fritz2017resource}]
$\mathsf{COrdCMon}$ is the category where each object is a cancellative
ordered commutative monoid, i.e., $M\in\mathsf{OrdCMon}$ such that
for all $x,y,z\in M$, $x+z\ge y+z$ implies $x\ge y$ \cite{kehayopulu1998separative,fritz2017resource}.
Morphisms in $\mathsf{COrdCMon}$ are order-preserving homomorphisms. 
\end{defn}

\medskip{}

\begin{defn}
[Ordered abelian groups \cite{fuchs2011partially}] $\mathsf{OrdAb}$
is the category where each object is an ordered abelian group, i.e.,
an abelian group $G$ with a partial order $\ge$ satisfying that
$x\ge y$ implies $x+z\ge y+z$ for all $z$ \cite{fuchs2011partially}.
Morphisms in $\mathsf{OrdAb}$ are order-preserving homomorphisms. 
\end{defn}

\medskip{}

\begin{defn}
[Ordered vector spaces \cite{narici2010topological}] $\mathsf{OrdVect}_{\mathbb{Q}}$
is the category where each object is an ordered vector space over
$\mathbb{Q}$, i.e., a vector space $V$ over $\mathbb{Q}$ with a
partial order $\ge$ such that for all $x,y,z\in V$, $a\in[0,\infty)$,
$x\ge y$ implies $x+z\ge y+z$, and $x\ge y$ implies $ax\ge ay$
\cite{narici2010topological}. Morphisms in $\mathsf{OrdVect}_{\mathbb{Q}}$
are order-preserving linear functions. 
\end{defn}

\medskip{}

The following inclusions are straightforward:
\[
\mathsf{OrdVect}_{\mathbb{Q}}\hookrightarrow\mathsf{OrdAb}\hookrightarrow\mathsf{COrdCMon}\hookrightarrow\mathsf{OrdCMon}.
\]
It has been shown in \cite{fritz2017resource} that the above three
embeddings are reflective subcategories. We briefly describe the reflections
below:
\begin{itemize}
\item The reflection of $M\in\mathsf{OrdCMon}$ in $\mathsf{COrdCMon}$
is constructed as follows \cite{fritz2017resource}. Define a preorder
$\succeq$ over $M$ as $x\succeq y$ if there exists $z$ such that
$x+z\ge y+z$. The reflection is the cancellative ordered commutative
monoid $(M/\sim,\,\succeq)$, where $\sim$ is an equivalence relation
defined as $x\sim y$ if $x\succeq y\succeq x$. This is referred
as \emph{catalytic regularization} in \cite{fritz2017resource}.
\end{itemize}
\medskip{}

\begin{itemize}
\item The reflection of $M\in\mathsf{COrdCMon}$ in $\mathsf{OrdAb}$ is
the standard Grothendieck group of differences. More explicitly, the
reflection $G$ contains formal differences in the form $x-y$, where
$x,y\in M$, modulo the equivalence relation $x-y\sim z-w$ if $x+w=z+y$
in $M$, with operations $-(x-y)=(y-x)$, $(x-y)+(z-w)=(x+z)-(y+w)$,
and $x-y\ge z-w$ if $x+w\ge z+y$ in $M$. 
\end{itemize}
\medskip{}

\begin{itemize}
\item The reflection of $G\in\mathsf{OrdAb}$ in $\mathsf{OrdVect}_{\mathbb{Q}}$
is constructed as follows. Define a preorder $\succeq$ over $G$
by $x\succeq y$ if there exists $n\in\mathbb{Z}_{>0}$ such that
$nx\ge ny$. This is referred as \emph{many-copy convertibility} in
\cite{fritz2017resource}. It was shown in \cite{fritz2017resource}
that the quotient group $\tilde{G}:=G/\{x\in G:\,0\succeq x\succeq0\}$
is torsion-free. The reflection of $G$ is the vector space $V$ generated
by $\tilde{G}$, with a positive cone given by the conic hull of the
positive cone $\{x\in G:\,x\succeq0\}$ of $G$.\footnote{An earlier version of the preprint of this paper incorrectly claims
that $\mathsf{OrdVect}_{\mathbb{R}}\hookrightarrow\mathsf{OrdAb}$
is a reflective subcategory. This is untrue since the vector space
$\mathbb{R}$ (equipped with the discrete ordering), after being treated
as an abelian group, would become indistinguishable with an uncountable-dimensional
vector space over $\mathbb{Q}$ by considering the Hamel basis of
$\mathbb{R}$ over $\mathbb{Q}$.}
\end{itemize}
\medskip{}

It may also be desirable for the codomain of entropy to satisfy an
``approximate'' version of the many-copy convertibility property
in \cite{fritz2017resource}. For elements $x,y$, if there exists
$a,b$ such that $nx+a\ge ny+b$ for all $n\in\mathbb{Z}_{>0}$, then
we might expect $x\ge y$ to also hold. This is captured by the following
definition, which is a straightforward extension of the concept of
integrally closed (partially) ordered groups \cite{clifford1940partially,glass1999partially}
to cancellative ordered commutative monoids.\medskip{}

\begin{defn}
[Integrally closed ordered commutative monoids \cite{clifford1940partially,glass1999partially}]
\label{def:icocm}An ordered commutative monoid $M\in\mathsf{COrdCMon}$
is called \emph{integrally closed} if for $x,y,a,b\in M$, if $nx+a\ge ny+b$
for all $n\in\mathbb{Z}_{>0}$, then $x\ge y$.\footnote{The definition here is based on the concept of integrally closed (partially)
ordered groups \cite{clifford1940partially,glass1999partially}, where
an ordered group is integrally closed if $nx\le a$ for all $n\in\mathbb{Z}_{>0}$
implies $x\le0$. It is clear that Definition \ref{def:icocm}, when
applied on an ordered group, gives the conventional definition of
integrally closed partially ordered groups. Also note that the Grothendieck
group of differences of an integrally closed ordered commutative monoid
is an integrally closed ordered abelian group.} Let $\mathsf{IcOrdCMon}$ be the category of integrally closed ordered
commutative monoid, $\mathsf{IcOrdAb}$ be the category of integrally
closed ordered abelian group \cite{clifford1940partially,glass1999partially}
(full subcategory of $\mathsf{OrdAb}$), and $\mathsf{IcOrdVect}_{\mathbb{Q}}$
be the category of integrally closed ordered vector spaces, which
coincides with the Archimedean ordered vector spaces \cite{schaefer1999topological,emelyanov2017archimedean}
(full subcategory of $\mathsf{OrdVect}_{\mathbb{Q}}$).
\end{defn}

\medskip{}

An integrally closed ordered commutative monoid $M$ must also be
cancellative. To show this, for $x,y,z\in M$ with $x+z\ge y+z$,
we have $nx+z\ge(n-1)x+y+z\ge\cdots\ge ny+z$ for every $n\in\mathbb{Z}_{>0}$,
and hence $x\ge y$ by integral closedness. Therefore, $\mathsf{IcOrdCMon}$
is a full subcategory of $\mathsf{COrdCMon}$.

Moreover, $\mathsf{IcOrdCMon}$ is a reflective subcategory of $\mathsf{COrdCMon}$,
where the reflection of $M\in\mathsf{COrdCMon}$ is constructed as
follows. Let $\succeq$ be the smallest integrally closed preorder
over $M$ that contains $\ge$, i.e., it is the intersection of all
preorders $\tilde{\succeq}$ over $M$ which satisfies that if $x\ge y$,
then $x\,\tilde{\succeq}\,y$; and if $nx+a\,\tilde{\succeq}\,ny+b$
for all $n\in\mathbb{Z}_{>0}$, then $x\,\tilde{\succeq}\,y$. The
reflection is then taken as $\bar{M}=(M/\sim,\,\succeq)$, where $\sim$
is an equivalence relation defined as $x\sim y$ if $x\succeq y\succeq x$.
The reflection morphism $r:M\to\bar{M}$ maps an element to its equivalence
class. To show the desired adjunction, for every morphism $f:M\to S$
in $\mathsf{COrdCMon}$ where $S\in\mathsf{IcOrdCMon}$, define a
preorder $\tilde{\succeq}$ over $M$ by $x\,\tilde{\succeq}\,y$
if $f(x)\ge f(y)$. For $x,y,a,b\in M$, if $nx+a\,\tilde{\succeq}\,ny+b$
for all $n\in\mathbb{Z}_{>0}$, then $nf(x)+f(a)\ge nf(y)+f(b)$ for
all $n\in\mathbb{Z}_{>0}$, and hence $f(x)\ge f(y)$ since $S$ is
integrally closed, and $x\,\tilde{\succeq}\,y$. Therefore, $\tilde{\succeq}$
is an integrally closed preorder. Let $\bar{f}:\bar{M}\to S$ be defined
as $\bar{f}([x])=f(x)$, where $[x]$ denotes the equivalence class
of $x\in M$ modulo $\sim$. To check that $\bar{f}$ is well-defined
and order-preserving, if $x\succeq y$, then $x\,\tilde{\succeq}\,y$
by definition of $\succeq$, and $f(x)\ge f(y)$. Hence $\bar{f}r=f$.
Such $\bar{f}$ satisfying $\bar{f}r=f$ is clearly unique. Hence
$\mathsf{IcOrdCMon}$ is a reflective subcategory of $\mathsf{COrdCMon}$.
We can similarly show that $\mathsf{IcOrdAb}$ is a reflective subcategory
of $\mathsf{OrdAb}$, and that $\mathsf{IcOrdVect}_{\mathbb{Q}}$
is a reflective subcategory of $\mathsf{OrdVect}_{\mathbb{Q}}$ (this
fact about $\mathsf{IcOrdVect}_{\mathbb{Q}}$ was also shown in \cite{emelyanov2017archimedean}).

\medskip{}

\section{Universal Entropy as the Reflection Morphism\label{sec:univ}}

In the previous section, we have justified considering the codomain
of entropy to be an object of $\mathsf{OrdCMon}$, which is a reflective
subcategory of $\mathsf{MonSiMonCat}$ via the embedding $R:\mathsf{OrdCMon}\to\mathsf{MonSiMonCat}$.
We can also be more restrictive and require the codomain to be an
object of $\mathsf{IcOrdCMon}$, $\mathsf{IcOrdAb}$ or other subcategories
of $\mathsf{OrdCMon}$. Therefore, an entropy is simply defined as
a MonSiMon functor to a specified subcategory of $\mathsf{OrdCMon}$. 

\medskip{}

\begin{defn}
[$\mathsf{V}$-entropy]\label{def:gcm_entropy_non}Consider a MonSiMon
category $\mathsf{C}\in\mathsf{MonSiMonCat}$, and a subcategory $\mathsf{V}$
of $\mathsf{OrdCMon}$. A $\mathsf{V}$\emph{-entropy} of $\mathsf{C}$
is a morphism $\mathsf{C}\to R(\mathsf{W})$ in $\mathsf{MonSiMonCat}$
where $\mathsf{W}\in\mathsf{V}$.
\end{defn}

\medskip{}

A $\mathsf{V}$-entropy must be in the form $(!,h)$, where $!:\mathsf{C}\to*$
is the unique lax monoidal functor, and $h:N\Rightarrow\Delta_{\mathsf{W}}$
(where $\mathsf{W}\in\mathsf{V}$) is a monoidal natural transformation,
where $\Delta_{\mathsf{W}}:\mathsf{C}^{\mathrm{op}}\to\mathsf{MonCat}_{\ell}$
denotes the lax monoidal functor that sends everything to the monoidal
posetal category $\mathsf{W}$, with a coherence map given by the
tensor product of $\mathsf{W}$. We sometimes simply call $h$ a $\mathsf{V}$-entropy.

The universal entropy is then the universal MonSiMon functor from
$\mathsf{C}$ to $\mathsf{V}$, i.e., a morphism $F:\mathsf{C}\to R(\mathsf{W})$
in $\mathsf{MonSiMonCat}$ (where $\mathsf{W}\in\mathsf{V}$), satisfying
that every $F':\mathsf{C}\to R(\mathsf{W}')$, $\mathsf{W}'\in\mathsf{V}$
can be factorized as $F'=R(G)F$ (where $G:\mathsf{W}\to\mathsf{W}'$
is a morphism in $\mathsf{V}$) in a unique manner. In this paper,
we focus on the case where $\mathsf{V}$ is a reflective subcategory
of $\mathsf{OrdCMon}$, which guarantees that $\mathsf{V}$ is a reflective
subcategory of $\mathsf{MonSiMonCat}$ as well by Lemma \ref{lem:lgcm}.
In this case, the universal MonSiMon functor can be given as the reflection
morphism.

\medskip{}

\begin{defn}
[Universal $\mathsf{V}$-entropy]\label{def:gcm_entropy}Consider
a MonSiMon category $\mathsf{C}\in\mathsf{MonSiMonCat}$, and a reflective
subcategory $\mathsf{V}$ of $\mathsf{OrdCMon}$, the \emph{universal}
$\mathsf{V}$\emph{-entropy} of $\mathsf{C}$ is the reflection morphism
from $\mathsf{C}$ to $\mathsf{V}$. More explicitly, if $L':\mathsf{MonSiMonCat}\to\mathsf{V}$
is the left adjoint to the embedding functor $R':\mathsf{V}\to\mathsf{MonSiMonCat}$
with an adjunction unit $\phi:\mathrm{id}_{\mathsf{MonSiMonCat}}\Rightarrow R'L'$,
then the universal $G$-entropy of $\mathsf{C}$ is given by the reflection
morphism
\[
\phi_{\mathsf{C}}:\mathsf{C}\to T(\mathsf{C}),
\]
where we call $T=R'L':\mathsf{MonSiMonCat}\to\mathsf{MonSiMonCat}$
the \emph{universal} $\mathsf{V}$\emph{-entropy monad}.
\end{defn}

\medskip{}

Definition \ref{def:gcm_entropy} allows us to define the universal
$\mathsf{V}$-entropy of any MonSiMon category, for $\mathsf{V}$
being $\mathsf{OrdCMon}$, $\mathsf{IcOrdCMon}$, $\mathsf{IcOrdAb}$,
or other reflective subcategories discussed in Section \ref{subsec:cat_of_codomains}.

In the remainder of this paper, we will often omit the embedding functors
$R$, $R'$ and regard $\mathsf{V}\subseteq\mathsf{OrdCMon}\subseteq\mathsf{MonSiMonCat}$
for any reflective subcategory $\mathsf{V}$ of $\mathsf{OrdCMon}$,
so a $\mathsf{V}$-entropy of $\mathsf{C}$ would be a morphism $\mathsf{C}\to\mathsf{W}$
in $\mathsf{MonSiMonCat}$ where $\mathsf{W}\in\mathsf{V}$.

We will now discuss how Definition \ref{def:gcm_entropy_non} captures
the notion of entropy. Assume $\mathsf{V}\subseteq\mathsf{OrdCMon}$.
Let the coherence maps of the lax monoidal functor $N:\mathsf{C}^{\mathrm{op}}\to\mathsf{MonCat}_{\ell}$
be $\zeta:N(-)\times N(-)\Rightarrow N(-\otimes-)$ (where $\zeta_{P,Q}$
maps the ``random variables''\footnote{The double quotes around ``random variables'' are because the considered
MonSiMon category might not be $(\mathsf{FinProb},-/\mathsf{FinProb})$,
though one may still think of $P\in\mathsf{C}$ as a ``probability
space'' and $X\in N(P)$ as a ``random variable'' for the sake
of intuition.} $X\in N(P)$ and $Y\in N(Q)$ to the ``joint random variable''
$\zeta_{P,Q}(X,Y)$ in the ``product probability space'' $P\otimes Q$)
and $\epsilon:*\to N(I_{\mathsf{C}})$ (picking the ``degenerate
random variable'' in the ``degenerate probability space'' $N(I_{\mathsf{C}})$).
Consider a $\mathsf{V}$-entropy given as a MonSiMon functor $(!,h)$
from $(\mathsf{C},N)$ to $(*,\,*\stackrel{\mathsf{W}}{\to}\mathsf{MonCat}_{\ell})$,
$\mathsf{W}\in\mathsf{V}$. The tensor product of $\mathsf{W}$ is
denoted as $+$, and the tensor unit is denoted as $0$. For a ``random
variable'' $X\in N(P)$ over the ``probability space'' $P\in\mathsf{C}$,
its entropy is $h_{P}(X)\in\mathsf{W}$. The following properties
follow from $h$ being a monoidal natural transformation.

\medskip{}

\begin{itemize}
\item (Monotonicity) For ``random variables'' $X,Y\in N(P)$ ($P\in\mathsf{C}$)
where ``$Y$ is a function of $X$'', i.e., there is a morphism
$X\to Y$, we have
\[
h_{P}(X)\ge h_{P}(Y).
\]
This follows from the functoriality of $h_{P}:N(P)\to\mathsf{W}$,
and is precisely the monotonicity property of entropy \eqref{eq:monotone}.
\end{itemize}
\medskip{}

\begin{itemize}
\item (Subadditivity) For ``random variables'' $X,Y\in N(P)$, their ``joint
random variable'' $X\otimes_{N(P)}Y$ satisfies
\begin{equation}
h_{P}(X)+h_{P}(Y)\ge h_{P}(X\otimes_{N(P)}Y),\label{eq:monnat_subadditive}
\end{equation}
which follows from $h_{P}$ being a lax monoidal functor. This is
precisely the subadditivity property of entropy \eqref{eq:subadditive}. 
\end{itemize}
\medskip{}

\begin{itemize}
\item (Extension of probability space) For ``random variable'' $X\in N(Q)$
and an ``extension of probability space'' $f:P\to Q$, the ``random
variable in the extended space'' $N(f)(X)\in N(P)$ has the same
entropy, i.e.,
\begin{equation}
h_{P}(N(f)(-))=h_{Q}(-),\label{eq:invar_extend}
\end{equation}
which follows from the naturality of $h$. Basically, entropy is invariant
under extension of probability space, which is expected from every
reasonable probabilistic concept \cite{tao2012topics}. 
\end{itemize}
\medskip{}

\begin{itemize}
\item (Additivity) For ``random variables'' $X\in N(P)$, $Y\in N(Q)$,
their ``product random variable'' $\zeta_{P,Q}(X,Y)$ over the ``product
probability space'' $P\otimes Q$ satisfies
\[
h_{P\otimes Q}(\zeta_{P,Q}(X,Y))=h_{P}(X)+h_{Q}(Y),
\]
which follows from $h$ being a monoidal natural transformation:
\[\begin{tikzcd}
	{N(P)\times N(Q)} && {\mathsf{W}\times \mathsf{W}} && {*} \\
	{N(P\otimes Q)} && {\mathsf{W}} && {N(I_{\mathsf{C}} )} & {\mathsf{W}}
	\arrow["{\zeta_{P,Q}}"', from=1-1, to=2-1]
	\arrow["{h_P \times h_Q}", from=1-1, to=1-3]
	\arrow["{+}", from=1-3, to=2-3]
	\arrow["{h_{P\otimes_{\mathsf{C}} Q}}"', from=2-1, to=2-3]
	\arrow["\epsilon"', from=1-5, to=2-5]
	\arrow["0", from=1-5, to=2-6]
	\arrow["{h_{I_{\mathsf{C}}}}"', from=2-5, to=2-6]
\end{tikzcd}\] This is precisely the additivity property of entropy \eqref{eq:additive}.
\end{itemize}
\medskip{}

Definition \ref{def:gcm_entropy} tells us that the universal entropy
is given by the MonSiMon functor that is universal among MonSiMon
functors that satisfies the aforementioned properties. Definition
\ref{def:gcm_entropy} allows us to find the universal $\mathsf{OrdCMon}$-entropy
of any $\mathsf{C}\in\mathsf{MonSiMonCat}$ by Lemma \ref{lem:lgcm}.
Different choices of reflective subcategories $\mathsf{V}$ will greatly
affect the resultant universal entropy. In the following subsections,
we will explore the universal $\mathsf{OrdCMon}$, $\mathsf{IcOrdCMon}$,
and $\mathsf{IcOrdAb}$-entropies of various MonSiMon categories.

Conventionally, we not only talk about the entropy of a random variable,
but also the entropy of a probability distribution. Here, in the special
case where $N=-/\mathsf{C}$ is the under category functor of $\mathsf{C}$
with coherence map given by $\otimes_{\mathsf{C}}$ (see Section \ref{sec:gcmcat}),
we can also define the entropy of a ``probability space'' $P\in\mathsf{C}$.
This is a consequence of the invariance under extension of probability
space property, which tells us that the entropy $h_{P}(X)$ of the
random variable $X:P\to Q$ only depends on $Q$.

\medskip{}

\begin{prop}
\label{prop:baseless}Consider the MonSiMon category $(\mathsf{C},N)$
where $N=-/\mathsf{C}$ with coherence map given by $\otimes_{\mathsf{C}}$,
and a subcategory $\mathsf{V}$ of $\mathsf{OrdCMon}$. For any $\mathsf{V}$-entropy
$(!,h)$, $h:N\Rightarrow\Delta_{\mathsf{W}}$, $\mathsf{W}\in\mathsf{V}$,
there is a unique strongly monoidal functor $H:\mathsf{C}\to\mathsf{W}$
satisfying $Uh=UH\circ\mathrm{Cod}$ (both are monoidal natural transformations),
i.e., we have the following equality in $\mathbb{U}\mathsf{MonCat}_{\ell}$:\\
\[\begin{tikzcd}
	{\mathsf{C}^\mathrm{op}} & {*} && {\mathsf{C}^\mathrm{op}} & {*} \\
	&& {=} \\
	{\mathsf{MonCat}_\ell} &&& {\mathsf{MonCat}_\ell} & {\mathsf{MonCat}_\ell} \\
	{\mathsf{Cat}} &&& {\mathsf{Cat}}
	\arrow[""{name=0, anchor=center, inner sep=0}, "N"', from=1-1, to=3-1]
	\arrow["{!}", from=1-1, to=1-2]
	\arrow[""{name=1, anchor=center, inner sep=0}, "{\mathsf{W}}", from=1-2, to=3-1]
	\arrow["U"', from=3-1, to=4-1]
	\arrow[""{name=2, anchor=center, inner sep=0}, "N"', from=1-4, to=3-4]
	\arrow["U"', from=3-4, to=4-4]
	\arrow["{!}", from=1-4, to=1-5]
	\arrow[""{name=3, anchor=center, inner sep=0}, "{\mathsf{C}}"{description}, from=1-5, to=3-5]
	\arrow["U", from=3-5, to=4-4]
	\arrow[""{name=4, anchor=center, inner sep=0}, "{\mathsf{W}}", curve={height=-30pt}, from=1-5, to=3-5]
	\arrow["h", shorten <=4pt, shorten >=4pt, Rightarrow, from=0, to=1]
	\arrow["{\mathrm{Cod}}", shorten <=9pt, shorten >=9pt, Rightarrow, from=2, to=3]
	\arrow["H", shorten <=6pt, shorten >=6pt, Rightarrow, from=3, to=4]
\end{tikzcd}\]where $U:\mathsf{MonCat}_{\ell}\to\mathsf{Cat}$ is the forgetful
functor, $\mathrm{Cod}:UN\Rightarrow U\Delta_{\mathsf{C}}$ is the
codomain natural transformation (a monoidal natural transformation)
with component $\mathsf{Cod}_{P}(X)=Q$ for $X:P\to Q$. In this case,
we call $H$ the \emph{baseless $\mathsf{V}$-entropy} corresponding
to $h$.
\end{prop}

\medskip{}

\begin{proof}
Given a $\mathsf{V}$-entropy $h:N\Rightarrow\Delta_{\mathsf{W}}$,
we define $H:\mathsf{C}\to\mathsf{W}$, $H(P)=h_{P}(\mathrm{id}_{P})$
for $P\in\mathsf{C}$, and $H(f)=h_{P}(f)$ for $f:P\to Q$. Since
the coherence map $\zeta$ of $N$ is given by the tensor product,
$H(P\otimes Q)=h_{P\otimes Q}(\mathrm{id}_{P\otimes Q})=h_{P\otimes Q}(\zeta_{P,Q}(\mathrm{id}_{P},\mathrm{id}_{Q}))=h_{P}(\mathrm{id}_{P})\otimes h_{Q}(\mathrm{id}_{Q})$,
where the last equality is because $h$ is a monoidal natural transformation.
Hence $H$ is strongly monoidal. For $X:P\to Q$, by invariance under
extension of probability space, we have $h_{P}(X)=h_{Q}(\mathrm{id}_{Q})=H(Q)=H(\mathrm{Cod}_{P}(X))$.
Hence $Uh=UH\circ\mathrm{Cod}$. For uniqueness, if $Uh=UH\circ\mathrm{Cod}$,
then $H(P)=H(\mathrm{Cod}_{P}(\mathrm{id}_{P}))=h_{P}(\mathrm{id}_{P})$
is uniquely determined by $h$.
\end{proof}
\medskip{}

Therefore, for the special case $N=-/\mathsf{C}$, there is a one-to-one
correspondence between $\mathsf{V}$-entropy $(!,h)$ and baseless
$\mathsf{V}$-entropy $H:\mathsf{C}\to\mathsf{W}$, given by $H(P)=h_{P}(\mathrm{id}_{P})$
and $h_{P}(X)=H(\mathrm{Cod}_{P}(X))$. The entropy of a random variable
$X:P\to Q$ can be obtained as $h_{P}(X)=H(Q)$, which only depends
on its distribution $Q$. To check whether $H:\mathsf{C}\to\mathsf{W}$
is a baseless entropy  when $\mathsf{W}\in\mathsf{OrdCMon}$, it
suffices to check whether $H$ satisfies the monotonicity \eqref{eq:monotone},
additivity \eqref{eq:additive} and subadditivity \eqref{eq:subadditive}
properties. More precisely, the monotonicity property is that for
$P,Q\in\mathsf{C}$ where there is a morphism $P\to Q$, we have
\begin{equation}
H(P)\ge H(Q),\label{eq:baseless_monotone}
\end{equation}
i.e., $H$ is a functor. The additivity property is that for $P,Q\in\mathsf{C}$,
\begin{equation}
H(P)+H(Q)=H(P\otimes Q),\label{eq:baseless_additive}
\end{equation}
i.e., $H$ is strongly monoidal. The subadditivity property is that
for $X,Y\in N(P)$ (where $P\in\mathsf{C}$), $X\otimes_{N(P)}Y$
(which is their ``joint random variable $(X,Y)$'') satisfies
\begin{equation}
H(\mathrm{Cod}_{P}(X))+H(\mathrm{Cod}_{P}(Y))\ge H(\mathrm{Cod}_{P}(X\otimes_{N(P)}Y)).\label{eq:baseless_subadditive}
\end{equation}
Recall that $\mathrm{Cod}_{P}:N(P)\to\mathsf{C}$ is the codomain
functor, and $\mathrm{Cod}_{P}(X)$ is the distribution of the random
variable $X$. These properties can clearly be deduced from the discussion
on monotonicity, additivity and subadditivity earlier in this section.

We can then define the universal baseless $\mathsf{V}$-entropy.

\medskip{}

\begin{defn}
[Universal baseless $\mathsf{V}$-entropy]\label{def:baseless-1}Consider
the MonSiMon category $(\mathsf{C},N)$ where $N=-/\mathsf{C}$ with
coherence map given by $\otimes_{\mathsf{C}}$, and a reflective subcategory
$\mathsf{V}$ of $\mathsf{OrdCMon}$. We call $H:\mathsf{C}\to\mathsf{W}$
a \emph{baseless }$\mathsf{V}$\emph{-entropy} if it is the baseless
$\mathsf{V}$-entropy corresponding to some $\mathsf{V}$-entropy
(see Proposition \ref{prop:baseless}). The \emph{universal baseless
$\mathsf{V}$-entropy} is a baseless $\mathsf{V}$-entropy $H:\mathsf{C}\to\mathsf{W}$
such that for every baseless $\mathsf{V}$-entropy $H':\mathsf{C}\to\mathsf{W}'$,
there is a unique $F:\mathsf{W}\to\mathsf{W}'$ in $\mathsf{V}$ such
that $H'=FH$. 
\end{defn}

\medskip{}

For the special case $N=-/\mathsf{C}$, it is straightforward to check
that $H$ is a universal baseless $\mathsf{V}$-entropy if and only
if $H$ is the baseless $\mathsf{V}$-entropy corresponding to a universal
$\mathsf{V}$-entropy $h$. These two concepts are equivalent in this
special case. Since it is often simpler to work with a monoidal functor
$H$ instead of a monoidal natural transformation $h$, we will often
use the baseless $\mathsf{V}$-entropy instead of the $\mathsf{V}$-entropy
in subsequent sections.

Nevertheless, we remark that the $\mathsf{V}$-entropy $(!,h)$ (which
is a MonSiMon functor) is the more natural notion of entropy compared
to the baseless $\mathsf{V}$-entropy $H$ (which is a strongly monoidal
functor satisfying an additional subadditivity property). The $\mathsf{V}$-entropy
$h$ is over random variables, whereas the baseless $\mathsf{V}$-entropy
$H$ is over probability distributions. 

Conventionally, although we can talk about the entropy of both random
variables and distributions, the usual understanding is that entropy
is associated with distributions (e.g. \cite{baez2011entropy}), and
the ``entropy of a random variable'' is merely a shorthand for the
``entropy of the distribution of the random variable''. In this
paper, it is the other way around, where we first define the $\mathsf{V}$-entropy,
and recover the baseless $\mathsf{V}$-entropy only for the special
case $N=-/\mathsf{C}$. The reason is threefold. First, the baseless
entropy is only defined when $N=-/\mathsf{C}$, and does not work
if $N(P)$ does not match $P/\mathsf{C}$ (e.g. the MonSiMon category
$(\mathsf{Prob},\,-/\mathsf{FinProb})$ in Section \ref{sec:shannon}).
Second, while the monotonicity and additivity properties can be stated
in terms of either $h$ or $H$, it is more natural to state the subadditivity
property as $h$ having lax monoidal components. Comparing \eqref{eq:monnat_subadditive}
with \eqref{eq:baseless_subadditive}, we can see that the former
is more natural. Subadditivity is fundamentally a property about random
variables, not distributions. Third, having the universal $\mathsf{V}$-entropy
being a reflection morphism in $\mathsf{MonSiMonCat}$ allows us to
connect the universal $\mathsf{V}$-entropies of various MonSiMon
categories in a natural manner (see Section \ref{sec:global}). Although
we often use the baseless $\mathsf{V}$-entropy instead of the $\mathsf{V}$-entropy
in subsequent sections, we emphasize that this is only due to notational
simplicity, and not because the baseless entropy is a more elegant
concept.

\medskip{}

\section{Shannon and Hartley Entropy for Finite Probability Spaces\label{sec:shannon}}

The following result gives the universal $\mathsf{COrdCMon}$-entropy
of $\mathsf{FinProb}$, which is stated in terms of the Shannon entropy
$H_{1}$ \cite{shannon1948mathematical} and the Hartley entropy $H_{0}$
\cite{hartley1928transmission}. The proof will be given later in
this section.

\medskip{}

\begin{thm}
\label{thm:finprob_ordcmon}The universal $\mathsf{COrdCMon}$-entropy
of the MonSiMon category $(\mathsf{FinProb},\,-/\mathsf{FinProb})$
is given by $(!,h)$:\\
\[\begin{tikzcd}
	{\mathsf{FinProb}} && {*} \\
	\\
	&& {\mathsf{MonCat}_\ell}
	\arrow[""{name=0, anchor=center, inner sep=0}, "{-/\mathsf{FinProb}}"', from=1-1, to=3-3]
	\arrow["{!}", from=1-1, to=1-3]
	\arrow["{(\mathsf{FinProb},\succsim_{01})}", from=1-3, to=3-3]
	\arrow["h"'{pos=0.4}, shorten <=6pt, shorten >=3pt, Rightarrow, from=0, to=1-3]
\end{tikzcd}\] where:
\begin{itemize}
\item $(\mathsf{FinProb},\succsim_{01})\in\mathsf{COrdCMon}$ is a cancellative
ordered commutative monoid,\footnote{Technically, we have to consider the quotient $\mathsf{FinProb}/\cong$
modulo the isomorphism relation $\cong$ in $\mathsf{FinProb}$ in
order to make $\succsim_{01}$ a partial order, though we omit the
quotient for the sake of notational simplicity.} where the monoid operation is the tensor product over $\mathsf{FinProb}$
(product distribution), and for $P,Q\in\mathsf{FinProb}$, $P\succsim_{01}Q$
if and only if
\[
P\cong Q\;\mathrm{or}\;\left(H_{0}(P)\ge H_{0}(Q)\;\mathrm{and}\;H_{1}(P)>H_{1}(Q)\right).
\]
\item $h:-/\mathsf{FinProb}\Rightarrow\Delta_{(\mathsf{FinProb},\succsim_{01})}$
has component $h_{P}$ mapping a random variable $X:P\to Q$ to its
distribution $Q$. 
\end{itemize}
Equivalently, the universal baseless $\mathsf{COrdCMon}$-entropy
is the obvious functor $\mathsf{FinProb}\to(\mathsf{FinProb},\succsim_{01})$.
\end{thm}

\medskip{}

Theorem \ref{thm:finprob_ordcmon} is unsatisfactory since the ``entropy''
of a distribution $P\in\mathsf{FinProb}$ is merely given as $P$
itself in a space with a new ordering $\succsim_{01}$. We expect
the entropy to be given as real numbers instead. The issue is that
$\succsim_{01}$ is ``too fine'' and is unable to identify distributions
that should have the same entropies. In order to recover the familiar
notions of entropy, we have to impose that the codomain of the entropy
is integrally closed (Definition \ref{def:icocm}), i.e., it is in
$\mathsf{IcOrdCMon}$. The integral closedness condition will make
the partial order ``coarser'' so as to identify distributions having
the same entropies as being equivalent under this partial order.

We now present one of our main results, which shows that the universal
$\mathsf{IcOrdCMon}$-entropy of $\mathsf{FinProb}$ is the pairing
of the Shannon entropy $H_{1}$ and the Hartley entropy $H_{0}$.
This allows us to recover two familiar notions of entropy. The proof
will be given later in this section.

\medskip{}

\begin{thm}
\label{thm:finprob_ic}The universal $\mathsf{IcOrdCMon}$-entropy
of the MonSiMon category $(\mathsf{FinProb},\,-/\mathsf{FinProb})$
is given by $(!,(h_{0},h_{1}))$:\\
\[\begin{tikzcd}
	{\mathsf{FinProb}} && {*} \\
	\\
	&& {\mathsf{MonCat}_\ell}
	\arrow[""{name=0, anchor=center, inner sep=0}, "{-/\mathsf{FinProb}}"', from=1-1, to=3-3]
	\arrow["{!}", from=1-1, to=1-3]
	\arrow["{(H_{0},H_{1})(\mathsf{FinProb})}", from=1-3, to=3-3]
	\arrow["{(h_0,h_1)}"'{pos=0.4}, shorten <=6pt, shorten >=3pt, Rightarrow, from=0, to=1-3]
\end{tikzcd}\]where 
\begin{align*}
(H_{0},H_{1})(\mathsf{FinProb}) & :=\left\{ (H_{0}(P),H_{1}(P)):\,P\in\mathsf{FinProb}\right\} \\
 & =\left\{ (a,b)\in(\log\mathbb{N})\times\mathbb{R}_{>0}:\,a\ge b\right\} \cup\{(0,0)\}\\
 & \subseteq\mathbb{R}^{2}\in\mathsf{IcOrdCMon}
\end{align*}
is the range of $(H_{0},H_{1})$ equipped with the product order over
$\mathbb{R}^{2}$, and $(h_{0},h_{1}):-/\mathsf{FinProb}\Rightarrow\Delta_{(H_{0},H_{1})(\mathsf{FinProb})}$
has component $(h_{0,P},h_{1,P})$ mapping a random variable $X:P\to Q$
to the pair formed by its Hartley and Shannon entropies $(H_{0}(Q),H_{1}(Q))\in\mathbb{R}^{2}$.
Equivalently, the universal baseless $\mathsf{IcOrdCMon}$-entropy
is the pairing $(H_{0},H_{1}):\mathsf{FinProb}\to(H_{0},H_{1})(\mathsf{FinProb})$.
\end{thm}

\medskip{}

If the slightly complicated codomain $(H_{0},H_{1})(\mathsf{FinProb})$
is still not satisfactory, we can consider codomains in $\mathsf{IcOrdAb}$
or $\mathsf{IcOrdVect}$ to further simplify the codomain. The following
results are direct corollaries of Theorem \ref{thm:finprob_ic}.

\medskip{}

\begin{cor}
\label{cor:finprob_ab_vect}For the MonSiMon category $(\mathsf{FinProb},\,-/\mathsf{FinProb})$:
\begin{enumerate}
\item The universal baseless $\mathsf{IcOrdAb}$-entropy is given by $(H_{0},H_{1}):\mathsf{FinProb}\to(\log\mathbb{Q}_{>0})\times\mathbb{R}$.
\item The universal baseless $\mathsf{IcOrdVect}_{\mathbb{Q}}$-entropy
is given by $(H_{0},H_{1}):\mathsf{FinProb}\to(\mathbb{Q}\log\mathbb{Q}_{>0})\times\mathbb{R}$,
where $(\mathbb{Q}\log\mathbb{Q}_{>0})\times\mathbb{R}=\{(a\log b,c):\,a\in\mathbb{Q},b\in\mathbb{Q}_{>0},c\in\mathbb{R}\}$.
\end{enumerate}
\end{cor}

\medskip{}

These universal entropies can be summarized by the following diagram,
where the categories on the left are the ``categories of codomains''
of entropies (see Section \ref{subsec:cat_of_codomains}), and the
categories on the right are elements of the corresponding categories
on the left. The arrow $(H_{0},H_{1})\circ\mathrm{Cod}$ is the universal
$\mathsf{IcOrdCMon}$-entropy of $(\mathsf{FinProb},\,-/\mathsf{FinProb})$
in Theorem \ref{thm:finprob_ic}.\footnote{This notation is justified by the one-to-one correspondence $Uh=UH\circ\mathrm{Cod}$
in Proposition \ref{prop:baseless}, where we omit the forgetful functor
$U$ for brevity.} We write $(\mathsf{FinProb},\,-/\mathsf{FinProb})$ as $-/\mathsf{FinProb}$
for brevity. Here $G$ denotes the reflection of $(\mathsf{FinProb},\succsim_{01})$
in $\mathsf{OrdAb}$, and $V$ denotes the reflection of $(\mathsf{FinProb},\succsim_{01})$
in $\mathsf{OrdVect}_{\mathbb{Q}}$ (see Section \ref{subsec:cat_of_codomains}).
The ``$?$'' is the codomain of the universal $\mathsf{OrdCMon}$
entropy of $(\mathsf{FinProb},\,-/\mathsf{FinProb})$, which is guaranteed
to exist but intractable and uninteresting, and hence omitted.\footnote{We may combine the categories on the left and those on the right into
a single gigantic diagram, by noting that ``$-/\mathsf{FinProb}\in\mathsf{MonSiMonCat}$''
can be written as ``$*\stackrel{-/\mathsf{FinProb}}{\longrightarrow}\mathsf{MonSiMonCat}$'',
and turning all the morphisms on the right into 2-cells. We prefer
not to draw the diagram this way since it would be quite unreadable.}\\
\[\begin{tikzcd}[column sep=scriptsize]
	{\mathsf{MonSiMonCat}} && \ni & {-/\mathsf{FinProb}} \\
	{\mathsf{OrdCMon}} && \ni & {?} \\
	{\mathsf{COrdCMon}} & {\mathsf{IcOrdCMon}} & \ni & {(\mathsf{FinProb},\succsim_{01})} & {(H_{0},H_{1})(\mathsf{FinProb})} \\
	{\mathsf{OrdAb}} & {\mathsf{IcOrdAb}} & \ni & G & {(\log\mathbb{Q}_{>0})\times\mathbb{R}} \\
	{\mathsf{OrdVect}_{\mathbb{Q}}} & {\mathsf{IcOrdVect}_{\mathbb{Q}}} & \ni & V & {(\mathbb{Q}\log\mathbb{Q}_{>0})\times\mathbb{R}}
	\arrow[hook, from=2-1, to=1-1]
	\arrow[hook, from=3-2, to=3-1]
	\arrow[hook, from=3-1, to=2-1]
	\arrow[hook, from=5-2, to=4-2]
	\arrow[hook, from=4-2, to=3-2]
	\arrow[hook, from=4-1, to=3-1]
	\arrow[hook, from=5-1, to=4-1]
	\arrow[hook, from=4-2, to=4-1]
	\arrow[hook, from=5-2, to=5-1]
	\arrow[from=1-4, to=2-4]
	\arrow[from=2-4, to=3-4]
	\arrow[from=3-4, to=4-4]
	\arrow[from=4-4, to=5-4]
	\arrow["{(H_0,H_1)}", from=3-4, to=3-5]
	\arrow[from=4-4, to=4-5]
	\arrow[from=5-4, to=5-5]
	\arrow[hook, from=3-5, to=4-5]
	\arrow[hook, from=4-5, to=5-5]
	\arrow["{(H_{0},H_{1})\circ\mathrm{Cod}}", from=1-4, to=3-5]
\end{tikzcd}\]

Before we present the proofs of Theorem \ref{thm:finprob_ordcmon}
and Theorem \eqref{thm:finprob_ic}, we review the concept of majorization
of probability distributions. A probability mass function $Q\in\mathsf{FinProb}$
\emph{majorizes} another probability mass function $P\in\mathsf{FinProb}$,
written as $P\preceq Q$, if $\max_{A:\,|A|\le k}\sum_{x\in A}Q(x)\ge\max_{A:\,|A|\le k}\sum_{x\in A}P(x)$
for all $k\in\mathbb{Z}_{>0}$, i.e., the sum of the largest $k$
entries of $Q$ is larger than or equal to the sum of the largest
$k$ entries of $P$ \cite{marshall1979inequalities}. 

We say that a functor $H:\mathsf{FinProb}\to\mathsf{W}$ (where $\mathsf{W}\in\mathsf{OrdCMon}$)
is \emph{Schur-concave} if $P\preceq Q$ implies $H(P)\ge H(Q)$ \cite{marshall1979inequalities},
i.e., $H$ is an order-preserving map from $(\mathsf{FinProb},\preceq)$
to $(\mathsf{W},\ge)$. Note that Shannon entropy and Hartley entropy
are Schur-concave.

The following result is a generalization of \cite[Lemma 1]{aczel1974shannon}
to arbitrary cancellative ordered commutative monoid.

\medskip{}

\begin{lem}
\label{lem:schur}If $H:\mathsf{FinProb}\to\mathsf{W}$ (where $\mathsf{W}\in\mathsf{COrdCMon}$)
is a baseless $\mathsf{COrdCMon}$-entropy (i.e., it satisfies monotonicity
\eqref{eq:baseless_monotone}, additivity \eqref{eq:baseless_additive}
and subadditivity \eqref{eq:baseless_subadditive}), then $H$ is
Schur-concave.
\end{lem}

\medskip{}

\begin{proof}
Consider distributions $P,Q\in\mathsf{FinProb}$ with disjoint supports,
and $\lambda\in[0,1]$. Let $E$ be an event with $\mathbb{P}(E)=\lambda$.
Let $X,Y$ be two independent random variables with $X\sim P$ and
$Y\sim Q$, independent of $E$. Assume that $\phi$ is a symbol not
in the support of $P$ or $Q$. Define the random variables
\[
\tilde{X}:=\begin{cases}
X & \mathrm{if}\;E,\\
\phi & \mathrm{if}\;E^{c},
\end{cases}\;\;\tilde{Y}:=\begin{cases}
Y & \mathrm{if}\;E,\\
\phi & \mathrm{if}\;E^{c},
\end{cases}
\]
\[
Z:=\begin{cases}
X & \mathrm{if}\;E,\\
Y & \mathrm{if}\;E^{c},
\end{cases}\;\;W:=\begin{cases}
Y & \mathrm{if}\;E,\\
X & \mathrm{if}\;E^{c}.
\end{cases}
\]
For example, $\tilde{X}$ is a random variable with $\tilde{X}=X$
if the event $E$ occurs, and $\tilde{X}=\phi$ when $E$ does not
occur. Write $H(X)$ for the value of $H$ evaluated at the distribution
of $X$. We have
\begin{align}
H(\tilde{X})+H(Y) & \stackrel{(a)}{=}H(\tilde{X},Y)\nonumber \\
 & \stackrel{(b)}{\le}H(Z,\tilde{Y})\nonumber \\
 & \stackrel{(c)}{\le}H(Z)+H(\tilde{Y}),\label{eq:schur_line1}
\end{align}
where (a) is by the additivity property of the baseless entropy $H$
since $\tilde{X}$ and $Y$ are independent, (b) is by the monotonicity
property of $H$ since $(\tilde{X},Y)$ is a function of $(Z,\tilde{Y})$
(we have $(\tilde{X},Y)=(Z,\tilde{Y})$ if $\tilde{Y}\neq\phi$, $(\tilde{X},Y)=(\phi,Z)$
if $\tilde{Y}=\phi$), and (c) is by the subadditivity property of
$H$. Similarly, we have
\begin{equation}
H(\tilde{Y})+H(X)\le H(W)+H(\tilde{X}).\label{eq:schur_line2}
\end{equation}
Adding \eqref{eq:schur_line1} and \eqref{eq:schur_line2}, and cancelling
$H(\tilde{X})$ and $H(\tilde{Y})$ since $\mathsf{W}\in\mathsf{OrdCMon}$
is cancellative, we have
\[
H(X)+H(Y)\le H(Z)+H(W).
\]
Equivalently,
\[
H(P)+H(Q)\le H(\lambda P+(1-\lambda)Q)+H(\lambda Q+(1-\lambda)P),
\]
where $\lambda P+(1-\lambda)Q$ denotes the mixture distribution of
$P$ and $Q$ with weights $\lambda$ and $1-\lambda$. Applying this
on $P=[p_{1},\ldots,p_{n}]$ (written as a probability vector) and
\[
Q:=\left[\frac{\lambda}{1+\lambda}p_{1}+\frac{1}{1+\lambda}p_{2},\,\frac{1}{1+\lambda}p_{1}+\frac{\lambda}{1+\lambda}p_{2},\,p_{3},\ldots,p_{n}\right],
\]
we have
\begin{align*}
H(P)+H(Q) & \le H\left(\left[\frac{2\lambda}{1+\lambda}p_{1}+\frac{1-\lambda}{1+\lambda}p_{2},\,\frac{1-\lambda}{1+\lambda}p_{1}+\frac{2\lambda}{1+\lambda}p_{2},\,p_{3},\ldots,p_{n}\right]\right)\\
 & \;\;\;+H\left(\left[\frac{1}{1+\lambda}p_{1}+\frac{\lambda}{1+\lambda}p_{2},\,\frac{\lambda}{1+\lambda}p_{1}+\frac{1}{1+\lambda}p_{2},\,p_{3},\ldots,p_{n}\right]\right).
\end{align*}
Note that $H$ is invariant under permutation of entries. Cancelling
$H(Q)$ on both sides, we have
\begin{align*}
H(P) & \le H\left(\left[\frac{2\lambda}{1+\lambda}p_{1}+\frac{1-\lambda}{1+\lambda}p_{2},\,\frac{1-\lambda}{1+\lambda}p_{1}+\frac{2\lambda}{1+\lambda}p_{2},\,p_{3},\ldots,p_{n}\right]\right).
\end{align*}
Observe that the right hand side represents a ``Robin Hood transfer''
\cite{marshall1979inequalities}. Since $P\preceq Q$ if and only
if we can obtain $P$ from $Q$ via a finite sequence of Robin Hood
transfers \cite{marshall1979inequalities}, we have $H(P)\ge H(Q)$
whenever $P\preceq Q$.
\end{proof}
\medskip{}

We will also utilize the result in \cite[Theorem 1]{muller2016generalization},
stated using the notations in this paper.

\medskip{}

\begin{thm}
[\cite{muller2016generalization}]\label{thm:h0h1_majorize}If $P,Q\in\mathsf{FinProb}$
satisfies $H_{0}(P)\ge H_{0}(Q)$ and $H_{1}(P)>H_{1}(Q)$, then there
exists $R\in\mathsf{FinProb}$ and $X_{1},X_{2},X_{3}\in N(R)$ satisfying
\[
P\otimes\mathrm{Cod}_{R}(X_{1}\otimes X_{2}\otimes X_{3})\preceq Q\otimes\mathrm{Cod}_{R}(X_{1})\otimes\mathrm{Cod}_{R}(X_{2})\otimes\mathrm{Cod}_{R}(X_{3}).
\]
Recall that $\mathrm{Cod}_{R}(X_{1})$ is the distribution of the
random variable $X_{1}$ in the probability space $R$, $X_{1}\otimes X_{2}\otimes X_{3}\in N(R)$
is the joint random variable ``$(X_{1},X_{2},X_{3})$'', and $P\otimes\mathrm{Cod}_{R}(\cdots)$
is the product distribution.
\end{thm}

\medskip{}

Lemma \ref{lem:schur} and Theorem \ref{thm:h0h1_majorize} imply
the following result.

\medskip{}

\begin{lem}
\label{lem:finprob_ord01}If $H:\mathsf{FinProb}\to\mathsf{W}$ (where
$\mathsf{W}\in\mathsf{COrdCMon}$) is a baseless $\mathsf{COrdCMon}$-entropy,
then $H$ is an order-preserving monoid homomorphism from $(\mathsf{FinProb},\succsim_{01})$
to $(\mathsf{W},\ge)$.
\end{lem}

\medskip{}

\begin{proof}
$H$ being a homomorphism follows from the additivity property of
$H$. It is left to prove that $H$ is order-preserving. Consider
$P,Q\in\mathsf{FinProb}$, $P\succsim_{01}Q$. If $P\cong Q$, then
clearly $H(P)=H(Q)$. If $H_{0}(P)\ge H_{0}(Q)$ and $H_{1}(P)>H_{1}(Q)$,
by Theorem \ref{thm:h0h1_majorize}, there exists $R\in\mathsf{FinProb}$
and $X_{1},X_{2},X_{3}\in N(R)$ satisfying
\[
P\otimes\mathrm{Cod}_{R}(X_{1}\otimes X_{2}\otimes X_{3})\preceq Q\otimes\mathrm{Cod}_{R}(X_{1})\otimes\mathrm{Cod}_{R}(X_{2})\otimes\mathrm{Cod}_{R}(X_{3}).
\]
Lemma \ref{lem:schur} gives
\begin{align*}
 & H(P\otimes\mathrm{Cod}_{R}(X_{1}\otimes X_{2}\otimes X_{3}))\\
 & \ge H(Q\otimes\mathrm{Cod}_{R}(X_{1})\otimes\mathrm{Cod}_{R}(X_{2})\otimes\mathrm{Cod}_{R}(X_{3})).
\end{align*}
By additivity and subadditivity of $H$,
\begin{align*}
 & H(P)+H(\mathrm{Cod}_{R}(X_{1})))+H(\mathrm{Cod}_{R}(X_{2})))+H(\mathrm{Cod}_{R}(X_{3})))\\
 & \ge H(P)+H(\mathrm{Cod}_{R}(X_{1}\otimes X_{2}\otimes X_{3})))\\
 & =H(P\otimes\mathrm{Cod}_{R}(X_{1}\otimes X_{2}\otimes X_{3}))\\
 & \ge H(Q\otimes\mathrm{Cod}_{R}(X_{1})\otimes\mathrm{Cod}_{R}(X_{2})\otimes\mathrm{Cod}_{R}(X_{3}))\\
 & =H(Q)+H(\mathrm{Cod}_{R}(X_{1})))+H(\mathrm{Cod}_{R}(X_{2})))+H(\mathrm{Cod}_{R}(X_{3}))).
\end{align*}
We have $H(P)\ge H(Q)$ after cancelling the terms on both sides.
\end{proof}
\medskip{}

We are ready to prove Theorem \ref{thm:finprob_ordcmon}.

\medskip{}

\begin{proof}
[Proof of Theorem \ref{thm:finprob_ordcmon}]We prove that the universal
baseless $\mathsf{COrdCMon}$-entropy is given by the obvious functor
$\iota:\mathsf{FinProb}\to(\mathsf{FinProb},\succsim_{01})$. First,
we show that $\iota$ is a baseless entropy. For monotonicity \eqref{eq:baseless_monotone},
if there is a morphism $f:P\to Q$, then either $P\cong Q$ ($f$
is an injection and hence a bijection), or $H_{0}(P)>H_{0}(Q)$, $H_{1}(P)>H_{1}(Q)$
($f$ is not injective, resulting in a loss of information). Both
cases give $P\succsim_{01}Q$. Additivity \eqref{eq:baseless_additive}
is immediate. For subadditivity \eqref{eq:baseless_subadditive},
consider random variables $X,Y\in N(P)$. By the subadditivity of
$H_{0}$ and $H_{1}$, we have $H_{0}(\mathrm{Cod}(X))+H_{0}(\mathrm{Cod}(Y))\ge H_{0}(\mathrm{Cod}(X\otimes Y))$,
and the same for $H_{1}$. If $X,Y$ are independent, we have $\mathrm{Cod}(X\otimes Y)\cong\mathrm{Cod}(X)\otimes\mathrm{Cod}(Y)$,
and hence $\mathrm{Cod}(X)\otimes\mathrm{Cod}(Y)\succsim_{01}\mathrm{Cod}(X\otimes Y)$.
If $X,Y$ are not independent, then 
\begin{align*}
 & H_{1}(\mathrm{Cod}(X)\otimes\mathrm{Cod}(Y))\\
 & =H_{1}(\mathrm{Cod}(X))+H_{1}(\mathrm{Cod}(Y))\\
 & >H_{1}(\mathrm{Cod}(X\otimes Y)),
\end{align*}
and we still have $\mathrm{Cod}(X)\otimes\mathrm{Cod}(Y)\succsim_{01}\mathrm{Cod}(X\otimes Y)$.

Next, we show the universal property. Fix any $\mathsf{W}'\in\mathsf{COrdCMon}$
and a baseless $\mathsf{COrdCMon}$-entropy $H':\mathsf{FinProb}\to\mathsf{W}'$.
By Lemma \ref{lem:finprob_ord01}, $H'$ can be regarded as an order-preserving
homomorphism $\tilde{H}':(\mathsf{FinProb},\succsim_{01})\to(\mathsf{W},\ge)$,
i.e., $H'$ can be factorized as the composition
\[
\mathsf{FinProb}\stackrel{\iota}{\to}(\mathsf{FinProb},\succsim_{01})\stackrel{\tilde{H}'}{\to}\mathsf{W}.
\]
To prove that $\iota$ is the universal baseless $\mathsf{COrdCMon}$-entropy,
it is left to show that the $\tilde{H}'$ above is unique, i.e., if
$F:(\mathsf{FinProb},\succsim_{01})\to\mathsf{W}$ satisfies $F\iota=H'$,
then $F=\tilde{H}'$. This is clearly true since $F(P)=(F\iota)(P)=H'(P)=\tilde{H}'(P)$
for $P\in\mathsf{FinProb}$.
\end{proof}
\medskip{}

Finally, we prove Theorem \ref{thm:finprob_ic}.

\medskip{}

\begin{proof}
[Proof of Theorem \ref{thm:finprob_ic}]We prove that the universal
baseless $\mathsf{IcOrdCMon}$-entropy is given by $(H_{0},H_{1}):\mathsf{FinProb}\to(H_{0},H_{1})(\mathsf{FinProb})$.
Note that $(H_{0},H_{1})$ is clearly a baseless entropy, which follows
immediately from the monotonicity, additivity and subadditivity of
$H_{0}$ and $H_{1}$. Next, we show the universal property. Fix any
$\mathsf{W}'\in\mathsf{IcOrdCMon}$ and a baseless $\mathsf{IcOrdCMon}$-entropy
$H':\mathsf{FinProb}\to\mathsf{W}'$. By Lemma \ref{lem:finprob_ord01},
$H'$ can be regarded as an order-preserving homomorphism $(\mathsf{FinProb},\succsim_{01})\to(\mathsf{W},\ge)$.
Let $P,Q\in\mathsf{FinProb}$ satisfy $H_{0}(P)\ge H_{0}(Q)$ and
$H_{1}(P)\ge H_{1}(Q)$. Consider any non-degenerate distribution
$R\in\mathsf{FinProb}$. We have $nH_{0}(P)+H_{0}(R)>nH_{0}(Q)$ and
$nH_{1}(P)+H_{1}(R)>nH_{1}(Q)$ for any $n\in\mathbb{Z}_{>1}$, and
hence $P^{\otimes n}\otimes R\succsim_{01}Q^{\otimes n}$, where $P^{\otimes n}=P\otimes\cdots\otimes P$
($n$ times). Since $H'$ is additive and order-preserving with respect
to $\succsim_{01}$, 
\begin{align*}
nH'(P)+H'(R) & =H'(P^{\otimes n}\otimes R)\\
 & \ge H'(Q^{\otimes n})\\
 & =nH'(Q).
\end{align*}
By the integral closedness of $\mathsf{W}'$, we have $H'(P)\ge H'(Q)$.
This implies that if $H_{0}(P)=H_{0}(Q)$ and $H_{1}(P)=H_{1}(Q)$,
then $H'(P)\ge H'(Q)$ and $H'(Q)\ge H'(P)$, implying $H'(P)=H'(Q)$.
$H'$ can be factorized as the composition
\[
\mathsf{FinProb}\stackrel{(H_{0},H_{1})}{\to}(H_{0},H_{1})(\mathsf{FinProb})\stackrel{F}{\to}\mathsf{W},
\]
where $F(a,b)=H'(P)$ where $P$ satisfies $H_{0}(P)=a$ and $H_{1}(P)=b$.
To check that $F$ is an order-preserving homomorphism, for $(a_{1},b_{1}),(a_{2},b_{2})\in(H_{0},H_{1})(\mathsf{FinProb})$,
letting $P_{i}\in\mathsf{FinProb}$ such that $H_{0}(P_{i})=a_{i}$
and $H_{1}(P_{i})=b_{i}$, we have $F(a_{1},b_{1})+F(a_{2},b_{2})=H'(P_{1})+H'(P_{2})=H'(P_{1}\otimes P_{2})=F(a_{1}+a_{2},b_{1}+b_{2})$.
If $(a_{2},b_{2})\ge(a_{1},b_{1})$, then $H_{0}(P_{2})\ge H_{0}(P_{1})$
and $H_{1}(P_{2})\ge H_{1}(P_{1})$, implying $F(a_{2},b_{2})=H'(P_{2})\ge H'(P_{1})=F(a_{1},b_{1})$
as we have proved before. Such $F$ is clearly unique. This completes
the proof.
\end{proof}
\medskip{}

We remark that the same result can be obtained if we allow non-stricly-positive
probability measures in the definition of $\mathsf{FinProb}$,\footnote{If we allow non-stricly-positive probability measures, we have to
slightly change the tensor product over $N(P)$, where the product
of $X:P\to Q$ and $Y:P\to R$ is given by $Z:P\to S$, where $S$
is a probability space with sample space $\Omega(S)=\Omega(Q)\times\Omega(R)$
(where $\Omega(S)$ denotes the sample space of $S$), with probabilities
$S(a,b)=\sum_{t\in\Omega(P):\,(a,b)=(X(t),Y(t))}P(t)$, and the random
variable $Z$ is simply the pairing $(X,Y)$, i.e., $Z(t)=(X(t),Y(t))$
\cite{baez2011entropy}.} since two probability mass functions $P:A\to[0,1]$ and $Q:B\to[0,1]$
must have the same $\mathsf{IcOrdCMon}$-entropy if they have the
same support $\mathrm{supp}(P)=\mathrm{supp}(Q)$ and they are equal
when restricted to the support, since there is a morphism $P\to Q$
in $\mathsf{FinProb}$ (a measure-preserving function $f:A\to B$
with $f(x)=x$ when $x\in\mathrm{supp}(P)$, and $f(x)$ is arbitrary
when $x\notin\mathrm{supp}(P)$) and vice versa, so $P$ and $Q$
have the same entropy since entropy takes values over an ordered commutative
monoid.

Also, we may instead consider the MonSiMon category $(\mathsf{Prob},\,-/\mathsf{FinProb})$,
where $\mathsf{Prob}$ is the category of (finite/infinite) probability
spaces, and $-/\mathsf{FinProb}$ is the functor mapping $P\in\mathsf{Prob}$
to the comma category $P/\iota$, where $\iota:\mathsf{FinProb}\to\mathsf{Prob}$
is the inclusion functor. This is the ``category of finite random
variables over general probability spaces''. It is straightforward
to check that this MonSiMon category has the same universal $\mathsf{IcOrdCMon}$-entropy
given by the pairing of the Shannon entropy and the Hartley entropy.
Nevertheless, $(\mathsf{Prob},\,-/\mathsf{FinProb})$ does not have
a universal baseless $\mathsf{IcOrdCMon}$-entropy since it is not
in the form $(\mathsf{C},\,-/\mathsf{C})$, and one cannot generally
define entropy over arbitrary probability spaces in $\mathsf{Prob}$.

One might also be interested in characterizing the Shannon entropy
or the Hartley entropy alone, instead of the pairing of them. Hartley
entropy can be obtained by first passing the probability distribution
through a forgetful functor $\mathsf{FinProb}\to\mathrm{Epi}(\mathsf{FinISet})$
mapping a probability distribution to its support, and then obtaining
the universal entropy of $\mathrm{Epi}(\mathsf{FinISet})$ in Section
\ref{sec:set}. For Shannon entropy, it can obtained by first passing
the probability distribution through an inclusion functor $\mathsf{FinProb}\to\mathsf{LProb}_{\rho}$
to the category of $\rho$-th-power-summable discrete distributions,
and then passing it through the universal entropy of $\mathsf{LProb}_{\rho}$
in Section \ref{sec:lrho}. It can also be characterized as the universal
entropy that satisfies an additional condition given by the chain
rule, as described in Section \ref{sec:conditional}.

We remark that the universal $\mathsf{OrdVect}_{\mathbb{Q}}$-entropy
of $(\mathsf{FinProb},\,-/\mathsf{FinProb})$ is guaranteed to exist,
but appears to be intractable. For an example of a baseless $\mathsf{OrdVect}_{\mathbb{Q}}$-entropy,
consider $(H_{1},H_{\alpha}):\mathsf{FinProb}\to(\mathbb{R}^{2},\tilde{\ge})$,
where $H_{\alpha}(P)=\frac{1}{1-\alpha}\log\sum_{x}(P(x))^{\alpha}$
(where $\alpha\neq1$) is the order-$\alpha$ R\'enyi entropy \cite{renyi1961measures},
and $(x_{1},x_{2})\,\tilde{\ge}\,(y_{1},y_{2})$ if $x_{1}>y_{1}$
or $(x_{1},x_{2})=(y_{1},y_{2})$. Since $H_{\alpha}$ satisfies the
additivity property (but not the subadditivity property), we can check
that $(H_{1},H_{\alpha})$ is a baseless $\mathsf{OrdVect}_{\mathbb{Q}}$-entropy.
Therefore, the universal baseless $\mathsf{OrdVect}_{\mathbb{Q}}$-entropy
would have to include $H_{\alpha}$ for all $\alpha$, as well as
all other functions that satisfies the additivity property (e.g. Burg
entropy \cite{burg1967maximum}). Not only is this significantly more
complicated than the universal baseless entropies in Corollary \ref{cor:finprob_ab_vect}
and Theorem \ref{thm:lprob}, this also bypasses the subadditivity
property and include components that only satisfies the additivity
property. This highlights the advantage of considering $\mathsf{IcOrdVect}_{\mathbb{Q}}$
(which disqualifies $(\mathbb{R}^{2},\tilde{\ge})$ since it is not
integrally closed) instead of $\mathsf{OrdVect}_{\mathbb{Q}}$.

\medskip{}

\section{Shannon Entropy for $\rho$-th-Power-Summable Probability Distributions\label{sec:lrho}}

Instead of $\mathsf{FinProb}$, we may consider the space of $\rho$-th-power-summable
probability distributions. Consider the category $\mathsf{LProb}_{\rho}$,
$\rho\ge0$, where each object is a strictly positive probability
mass function $P$ over a finite or countable set $S$ such that $\sum_{x\in S}(P(x))^{\rho}<\infty$,
and morphisms are measure-preserving mappings. When $0\le\rho<1$,
each object in $\mathsf{LProb}_{\rho}$ has a finite Shannon entropy.
By the property of Shannon entropy, we can see that $H_{1}:\mathsf{LProb}_{\rho}\to(\mathbb{R},\le)$
is a baseless entropy for $\rho<1$, though the Hartley entropy $H_{0}:\mathsf{LProb}_{\rho}\to(\mathbb{R},\le)$
is a baseless entropy only for $\rho=0$ (where $\mathsf{LProb}_{\rho}$
is the same as $\mathsf{FinProb}$).

The following result gives the universal $\mathsf{IcOrdCMon}$-entropy
of $(\mathsf{LProb}_{\rho},\,-/\mathsf{LProb}_{\rho})$. In particular,
when $0<\rho<1$, the universal baseless $\mathsf{IcOrdCMon}$-entropy
is simply the Shannon entropy $H_{1}$. As a result, the universal
baseless $\mathsf{IcOrdAb}$ and $\mathsf{IcOrdVect}_{\mathbb{Q}}$-entropies
are also given by the Shannon entropy.

\medskip{}

\begin{thm}
\label{thm:lprob}The universal $\mathsf{IcOrdCMon}$-entropy of $(\mathsf{LProb}_{\rho},\,-/\mathsf{LProb}_{\rho})$
is given by $(!,h)$ where $h$ is
\[
\begin{array}{ll}
(h_{0},h_{1}):\,-/\mathsf{LProb}_{\rho}\Rightarrow\Delta_{(H_{0},H_{1})(\mathsf{FinProb})} & \;\text{if}\;\,\rho=0,\\
h_{1}:\,-/\mathsf{LProb}_{\rho}\Rightarrow\Delta_{\mathbb{R}_{\ge0}} & \;\text{if}\;\,0<\rho<1,\\
z:\,-/\mathsf{LProb}_{\rho}\Rightarrow\Delta_{\{0\}} & \;\text{if}\;\,\rho\ge1,
\end{array}
\]
where $z_{P}=\Delta_{0}$ maps anything to $0$. Equivalently, the
universal baseless $\mathsf{IcOrdCMon}$-entropy is
\[
\begin{array}{ll}
(H_{0},H_{1}):\mathsf{FinProb}\to(H_{0},H_{1})(\mathsf{FinProb}) & \;\text{if}\;\,\rho=0,\\
H_{1}:\mathsf{FinProb}\to\mathbb{R}_{\ge0} & \;\text{if}\;\,0<\rho<1,\\
\Delta_{0}:\mathsf{FinProb}\to\{0\} & \;\text{if}\;\,\rho\ge1.
\end{array}
\]
\end{thm}

\medskip{}

The universal entropies of $(\mathsf{LProb}_{\rho},\,-/\mathsf{LProb}_{\rho})$
for $0<\rho<1$ can be summarized by the following diagram, where
the categories on the left are the ``categories of codomains'' of
entropies (see Section \ref{subsec:cat_of_codomains}), and the categories
on the right are elements of the corresponding categories on the left.
The arrow $H_{1}\circ\mathrm{Cod}$ is the universal $\mathsf{IcOrdCMon}$-entropy
of $(\mathsf{LProb}_{\rho},\,-/\mathsf{LProb}_{\rho})$ in Theorem
\ref{thm:lprob}.\footnote{This notation is justified by the one-to-one correspondence $Uh=UH\circ\mathrm{Cod}$
in Proposition \ref{prop:baseless}, where we omit the forgetful functor
$U$ for brevity.} The ``$?$'' are the codomains of the universal $\mathsf{OrdCMon}$,
$\mathsf{COrdCMon}$, $\mathsf{OrdAb}$ and $\mathsf{OrdVect}_{\mathbb{Q}}$
entropies of $(\mathsf{LProb}_{\rho},\,-/\mathsf{LProb}_{\rho})$,
which are guaranteed to exist but appear to be intractable and uninteresting,
and hence omitted. This also highlights the usefulness of the integral
closedness property in the three ``$\mathsf{Ic}\cdots$'' categories
of codomains, which greatly simplifies the entropy.\\
\[\begin{tikzcd}[column sep=scriptsize]
	{\mathsf{MonSiMonCat}} && \ni & {-/\mathsf{LProb}_{\rho}} \\
	{\mathsf{OrdCMon}} && \ni & {?} \\
	{\mathsf{COrdCMon}} & {\mathsf{IcOrdCMon}} & \ni & {?} & {\mathbb{R}_{\ge0}} \\
	{\mathsf{OrdAb}} & {\mathsf{IcOrdAb}} & \ni & {?} & {\mathbb{R}} \\
	{\mathsf{OrdVect}_{\mathbb{Q}}} & {\mathsf{IcOrdVect}_{\mathbb{Q}}} & \ni & {?} & {\mathbb{R}}
	\arrow[hook, from=2-1, to=1-1]
	\arrow[hook, from=3-2, to=3-1]
	\arrow[hook, from=3-1, to=2-1]
	\arrow[hook, from=5-2, to=4-2]
	\arrow[hook, from=4-2, to=3-2]
	\arrow[hook, from=4-1, to=3-1]
	\arrow[hook, from=5-1, to=4-1]
	\arrow[hook, from=4-2, to=4-1]
	\arrow[hook, from=5-2, to=5-1]
	\arrow[from=1-4, to=2-4]
	\arrow[from=2-4, to=3-4]
	\arrow[from=4-4, to=5-4]
	\arrow[from=3-4, to=3-5]
	\arrow[from=4-4, to=4-5]
	\arrow[from=5-4, to=5-5]
	\arrow[hook, from=3-5, to=4-5]
	\arrow[hook, from=4-5, to=5-5]
	\arrow[from=3-4, to=4-4]
	\arrow["{H_1 \circ\mathrm{Cod}}", from=1-4, to=3-5]
\end{tikzcd}\]
\begin{proof}
The case $\rho=0$ is proved in Theorem \ref{thm:finprob_ic}. Consider
the case $0<\rho<1$. Consider any baseless $\mathsf{IcOrdCMon}$-entropy
$H:\mathsf{LProb}_{\rho}\to\mathsf{W}$. Applying Lemma \ref{lem:finprob_ord01}
on the restriction of $H$ to $\mathsf{FinProb}$, we know that for
$P,Q\in\mathsf{FinProb}$, $P\succsim_{01}Q$ implies $H(P)\ge H(Q)$.
Consider $P,Q\in\mathsf{FinProb}$ where $H_{1}(P)\ge H_{1}(Q)$.
Let $X\sim\mathrm{Geom}(1/2)$ be a geometric random variable with
probability mass function (pmf) $P_{X}\in\mathsf{LProb}_{\rho}$,
and let $Y_{n}=\min\{X,n\}$ with pmf $P_{Y_{n}}\in\mathsf{FinProb}$.
Fix $k\in\mathbb{N}$. We have $H_{1}(P^{\otimes k}\otimes P_{Y_{n}})=kH_{1}(P)+H_{1}(P_{Y_{n}})>H_{1}(Q^{\otimes k})$,
and $H_{0}(P^{\otimes k}\otimes P_{Y_{n}})=kH_{0}(P)+n>H_{0}(Q^{\otimes k})$
for $n$ large enough. Hence for $n$ large enough, we have $P^{\otimes k}\otimes P_{Y_{n}}\succsim_{01}Q^{\otimes k}$,
and hence (writing $H(X)=H(P_{X})$)
\begin{align*}
kH(P)+H(X) & \ge kH(P)+H(Y_{n})\\
 & \ge H(P^{\otimes k}\otimes P_{Y_{n}})\\
 & \ge H(Q^{\otimes k})\\
 & =kH(Q).
\end{align*}
By the integral closedness of $\mathsf{W}$, we have $H(P)\ge H(Q)$.
Hence, when restricted on $\mathsf{FinProb}$, $H$ is order-preserving
with respect to $H_{1}$ and is determined by $H_{1}$, i.e., there
exists an order-preserving function $F:\mathbb{R}_{\ge0}\to\mathsf{W}$
such that $H(P)=F(H_{1}(P))$ for $P\in\mathsf{FinProb}$. Since $H$
and $H_{1}$ are additive, $F$ is also a monoid homomorphism.

It is left to extend this to all $P\in\mathsf{LProb}_{\rho}$. Consider
any nondegenerate $P\in\mathsf{LProb}_{\rho}$ not in $\mathsf{FinProb}$.
Without loss of generality, assume $P$ is a probability mass function
over $\mathbb{N}$. Let $a_{n}\in\mathbb{N}$ be such that $\sum_{x=a_{n}+1}^{\infty}(P(x))^{\rho}\le2^{-n}$.
Let $X_{1},X_{2},\ldots$ be an i.i.d. sequence of random variables
following $P$, and $Y_{n}=\min\{X_{n},a_{n}\}$ (with pmf $P_{Y_{n}}\in\mathsf{FinProb}$),
$Z_{n}=\max\{X_{n},a_{n}\}$. Since $Y_{n}$ is a function of $X_{n}$,
we have $H(P)=H(X_{n})\ge H(Y_{n})=F(H_{1}(Y_{n}))$. By $\lim_{n\to\infty}H_{1}(P_{Y_{n}})=H_{1}(P)$,
we have $H(P)\ge F(t)$ for every $0\le t<H_{1}(P)$, implying 
\begin{align*}
kH(P)+F(1) & \ge kF(H_{1}(P)-k^{-1})+F(1)\\
 & =kF(H_{1}(P)-k^{-1})+kF(k^{-1})\\
 & =kF(H_{1}(P))
\end{align*}
for $k\in\mathbb{N}$ large enough such that $H_{1}(P)-k^{-1}>0$.
By integral closedness, $H(P)\ge F(H_{1}(P))$.

We now prove the other direction ``$H(P)\le F(H_{1}(P))$''. Since
$X_{n}$ is a function of $(Y_{n},Z_{n})$, by the subadditivity property
of entropy, 
\begin{align}
H(P) & \le H(Y_{n})+H(Z_{n})\nonumber \\
 & =F(H_{1}(Y_{n}))+H(Z_{n})\nonumber \\
 & \le F(H_{1}(P))+H(Z_{n}).\label{eq:hP_upperbd}
\end{align}
Let $S=\{n\in\mathbb{N}:\,X_{n}>a_{n}\}$. By the Borel-Cantelli lemma,
since 
\begin{align}
\sum_{n=1}^{\infty}\mathbb{P}(X_{n}>a_{n}) & =\sum_{n=1}^{\infty}\sum_{x=a_{n}+1}^{\infty}P(x)\nonumber \\
 & \le\sum_{n=1}^{\infty}\sum_{x=a_{n}+1}^{\infty}(P(x))^{\rho}\nonumber \\
 & \le\sum_{n=1}^{\infty}2^{-n}\nonumber \\
 & =1,\label{eq:borel}
\end{align}
$S$ is almost surely finite. Consider the random variable $W=(S,(X_{n})_{n\in S})$
which lies in the disjoint union $\bigcup_{s\subseteq\mathbb{N}\;\mathrm{finite}}\mathbb{N}^{s}$,
which is a countable set. We have (note that $s$ sums over finite
subsets of $\mathbb{N}$ below)
\begin{align}
 & \sum_{(s,(x_{n})_{n\in s}):\,\forall n.\,x_{n}>a_{n}}(\mathbb{P}(W=(s,(x_{n})_{n\in s})))^{\rho}\nonumber \\
 & =\sum_{(s,(x_{n})_{n\in s}):\,\forall n.\,x_{n}>a_{n}}\Big(\prod_{n\in s}(\mathbb{P}(X_{n}=x_{n}))^{\rho}\Big)\Big(\prod_{n\in\mathbb{N}\backslash s}(\mathbb{P}(X_{n}\le a_{n}))^{\rho}\Big)\nonumber \\
 & \le\sum_{(s,(x_{n})_{n\in s}):\,\forall n.\,x_{n}>a_{n}}\,\prod_{n\in s}(P(x_{n}))^{\rho}\nonumber \\
 & =\sum_{s}\prod_{n\in s}\sum_{x=a_{n}+1}^{\infty}(P(x))^{\rho}\nonumber \\
 & \le\sum_{s}\prod_{n\in s}2^{-n}\nonumber \\
 & =\prod_{n=1}^{\infty}(1+2^{-n})\nonumber \\
 & =\exp\left(\sum_{n=1}^{\infty}\log(1+2^{-n})\right)\nonumber \\
 & <e.\label{eq:pw_in}
\end{align}
Hence $P_{W}\in\mathsf{LProb}_{\rho}$. Note that $Z_{n}$ is a function
of $W$ since $Z_{n}=X_{n}$ if $n\in S$, and $Z_{n}=a_{n}$ if $n\notin S$.
Hence $(Z_{1},\ldots,Z_{n})$ is a function of $W$, and
\begin{align}
H(W) & \ge H(Z_{1},\ldots,Z_{n})=\sum_{i=1}^{n}H(Z_{i}),\label{eq:HW_sumHZ}
\end{align}
where the last equality is due to the additivity property of entropy
since $Z_{1},\ldots,Z_{n}$ are mutually independent. Combining this
with \eqref{eq:hP_upperbd}, for every $n\in\mathbb{N}$,
\begin{align*}
nH(P) & \le\sum_{i=1}^{n}\left(F(H_{1}(P))+H(Z_{i})\right)\\
 & =nF(H_{1}(P))+\sum_{i=1}^{n}H(Z_{i})\\
 & \le nF(H_{1}(P))+H(W),
\end{align*}
and hence $H(P)\le F(H_{1}(P))$ by integral closedness. Therefore,
we can conclude that $H(P)=F(H_{1}(P))$, and $H_{1}$ is the universal
baseless $\mathsf{IcOrdCMon}$-entropy.

Finally, consider the case $\rho=1$ (note that any $\rho\ge1$ is
the same as $\rho=1$ since every probablity mass function satisfies
$\sum_{x\in S}(P(x))^{\rho}<\infty$ for $\rho\ge1$). Consider any
baseless $\mathsf{IcOrdCMon}$-entropy $H:\mathsf{LProb}_{1}\to\mathsf{W}$.
Applying Lemma \ref{lem:finprob_ord01} on the restriction of $H$
to $\mathsf{FinProb}$, we know that for $P,Q\in\mathsf{FinProb}$,
$P\succsim_{01}Q$ implies $H(P)\ge H(Q)$. For any $P\in\mathsf{FinProb}$,
letting $R\in\mathsf{LProb}_{1}$ with $H_{1}(R)=\infty$ (e.g., $R(x)\propto1/(x\log^{2}(x+1))$
for $x\in\mathbb{N}$), we have $R\succsim_{01}P^{\otimes n}$ for
every $n\in\mathbb{N}$, and hence $H(R)\ge nH(P)$, and $H(P)=0$
due to integral closedness. To extend this to $P\in\mathsf{LProb}_{1}$,
applying the same construction as in \eqref{eq:hP_upperbd} on $\rho=1$,
we have
\begin{align*}
H(P) & \le H(Y_{n})+H(Z_{n})=H(Z_{n}),
\end{align*}
and \eqref{eq:HW_sumHZ} gives
\[
H(W)\ge\sum_{i=1}^{n}H(Z_{i}).
\]
Hence 
\[
nH(P)\le\sum_{i=1}^{n}H(Z_{i})\le H(W),
\]
and $H(P)=0$ due to integral closedness. Therefore, $\Delta_{0}:\mathsf{FinProb}\to\{0\}$
is the universal baseless $\mathsf{IcOrdCMon}$-entropy.
\end{proof}
\medskip{}

We can also consider the category $\mathsf{HProb}$, where each object
is a strictly positive probability mass function $P$ over a finite
or countable set $S$ with finite Shannon entropy $H_{1}(P)<\infty$,
and morphisms are measure-preserving mappings. It includes $\mathsf{LProb}_{\rho}$
as a subcategory for $0\le\rho<1$. Its universal entropy is, unsurprisingly,
the Shannon entropy.

\medskip{}

\begin{thm}
\label{thm:finarvect-2-1}The universal $\mathsf{IcOrdCMon}$-entropy
of $(\mathsf{HProb},\,-/\mathsf{HProb})$ is given by $(!,h_{1})$.
Equivalently, the universal baseless $\mathsf{IcOrdCMon}$-entropy
is $H_{1}$.
\end{thm}

\begin{proof}
The proof follows from the same arguments as the proof of Theorem
\ref{thm:lprob}, except that here we define $a_{n}\in\mathbb{N}$
be such that $\sum_{x=a_{n}+1}^{\infty}P(x)\le2^{-n}$ and $H_{1}(P_{Z_{n}})\le2^{-n}$,
where $X_{1},X_{2},\ldots$ is an i.i.d. sequence of random variables
following $P$, $Y_{n}=\min\{X_{n},a_{n}\}$, $Z_{n}=\max\{X_{n},a_{n}\}$.
We can still invoke the Borel-Cantelli lemma \eqref{eq:borel} since
$\sum_{x=a_{n}+1}^{\infty}P(x)\le2^{-n}$. Instead of \eqref{eq:pw_in},
we use the fact that $H_{1}(P_{W})=\sum_{n=1}^{\infty}H_{1}(P_{Z_{n}})\le1$,
so $P_{W}\in\mathsf{HProb}$. The rest of the proof is the same as
Theorem \ref{thm:lprob}.
\end{proof}
\medskip{}

\section{Universal Entropy of Various Categories\label{sec:examples}}

In this section, we derive the universal entropies of various categories.
While these results are technically simple, by relating these categories
with $\mathsf{FinProb}$, we obtain insights on different contexts
from which information and entropy arise. We will see in Section \ref{sec:global}
that these universal entropies can be linked together via functors
between these categories in a natural manner. 

\medskip{}

\subsection{Universal Entropy of Sets\label{sec:set}}

Consider the monoidal category $\mathrm{Epi}(\mathsf{FinISet})$ consisting
of finite inhabited (nonempty) sets as objects, and epimorphisms between
them as morphisms. The tensor product is the cartesian product, and
the under category $N(A)=A/\mathrm{Epi}(\mathsf{FinISet})$ is equipped
with a tensor product, where for $f:A\to B$, $g:A\to C$, $f\otimes g:A\to\mathrm{im}((f,g))$
is the restriction of the codomain of $(f,g):A\to B\times C$ to its
image. Note that $f\otimes g$ is also the categorical product over
$A/\mathrm{Epi}(\mathsf{FinSet})$. It is straightforward to verify
that $-/\mathrm{Epi}(\mathsf{FinISet}):\mathrm{Epi}(\mathsf{FinISet})^{\mathrm{op}}\to\mathsf{MonCat}_{s}$
is a lax monoidal functor. We now find the universal baseless entropies
of $\mathrm{Epi}(\mathsf{FinISet})$, which is simply given by the
logarithm of the cardinality, i.e., the Hartley entropy of the set
\cite{hartley1928transmission}.

\medskip{}

\begin{prop}
\label{prop:finset}For the MonSiMon category $(\mathrm{Epi}(\mathsf{FinISet}),\,-/\mathrm{Epi}(\mathsf{FinISet}))$:
\begin{enumerate}
\item The universal baseless $\mathsf{OrdCMon}$-entropy is given by $\log|\cdot|:\mathrm{Epi}(\mathsf{FinISet})\to\log\mathbb{N}$,
where $\log\mathbb{N}=\{\log n:\,n\in\mathbb{N}\}$. Same for $\mathsf{COrdCMon}$
and $\mathsf{IcOrdCMon}$.
\item The universal baseless $\mathsf{OrdAb}$-entropy is given by $\log|\cdot|:\mathrm{Epi}(\mathsf{FinISet})\to\log\mathbb{Q}_{>0}$.
Same for $\mathsf{IcOrdAb}$.
\item The universal baseless $\mathsf{OrdVect}_{\mathbb{Q}}$-entropy is
given by $\log|\cdot|:\mathrm{Epi}(\mathsf{FinISet})\to\mathbb{Q}\log\mathbb{Q}_{>0}$.
Same for $\mathsf{IcOrdVect}_{\mathbb{Q}}$.
\end{enumerate}
\end{prop}

\medskip{}

These universal entropies can be summarized by the following diagram,
where the categories on the left are the ``categories of codomains''
of entropies (see Section \ref{subsec:cat_of_codomains}), and the
categories on the right are elements of the corresponding categories
on the left. The arrow $(\log|\cdot|)\circ\mathrm{Cod}$ is the universal
$\mathsf{OrdCMon}$-entropy of $(\mathrm{Epi}(\mathsf{FinISet}),-/\mathrm{Epi}(\mathsf{FinISet}))$
in Proposition \ref{prop:finset}.\footnote{This notation is justified by the one-to-one correspondence $Uh=UH\circ\mathrm{Cod}$
in Proposition \ref{prop:baseless}, where we omit the forgetful functor
$U$ for brevity.}

\[\begin{tikzcd}
	{\mathsf{MonSiMonCat}} && \ni & {-/\mathrm{Epi}(\mathsf{FinISet})} \\
	{\mathsf{OrdCMon}} && \ni & {\log\mathbb{N}} \\
	{\mathsf{COrdCMon}} & {\mathsf{IcOrdCMon}} & \ni & {\log\mathbb{N}} & {\log\mathbb{N}} \\
	{\mathsf{OrdAb}} & {\mathsf{IcOrdAb}} & \ni & {\log\mathbb{Q}_{>0}} & {\log\mathbb{Q}_{>0}} \\
	{\mathsf{OrdVect}_{\mathbb{Q}}} & {\mathsf{IcOrdVect}_{\mathbb{Q}}} & \ni & {\mathbb{Q}\log\mathbb{Q}_{>0}} & {\mathbb{Q}\log\mathbb{Q}_{>0}}
	\arrow[hook, from=2-1, to=1-1]
	\arrow[hook, from=3-2, to=3-1]
	\arrow[hook, from=3-1, to=2-1]
	\arrow[hook, from=5-2, to=4-2]
	\arrow[hook, from=4-2, to=3-2]
	\arrow[hook, from=4-1, to=3-1]
	\arrow[hook, from=5-1, to=4-1]
	\arrow[hook, from=4-2, to=4-1]
	\arrow[hook, from=5-2, to=5-1]
	\arrow["{(\log|\cdot|)\circ\mathrm{Cod}}", from=1-4, to=2-4]
	\arrow[Rightarrow, no head, from=2-4, to=3-4]
	\arrow[hook, from=3-4, to=4-4]
	\arrow[hook, from=4-4, to=5-4]
	\arrow[Rightarrow, no head, from=3-4, to=3-5]
	\arrow[Rightarrow, no head, from=4-4, to=4-5]
	\arrow[Rightarrow, no head, from=5-4, to=5-5]
	\arrow[hook, from=3-5, to=4-5]
	\arrow[hook, from=4-5, to=5-5]
\end{tikzcd}\]\medskip{}

\begin{proof}
We prove the universal baseless $\mathsf{OrdCMon}$-entropy result,
from which the other results follow directly. Define the functor $H:\mathrm{Epi}(\mathsf{FinISet})\to\log\mathbb{N}$,
$X\mapsto\log|X|$. Note that $H$ is a baseless $\mathsf{OrdCMon}$-entropy
since $|\mathrm{im}((f,g))|\le|B|\cdot|C|$ for $f:A\to B$, $g:A\to C$.
We now show the universal property of $H$. Fix any $\mathsf{W}'\in\mathsf{OrdCMon}$
and strong monoidal functor $H':\mathrm{Epi}(\mathsf{FinISet})\to\mathsf{W}'$.
Since $H'$ sends isomorphic objects to isomorphic objects, and $\mathsf{W}'$
is skeletal, $H'(X)$ only depends on $|X|$, and we can let $H'(X)=F(|X|)$
where $F:\mathbb{N}\to\mathsf{W}'$ is determined uniquely from $H'$.
Since $H'$ is a functor, $H'(X)\ge H'(Y)$ whenever $|X|\ge|Y|$
since there is an epimorphism $X\to Y$, and hence $F$ is order-preserving.
Since $H'$ is strongly monoidal, $H'(X\times Y)=H'(X)+H'(Y)$, and
$F(xy)=F(x)+F(y)$ for $x,y\in\mathbb{N}$. Hence $F$, as a function
$\log\mathbb{N}\to\mathsf{W}'$, is a morphism in $\mathsf{OrdCMon}$.
\end{proof}
\medskip{}

Recall that we characterized the pairing of the Shannon and Hartley
entropies as the universal $\mathsf{IcOrdAb}$-entropy of $\mathsf{FinProb}$
in Section \ref{sec:shannon}. If we want to obtain the Hartley entropy
alone, we can consider the ``support functor'' $\mathrm{Supp}:\mathsf{FinProb}\to\mathrm{Epi}(\mathsf{FinISet})$,
which sends a probability distribution to its support set, forgetting
the probabilities. This gives a MonSiMon functor $(\mathrm{Supp},\,\gamma)$
from $(\mathsf{FinProb},-/\mathsf{FinProb})$ to $(\mathrm{Epi}(\mathsf{FinISet}),-/\mathrm{Epi}(\mathsf{FinISet}))$,
where $\gamma:-/\mathsf{FinProb}\Rightarrow\mathrm{Supp}^{\mathrm{op}}(-)/\mathrm{Epi}(\mathsf{FinISet})$,
$\gamma_{P}(X)=\mathrm{Supp}(X)$ for $X:P\to Q$.  Then the Hartley
entropy $H_{0}\circ\mathrm{Cod}$ is given by the diagonal arrow in
the following commutative diagram in $\mathsf{MonSiMonCat}$, where
the two horizontal arrows are universal $\mathsf{IcOrdAb}$-entropies
of $-/\mathsf{FinProb}$ and $-/\mathrm{Epi}(\mathsf{FinISet})$,
respectively. Commutativity follows from the naturality of $\phi$
in Definition \ref{def:gcm_entropy}.\\
\[\begin{tikzcd}
	{-/\mathsf{FinProb}} && {(\log\mathbb{Q}_{>0}) \times \mathbb{R}} \\
	\\
	{-/\mathrm{Epi}(\mathsf{FinISet})} && {\log\mathbb{Q}_{>0}}
	\arrow["{(\mathrm{Supp},\,\gamma)}"', from=1-1, to=3-1]
	\arrow["{\pi_1}", from=1-3, to=3-3]
	\arrow["{(H_{0},H_{1})\circ\mathrm{Cod}}", from=1-1, to=1-3]
	\arrow["{(\log|\cdot|)\circ\mathrm{Cod}}", from=3-1, to=3-3]
	\arrow["{H_{0}\circ\mathrm{Cod}}"{description}, from=1-1, to=3-3]
\end{tikzcd}\]

\medskip{}

\subsection{Universal Entropy of Vector Spaces and Gaussian Distributions\label{sec:vec}}

Consider the category $\mathrm{Epi}(\mathsf{FinVect}_{\mathbb{F}})$
consisting of finite dimensional vector spaces over the field $\mathbb{F}$
as objects, and epimorphisms between them as morphisms. The tensor
product is the cartesian product, and the under category $N(A)=A/\mathrm{Epi}(\mathsf{FinVect}_{\mathbb{F}})$
is equipped with a tensor product, where for $f:A\to B$, $g:A\to C$,
$f\otimes g:A\to\mathrm{im}((f,g))$ is the restriction of the codomain
of $(f,g):A\to B\times C$ to its image. Note that $f\otimes g$ is
also the categorical product over $A/\mathrm{Epi}(\mathsf{FinVect}_{\mathbb{F}})$.
We now find the universal entropy of $\mathrm{Epi}(\mathsf{FinVect}_{\mathbb{F}})$,
which is simply given by the dimension.

\medskip{}

\begin{prop}
\label{prop:vect}For the MonSiMon category $(\mathrm{Epi}(\mathsf{FinVect}_{\mathbb{F}}),\,-/\mathrm{Epi}(\mathsf{FinVect}_{\mathbb{F}}))$:
\begin{enumerate}
\item The universal baseless $\mathsf{OrdCMon}$-entropy is given by $\mathrm{dim}:\mathrm{Epi}(\mathsf{FinVect}_{\mathbb{F}})\to\mathbb{Z}_{\ge0}$.
Same for $\mathsf{COrdCMon}$ and $\mathsf{IcOrdCMon}$.
\item The universal baseless $\mathsf{OrdAb}$-entropy is given by $\mathrm{dim}:\mathrm{Epi}(\mathsf{FinVect}_{\mathbb{F}})\to\mathbb{Z}$.
Same for $\mathsf{IcOrdAb}$.
\item The universal baseless $\mathsf{OrdVect}_{\mathbb{Q}}$-entropy is
given by $\mathrm{dim}:\mathrm{Epi}(\mathsf{FinVect}_{\mathbb{F}})\to\mathbb{Q}$.
Same for $\mathsf{IcOrdVect}_{\mathbb{Q}}$.
\end{enumerate}
\end{prop}

\medskip{}

These universal entropies can be summarized by the following diagram.
Refer to Section \ref{sec:set} for its explanation.

\[\begin{tikzcd}
	{\mathsf{MonSiMonCat}} && \ni & {-/\mathrm{Epi}(\mathsf{FinVect}_{\mathbb{F}})} \\
	{\mathsf{OrdCMon}} && \ni & {\mathbb{Z}_{\ge0}} \\
	{\mathsf{COrdCMon}} & {\mathsf{IcOrdCMon}} & \ni & {\mathbb{Z}_{\ge0}} & {\mathbb{Z}_{\ge0}} \\
	{\mathsf{OrdAb}} & {\mathsf{IcOrdAb}} & \ni & {\mathbb{Z}} & {\mathbb{Z}} \\
	{\mathsf{OrdVect}_{\mathbb{Q}}} & {\mathsf{IcOrdVect}_{\mathbb{Q}}} & \ni & {\mathbb{Q}} & {\mathbb{Q}}
	\arrow[hook, from=2-1, to=1-1]
	\arrow[hook, from=3-2, to=3-1]
	\arrow[hook, from=3-1, to=2-1]
	\arrow[hook, from=5-2, to=4-2]
	\arrow[hook, from=4-2, to=3-2]
	\arrow[hook, from=4-1, to=3-1]
	\arrow[hook, from=5-1, to=4-1]
	\arrow[hook, from=4-2, to=4-1]
	\arrow[hook, from=5-2, to=5-1]
	\arrow["{\mathrm{dim}\circ\mathrm{Cod}}", from=1-4, to=2-4]
	\arrow[Rightarrow, no head, from=2-4, to=3-4]
	\arrow[hook, from=3-4, to=4-4]
	\arrow[hook, from=4-4, to=5-4]
	\arrow[Rightarrow, no head, from=3-4, to=3-5]
	\arrow[Rightarrow, no head, from=4-4, to=4-5]
	\arrow[Rightarrow, no head, from=5-4, to=5-5]
	\arrow[hook, from=3-5, to=4-5]
	\arrow[hook, from=4-5, to=5-5]
\end{tikzcd}\]

\medskip{}

\begin{proof}
We prove the universal baseless $\mathsf{OrdCMon}$-entropy result,
from which the other results follow directly. Fix any $\mathsf{W}'\in\mathsf{OrdCMon}$
and strong monoidal functor $H':\mathrm{Epi}(\mathsf{FinVect}_{\mathbb{F}})\to\mathsf{W}'$.
Since $H'$ sends isomorphic objects to isomorphic objects, and $\mathsf{W}'$
is skeletal, $H'(X)$ only depends on $\mathrm{dim}(X)$. Since $H'$
is a functor, $H'(X)\ge H'(Y)$ whenever $\mathrm{dim}(X)\ge\mathrm{dim}(Y)$
since there is an epimorphism $X\to Y$. Since $H'$ is strongly monoidal,
$H'(X\times Y)=H'(X)+H'(Y)$. These imply that $H'(X)=\mathrm{dim}(X)H'(\mathbb{F})$,
and hence $H'=tH$ where $t:\mathbb{Z}_{\ge0}\to\mathsf{W}'$, $x\mapsto xH'(\mathbb{F})$.
Such $t$ is clearly unique.
\end{proof}
\medskip{}

The ``entropy'' or ``amount of information'' of a vector space
being given by its dimension is unsurprising, and is a common observation
in linear codes in coding theory and information theory. More precisely,
if $\mathbb{F}$ is a finite field, we can embed the MonSiMon category
$-/\mathrm{Epi}(\mathsf{FinVect}_{\mathbb{F}})$ into the MonSiMon
category $-/\mathsf{FinProb}$ by mapping a vector space $A$ to the
uniform distribution over $A$, which has an entropy $\log|A|=(\dim A)\log|\mathbb{F}|$,
proportional to $\dim A$. This gives the following commutative diagram
in $\mathsf{MonSiMonCat}$, where the vertical arrow on the left is
the aforementioned embedding, and the two horizontal arrows $\mathrm{dim}\circ\mathrm{Cod}$
and $(H_{0},H_{1})\circ\mathrm{Cod}$ are universal $\mathsf{IcOrdAb}$-entropies
of $-/\mathrm{Epi}(\mathsf{FinVect}_{\mathbb{F}})$ and $-/\mathsf{FinProb}$,
respectively.\footnote{This notation is justified by the one-to-one correspondence $Uh=UH\circ\mathrm{Cod}$
in Proposition \ref{prop:baseless}, where we omit the forgetful functor
$U$ for brevity.} Commutativity follows from the naturality of $\phi$ in Definition
\ref{def:gcm_entropy}.\\
\[\begin{tikzcd}
	{-/\mathrm{Epi}(\mathsf{FinVect}_{\mathbb{F}})} && {\mathbb{Z}} \\
	{-/\mathsf{FinProb}} && {(\log\mathbb{Q}_{>0})\times\mathbb{R}}
	\arrow["{\mathrm{dim}\circ\mathrm{Cod}}", from=1-1, to=1-3]
	\arrow["{(H_{0},H_{1})\circ\mathrm{Cod}}", from=2-1, to=2-3]
	\arrow[hook', from=1-1, to=2-1]
	\arrow["{x\mapsto(x \log|\mathbb{F}|,x \log|\mathbb{F}|)}", from=1-3, to=2-3]
\end{tikzcd}\]

Also, we can consider the category of multivariate Gaussian distributions,
denoted as $\mathsf{Gauss}$. An object is a multivariate Gaussian
distribution that is supported over the whole ambient space, i.e.,
it is a pair $(\mu,\Sigma)$, where $\mu\in\mathbb{R}^{n}$ is the
mean vector, and $\Sigma\in\mathbb{R}^{n\times n}$ is the covariance
matrix (a positive definite matrix), where $n\ge0$ is its information
dimension \cite{renyi1959dimension}. A morphism is a measure-preserving
affine map, i.e., a morphism $(\mu_{1},\Sigma_{1})\to(\mu_{2},\Sigma_{2})$
is a pair $(A,b)$ such that if $x$ follows $\mathrm{Gaussian}(\mu_{1},\Sigma_{1})$,
then $Ax+b$ follows $\mathrm{Gaussian}(\mu_{2},\Sigma_{2})$ (i.e.,
$\mu_{2}=A\mu_{1}+b$ and $\Sigma_{2}=A\Sigma_{1}A^{T}$). The tensor
product is the product distribution, i.e., $(\mu_{1},\Sigma_{1})\otimes(\mu_{2},\Sigma_{2})=\left(\left[\begin{array}{c}
\mu_{1}\\
\mu_{2}
\end{array}\right],\,\left[\begin{array}{cc}
\Sigma_{1}\\
 & \Sigma_{2}
\end{array}\right]\right)$. The under category $N((\mu,\Sigma))=(\mu,\Sigma)/\mathsf{Gauss}$,
which is the category of Gaussian random variables formed by linear
functions of the Gaussian probability space, is equipped with a tensor
product given by joint random variable, i.e., the tensor product between
$(A_{1},b_{1}):(\mu,\Sigma)\to(\mu_{1},\Sigma_{1})$ and $(A_{2},b_{2}):(\mu,\Sigma)\to(\mu_{2},\Sigma_{2})$
is given by the affine map $x\mapsto\left[\begin{array}{c}
A_{1}x+b_{1}\\
A_{2}x+b_{2}
\end{array}\right]$ restricted to its range so that it is surjective, with a codomain
given by the pushforward measure of $\mathrm{Gaussian}(\mu,\Sigma)$
by this affine map.

The universal entropy of $\mathsf{Gauss}$ is again given by the dimension
of its support, which is its information dimension \cite{renyi1959dimension}.
This is because any two Gaussian distributions with the same dimension
are isomorphic in $\mathsf{Gauss}$. The proof is the same as Proposition
\ref{prop:vect} and is omitted.

\medskip{}

\begin{prop}
\label{prop:vect-1}For the MonSiMon category $(\mathsf{Gauss},\,-/\mathsf{Gauss})$:
\begin{enumerate}
\item The universal baseless $\mathsf{OrdCMon}$-entropy is given by $\mathrm{dim}:\mathsf{Gauss}\to\mathbb{Z}_{\ge0}$.
Same for $\mathsf{COrdCMon}$ and $\mathsf{IcOrdCMon}$.
\item The universal baseless $\mathsf{OrdAb}$-entropy is given by $\mathrm{dim}:\mathsf{Gauss}\to\mathbb{Z}$.
Same for $\mathsf{IcOrdAb}$.
\item The universal baseless $\mathsf{OrdVect}_{\mathbb{Q}}$-entropy is
given by $\mathrm{dim}:\mathsf{Gauss}\to\mathbb{Q}$. Same for $\mathsf{IcOrdVect}_{\mathbb{Q}}$.
\end{enumerate}
\end{prop}

\medskip{}

One might expect the entropy to be given by the differential entropy,
which is often the meaning of the term ``entropy'' when applied
on continuous random variables. This is untrue for $\mathsf{Gauss}$
since we have isomorphic distributions with different differential
entropies (e.g. $N(0,1)$ has a differential entropy that is one bit
less than $N(0,4)$, though one can obtain a $N(0,4)$ random variable
by scaling up a $N(0,1)$ random variable). Differential entropy is
generally not invariant under bijective mappings. It is not a measure
of information in a random variable, but rather, very loosely speaking,
a measure of ``information per unit length'' which is sensitive
to the scale. Therefore, differential entropy would not fit the definition
of ``entropy'' in this paper.

\medskip{}

\subsection{Universal Entropy of the Opposite Category of Sets and ``Information
as Commodity''}

In Section \ref{sec:set}, we have derived the universal entropy of
the epimorphism category of $\mathsf{FinISet}$. Now we will consider
$\mathsf{FinSet}^{\mathrm{op}}$ instead, which reveals an interesting
``dual'' result.

Consider the category $\mathrm{Epi}(\mathsf{FinSet}^{\mathrm{op}})$
consisting of finite sets as objects, and a morphism $A\to B$ is
given by an injective function $f:B\to A$. The tensor product is
the disjoint union $A\sqcup B$, and the under category $N(A)=A/\mathrm{Epi}(\mathsf{FinSet}^{\mathrm{op}})$
is equipped with a tensor product, where for injective functions $f:B\to A$,
$g:C\to A$, the tensor product is given by the injective function
$h:D\to A$, where $D=f(B)\cup g(C)$, $h(x)=x$. Note that the tensor
product is also the categorical product over $A/\mathrm{Epi}(\mathsf{FinSet}^{\mathrm{op}})$,
with projection morphisms given by $f$ and $g$ restricted to $D$.
\\
\[\begin{tikzcd}
	& A \\
	B & D & C
	\arrow["f", from=2-1, to=1-2]
	\arrow["g"', from=2-3, to=1-2]
	\arrow["h"', from=2-2, to=1-2]
	\arrow["{f|_D}"', from=2-1, to=2-2]
	\arrow["{g|_D}", from=2-3, to=2-2]
\end{tikzcd}\]The under category $N(A)$ can also be regarded as the poset of subobjects
or subsets of $A$, with the tensor product given by the union of
two subsets. It is straightforward to verify that $-/\mathrm{Epi}(\mathsf{FinSet}^{\mathrm{op}}):\mathrm{Epi}(\mathsf{FinSet}^{\mathrm{op}})^{\mathrm{op}}\to\mathsf{MonCat}_{s}$
is a functor. We now find the universal entropy of $\mathrm{Epi}(\mathsf{FinSet}^{\mathrm{op}})$,
which is given by the cardinality.

\medskip{}

\begin{prop}
For the MonSiMon category $(\mathrm{Epi}(\mathsf{FinSet}^{\mathrm{op}}),\,-/\mathrm{Epi}(\mathsf{FinSet}^{\mathrm{op}}))$:
\begin{enumerate}
\item The universal baseless $\mathsf{OrdCMon}$-entropy is given by $|\cdot|:\mathrm{Epi}(\mathsf{FinSet}^{\mathrm{op}})\to\mathbb{Z}_{\ge0}$.
Same for $\mathsf{COrdCMon}$ and $\mathsf{IcOrdCMon}$.
\item The universal baseless $\mathsf{OrdAb}$-entropy is given by $|\cdot|:\mathrm{Epi}(\mathsf{FinSet}^{\mathrm{op}})\to\mathbb{Z}$.
Same for $\mathsf{IcOrdAb}$.
\item The universal baseless $\mathsf{OrdVect}_{\mathbb{Q}}$-entropy is
given by $|\cdot|:\mathrm{Epi}(\mathsf{FinSet}^{\mathrm{op}})\to\mathbb{Q}$.
Same for $\mathsf{IcOrdVect}_{\mathbb{Q}}$.
\end{enumerate}
\end{prop}

\medskip{}

\begin{proof}
We prove the universal baseless $\mathsf{OrdCMon}$-entropy result,
from which the other results follow directly. Fix any $\mathsf{W}'\in\mathsf{OrdCMon}$
and strong monoidal functor $H':\mathrm{Epi}(\mathsf{FinSet}^{\mathrm{op}})\to\mathsf{W}'$.
Since $H'$ sends isomorphic objects to isomorphic objects, and $\mathsf{W}'$
is skeletal, $H'(X)$ only depends on $|X|$. Since $H'$ is a functor,
$H'(X)\ge H'(Y)$ whenever $|X|\ge|Y|$ since there is an injective
function $Y\to X$. Since $H'$ is strongly monoidal, $H'(X\sqcup Y)=H'(X)+H'(Y)$.
These imply that $H'(X)=|X|H'(\{1\})$, and hence $H'=tH$ where $t:\mathbb{Z}_{\ge0}\to\mathsf{W}'$,
$x\mapsto xH'(\{1\})$. Such $t$ is clearly unique.
\end{proof}
\medskip{}

To see how cardinality can be interpreted as ``entropy'', consider
the scenario where the information is contained in a number of discrete,
inseparable, identical units of commodities, for example, when we
have a number of hard disks. Morphisms are operations we can perform
on these hard disks without reading their contents or introducing
new information, i.e., reordering them, or discarding some of them.
Note that duplicating a hard disk is not a meaningful operation, since
the information in two hard disks with the same content is the same
as the information in any one of them. Hence, such an ``information-noncreating''
operation can be represented by an injective function $f$ in the
opposite direction, where $f(x)$ is the old location of the hard
disk with a new location $x$. The ``tensor product'' between two
sets of hard disks is the disjoint union of the two sets. Given a
set of hard disks, a ``random variable'' is simply a subset of these
hard disks, and the ``joint random variable'' between two random
variables is the union of the two sets. Assuming each hard disk has
$b$ bits of information. Then the entropy of a set of $n$ unrelated
hard disks has an entropy $nb$ bits. Hence, the entropy of a set
of hard disks is proportional to the cardinality of the set. 

Indeed, we can embed the MonSiMon category $-/\mathrm{Epi}(\mathsf{FinSet}^{\mathrm{op}})$
into the MonSiMon category $-/\mathsf{FinProb}$ by mapping a set
$A$ to the product distribution $P^{A}$ (i.e., the distribution
of $|A|$ i.i.d. random variables following $P$, indexed by elements
in $A$), where $P$ is any fixed distribution, and mapping a morphism
to the corresponding map that reorders and discards some of the entries
of the i.i.d. sequence. This gives the following commutative diagram
in $\mathsf{MonSiMonCat}$, where the vertical arrow on the left is
the aforementioned embedding, and the two horizontal arrows $|\cdot|\circ\mathrm{Cod}$
and $(H_{0},H_{1})\circ\mathrm{Cod}$ are universal $\mathsf{IcOrdAb}$-entropies
of $-/\mathrm{Epi}(\mathsf{FinSet}^{\mathrm{op}})$ and $-/\mathsf{FinProb}$,
respectively. \\
\[\begin{tikzcd}
	{-/\mathrm{Epi}(\mathsf{FinSet}^{\mathrm{op}})} && {\mathbb{Z}} \\
	{-/\mathsf{FinProb}} && {(\log\mathbb{Q}_{>0})\times\mathbb{R}}
	\arrow["{|\cdot|\circ\mathrm{Cod}}", from=1-1, to=1-3]
	\arrow["{(H_{0},H_{1})\circ\mathrm{Cod}}", from=2-1, to=2-3]
	\arrow[hook', from=1-1, to=2-1]
	\arrow["{x\mapsto(xH_0(P),xH_1(P))}", from=1-3, to=2-3]
\end{tikzcd}\]

This embedding basically says ``commodity is information'', which
is technically true in physics since every object contains information
about its state. However, the converse ``information is commodity''
does not hold since there is no embedding in the other direction.
The classical example (e.g., see \cite{yeung2012first}) is to consider
$X_{1},X_{2}$ to be two i.i.d. fair coin flips, and $X_{3}=X_{1}\oplus X_{2}$
is the XOR of $X_{1}$ and $X_{2}$. Any two of $X_{1},X_{2},X_{3}$
will contain $2$ bits of information. However, all of $X_{1},X_{2},X_{3}$
also only contains $2$ bits of information. It is impossible to find
sets $A_{1},A_{2},A_{3}$ such that the cardinality of any one of
them is $b$, the cardinality of the union of any two is $2b$, and
the cardinality of the union of the three is $2b$. Hence, information
cannot be treated as sets or commodities.

The fact that information cannot be treated as commodities is arguably
the reason why coding is effective. If bits are really commodities,
then the only error correcting codes would be the repetition codes
where we simply repeat each bit a number of times. Also, the work
on network coding \cite{ahlswede2000network,li2003linear} has shown
that it is suboptimal to treating packets as units of commodities
that can only be routed through a network, and the problem of finding
the optimal communication rates between pairs of nodes over a network
is significantly more complicated than multi-commodity flow.

We remark that there are works on embedding random variables into
sets (e.g., \cite{ting1962amount,yeung1991new,ellerman2017logical,down2023logarithmic,li2023poisson}),
though they cannot be made into MonSiMon functors, and are out of
the scope of this paper.

\medskip{}

\subsection{Universal Entropy of the Augmented Simplex Category and ``Information
of Transformations''}

We consider the augmented simplex category $\Delta_{+}$, where objects
are ordered sets in the form $[n]=\{0,\ldots,n\}$ (where $n\ge-1$;
note that $[-1]=\emptyset$), and morphisms are order-preserving functions
between these sets. Consider $\mathrm{Epi}(\Delta_{+})$ which contain
epimorphisms (surjective functions) in $\Delta_{+}$. The tensor product
is given by concatenation of ordered sets. The under category $N([a])=[a]/\mathrm{Epi}(\Delta_{+})$
is equipped with a tensor product, where for surjective $f:[a]\to[b]$,
$g:[a]\to[c]$, $f\otimes g:[a]\to\mathrm{im}((f,g))$ is the restriction
of the codomain of $(f,g):[a]\to[b]\times[c]$ to its image, where
$\mathrm{im}((f,g))$ is equipped with the product order (which is
guaranteed to be a total order over $\mathrm{im}((f,g))$, so $\mathrm{im}((f,g))$
is a totally ordered set). We now find the universal entropies of
$\mathrm{Epi}(\Delta_{+})$.
\begin{prop}
For the MonSiMon category $(\mathrm{Epi}(\Delta_{+}),\,-/\mathrm{Epi}(\Delta_{+}))$:
\begin{enumerate}
\item The universal baseless $\mathsf{OrdCMon}$-entropy is given by $\mathrm{Epi}(\Delta_{+})\to\mathbb{Z}_{\ge0}$,
$[n]\mapsto n+1$. Same for $\mathsf{COrdCMon}$ and $\mathsf{IcOrdCMon}$.
\item The universal baseless $\mathsf{OrdAb}$-entropy is given by $\mathrm{Epi}(\Delta_{+})\to\mathbb{Z}$,
$[n]\mapsto n+1$. Same for $\mathsf{IcOrdAb}$.
\item The universal baseless $\mathsf{OrdVect}_{\mathbb{Q}}$-entropy is
given by $\mathrm{Epi}(\Delta_{+})\to\mathbb{Q}$, $[n]\mapsto n+1$.
Same for $\mathsf{IcOrdVect}_{\mathbb{Q}}$.
\end{enumerate}
\end{prop}

\begin{proof}
We prove the universal baseless $\mathsf{OrdCMon}$-entropy result,
from which the other results follow directly. Fix any $\mathsf{W}'\in\mathsf{OrdCMon}$
and strong monoidal functor $H':\mathrm{Epi}(\Delta_{+})\to\mathsf{W}'$.
Since $H'$ is a functor, $H'([m])\ge H'([n])$ whenever $m\ge n$.
Since $H'$ is strongly monoidal, $H'([m]\otimes[n])=H'([m])+H'([n])$.
These imply that $H'([n])=(n+1)H'([0])$, and hence $H'=tH$ where
$t:\mathbb{Z}_{\ge0}\to\mathsf{W}'$, $x\mapsto xH'([0])$. Such $t$
is clearly unique.
\end{proof}
\medskip{}

To relate this ``entropy'' with the conventional notion of entropy
of random variables, we consider an embedding from the MonSiMon category
$-/\mathrm{Epi}(\Delta_{+})$ into the MonSiMon category $-/\mathsf{FinProb}$,
where $[n]$ is mapped to the distribution of an i.i.d. sequence $X_{0},\ldots,X_{n}$
of uniformly distributed elements in a fixed finite group $G$, which
can be regarded as a ``sequence of transformations''. The Shannon
and Hartley entropy of this sequence is $H(X_{0},\ldots,X_{n})=(n+1)\log|G|$,
proportional to $n+1$. A morphism from $[n]$ to $[m]$ in $\mathrm{Epi}(\Delta_{+})$
can be regarded as a way to combine consecutive transformations, that
is, an order-preserving surjection $f:[n]\to[m]$ is mapped to a mapping
from $X_{0},\ldots,X_{n}$ to $Y_{0},\ldots,Y_{m}$ where $Y_{i}=\prod_{k\in f^{-1}(\{i\})}X_{k}$.
The tensor product of $\mathrm{Epi}(\Delta_{+})$ corresponds to concatenating
two sequences of transformations. To check that this embedding preserves
the categorical product over the under category $N([a])=[a]/\mathrm{Epi}(\Delta_{+})$,
for $f:[a]\to[b]$ and $g:[a]\to[c]$, the product $f\otimes g:[a]\to\mathrm{im}((f,g))$
corresponds to the sequence $(\prod_{k\in(f,g)^{-1}(\{(i,j)\})}X_{k})_{(i,j)\in\mathrm{im}((f,g))}$
(this is a sequence since $\mathrm{im}((f,g))$ is totally ordered),
which contains the same information as $(\prod_{k\in f^{-1}(\{i\})}X_{k})_{i\in[b]}$
(corresponding to $f$) together with $(\prod_{k\in g^{-1}(\{j\})}X_{k})_{j\in[c]}$
(corresponding to $g$). 

This embedding from $-/\mathrm{Epi}(\Delta_{+})$ to $-/\mathsf{FinProb}$
gives the following commutative diagram in $\mathsf{MonSiMonCat}$,
where the vertical arrow on the left is the embedding, and the two
horizontal arrows $([n]\mapsto n+1)\circ\mathrm{Cod}$ and $(H_{0},H_{1})\circ\mathrm{Cod}$
are universal $\mathsf{IcOrdAb}$-entropies of $-/\mathrm{Epi}(\Delta_{+})$
and $-/\mathsf{FinProb}$, respectively. \\
\[\begin{tikzcd}
	{-/\mathrm{Epi}(\Delta_{+})} && {\mathbb{Z}} \\
	{-/\mathsf{FinProb}} && {(\log\mathbb{Q}_{>0})\times\mathbb{R}}
	\arrow["{([n]\mapsto n+1)\circ\mathrm{Cod}}", from=1-1, to=1-3]
	\arrow["{(H_{0},H_{1})\circ\mathrm{Cod}}", from=2-1, to=2-3]
	\arrow[hook', from=1-1, to=2-1]
	\arrow["{x\mapsto(x \log|G|,x \log|G|)}", from=1-3, to=2-3]
\end{tikzcd}\]

\medskip{}

\subsection{Universal Entropy of Finite Abelian Groups\label{subsec:ab}}

Consider the category $\mathrm{Epi}(\mathsf{FinAb})$ consisting of
finite abelian groups as objects, and epimorphisms between them as
morphisms. The tensor product is the direct sum, and the under category
$N(A)=A/\mathrm{Epi}(\mathsf{FinAb})$ is equipped with a tensor product,
where for $f:A\to B$, $g:A\to C$, $f\otimes g:A\to\mathrm{im}((f,g))$
is the restriction of the codomain of $(f,g):A\to B\oplus C$ to its
image. Note that $f\otimes g$ is also the categorical product over
$A/\mathrm{Epi}(\mathsf{FinAb})$. We now find the universal entropies
of $(\mathrm{Epi}(\mathsf{FinAb}),\,-/\mathrm{Epi}(\mathsf{FinAb}))$.

\medskip{}

\begin{thm}
For $A\in\mathsf{FinAb}$, $i,k\in\mathbb{N}$, let $M_{i,k}(A)$
be the multiplicity of $\mathbb{Z}_{p_{i}^{j}}$ ($p_{i}$ is the
$i$-th prime number) in the decomposition of $A$, i.e., the decomposition
of $A$ is $A\cong\bigoplus_{i,j\in\mathbb{N}}\mathbb{Z}_{p_{i}^{j}}^{M_{i,k}(A)}$.
Define a function $M:\mathsf{FinAb}\to\mathbb{Z}_{\ge0}^{\oplus\mathbb{N}^{2}}$,
$M(A)=(M_{i,k}(A))_{i,k\in\mathbb{N}}$. For the MonSiMon category
$(\mathrm{Epi}(\mathsf{FinAb}),\,-/\mathrm{Epi}(\mathsf{FinAb}))$:
\begin{enumerate}
\item The universal baseless $\mathsf{OrdCMon}$-entropy is given by $M:\mathrm{Epi}(\mathsf{FinAb})\to(\mathbb{Z}_{\ge0}^{\oplus\mathbb{N}^{2}},\succeq)$,\footnote{$\mathbb{Z}_{\ge0}^{\oplus\mathbb{N}^{2}}=\bigoplus_{i,j\in\mathbb{N}}\mathbb{Z}_{\ge0}$
denotes the direct sum of monoids, where each element is in the form
$(\alpha_{i,j})_{i,j\in\mathbb{N}}$ with finitely many nonzero entries.} where $a\succeq\beta$ if $\sum_{j=k}^{\infty}\alpha_{i,j}\ge\sum_{j=k}^{\infty}\beta_{i,j}$
for every $i,k\in\mathbb{N}$. Same for $\mathsf{COrdCMon}$ and $\mathsf{IcOrdCMon}$.
\item The universal baseless $\mathsf{OrdAb}$-entropy is given by $M:\mathrm{Epi}(\mathsf{FinAb})\to(\mathbb{Z}^{\oplus\mathbb{N}^{2}},\succeq)$.
Same for $\mathsf{IcOrdAb}$.
\item The universal baseless $\mathsf{OrdVect}_{\mathbb{Q}}$-entropy is
given by $M:\mathrm{Epi}(\mathsf{FinAb})\to(\mathbb{Q}^{\oplus\mathbb{N}^{2}},\succeq)$.
Same for $\mathsf{IcOrdVect}_{\mathbb{Q}}$.
\end{enumerate}
\end{thm}

\begin{proof}
We first state a well-known fact about finite abelian groups (often
stated in other forms, e.g. see \cite{fuchs1960abelian}). Consider
finite abelian groups $A,B$. The following are equivalent:
\begin{enumerate}
\item There is a subgroup of $A$ isomorphic to $B$.
\item There is a quotient group of $A$ isomorphic to $B$.
\item $M(A)\succeq M(B)$.
\end{enumerate}
For the sake of completeness, we include the proof of $1\Rightarrow3$
(the arguments are definitely not novel, though we could not find
a reference). We appeal to \cite{fuchs1960abelian} for $2\Rightarrow1$,
and $3\Rightarrow2$ follows directly from the fact that $\mathbb{Z}_{p_{i}^{j}}$
has a quotient isomorphic to $\mathbb{Z}_{p_{i}^{k}}$ if $j\ge k$.
Consider $A=\bigoplus_{i,j\in\mathbb{N}}\mathbb{Z}_{p_{i}^{j}}^{\alpha_{i,j}}$.
Fix $i,k\in\mathbb{Z}$. Let $S\subseteq A$ be the set of elements
with order that divides $p_{i}^{k}$. We have $|p_{i}^{k-1}S|=\prod_{j=k}^{\infty}p_{i}^{\alpha_{i,j}}$
since $p_{i}^{k-1}\mathbb{Z}_{p_{i}^{j}}=\{0\}$ for $j<k$; and for
$j\ge k$, the set of elements in $\mathbb{Z}_{p_{i}^{j}}$ with order
that divides $p_{i}^{k}$ is $\{0,p_{i}^{j-k},2p_{i}^{j-k},\ldots,(p^{k}-1)p_{i}^{j-k}\}$,
and its image after multiplying $p_{i}^{k-1}$ is $\{0,p_{i}^{j-1},2p_{i}^{j-1},\ldots,(p-1)p_{i}^{j-1}\}$.
Assume $B\subseteq A$ is a subgroup with $B\cong\mathbb{Z}_{p_{i}^{k}}^{t}$
for some $t\in\mathbb{N}$. We have $|p_{i}^{k-1}B|=p_{i}^{t}$. Since
$B\subseteq S$, $p_{i}^{k-1}B\subseteq p_{i}^{k-1}S$, we have $p_{i}^{t}\le\prod_{j=k}^{\infty}p_{i}^{\alpha_{i,j}}$,
$t\le\sum_{j=k}^{\infty}\alpha_{i,j}$. The result follows from that
$\bigoplus_{i,j\in\mathbb{N}}\mathbb{Z}_{p_{i}^{j}}^{\beta_{i,j}}$
has a subgroup isomorphic to $\mathbb{Z}_{p_{i}^{k}}^{\sum_{j=k}^{\infty}\beta_{i,j}}$.

We then show that $M:\mathrm{Epi}(\mathsf{FinAb})\to(\mathbb{Z}_{\ge0}^{\oplus\mathbb{N}^{2}},\succeq)$
is a baseless $\mathsf{OrdCMon}$-entropy. For a morphism $f:A\to B$
in $\mathrm{Epi}(\mathsf{FinAb})$, by the fundamental theorem on
homomorphisms, $B$ is isomorphic to the quotient group $A/\mathrm{ker}(f)$,
and hence $M(A)\succeq M(B)$, meaning that $M$ is a functor. It
is clearly strongly monoidal. It is left to check the subadditivity
property that $M\mathrm{Cod}:A/\mathrm{Epi}(\mathsf{FinAb})\to(\mathbb{Z}_{\ge0}^{\oplus\mathbb{N}^{2}},\succeq)$
is a lax monoidal functor for all $A\in\mathrm{Epi}(\mathsf{FinAb})$.
Consider epimorphisms $f:A\to B$, $g:A\to C$ in $\mathrm{Epi}(\mathsf{FinAb})$,
and let $f\otimes g:A\to D$. By construction, $D$ is the image of
$(f,g):A\to(B\oplus C)$, and is a subgroup of $B\oplus C$. Hence
$M(D)\preceq M(B\oplus C)=M(B)+M(C)$ for $i,k\in\mathbb{N}$.

Finally, we prove the universal property of $M$. Let $H':\mathrm{Epi}(\mathsf{FinAb})\to\mathsf{W}'$
be any baseless $\mathsf{OrdCMon}$-entropy. Since $\mathsf{W}'$
is skeletal, $H'$ maps isomorphic groups to the same object, so $H'(A)$
is determined by $M(A)$. Let $H'(A)=t(M(A))$ where $t:M(\mathsf{FinAb})\to\mathsf{W}'$
is a function. It is left to check that $t$ is a morphism in $\mathsf{OrdCMon}$.
Let $\theta_{i,j}:=t(M(\mathbb{Z}_{p_{i}^{j}}))$. Since $H'$ is
a strong monoidal functor, for $A\cong\bigoplus_{i,j\in\mathbb{N}}\mathbb{Z}_{p_{i}^{j}}^{\alpha_{i,j}}$,
we have 
\[
t(\alpha)=H'(A)=\sum_{i,j}\alpha_{i,j}H'(\mathbb{Z}_{p_{i}^{j}})=\sum_{i,j}\alpha_{i,j}\theta_{i,j}.
\]
Since there is a morphism from $\mathbb{Z}_{p_{i}^{j+1}}$ to $\mathbb{Z}_{p_{i}^{j}}$
and to $\{0\}$ in $\mathrm{Epi}(\mathsf{FinAb})$, we have $\theta_{i,j+1}\ge\theta_{i,j}\ge0$.
Note that for every $\alpha\succeq\beta$, i.e., $\sum_{j=k}^{\infty}\alpha_{i,j}\ge\sum_{j=k}^{\infty}\beta_{i,j}$
for $i,k\in\mathbb{N}$, it is possible to transform $\alpha$ into
$\beta$ using a finite number of steps in this form: if $\alpha_{i,j}\ge1$,
decrease $\alpha_{i,j}$ by $1$, and increase $\alpha_{i,j-1}$ by
$1$ if $j\ge2$. Since $\theta_{i,j+1}\ge\theta_{i,j}\ge0$, such
a step cannot increase $t(\alpha)$, and hence $t(\alpha)\ge t(\beta)$
if $\alpha\succeq\beta$. Therefore, $t$ is an order-preserving homomorphism.
It is clear that such $t$ is unique.
\end{proof}
\medskip{}

The characterization of the universal entropy of $\mathrm{Epi}(\mathsf{FinGrp})$,
where $\mathsf{FinGrp}$ is the category of finite groups, is left
for future studies.

\medskip{}

\section{Global Naturality\label{sec:global}}

Recall that Definition \ref{def:gcm_entropy} gives the universal
$\mathsf{V}$-entropy of various MonSiMon categories $\mathsf{C}\in\mathsf{MonSiMonCat}$
as components $\phi_{\mathsf{C}}:\mathsf{C}\to T(\mathsf{C})$ of
the unit of the universal $\mathsf{V}$-entropy monad $T$. This allows
us to link the universal entropies of all MonSiMon categories in a
natural and functorial manner, i.e., the mapping $T$ sending a MonSiMon
category $\mathsf{C}$ to the codomain of its universal entropy $T(\mathsf{C})$
is a functor, and the mapping sending $\mathsf{C}$ to the arrow $\phi_{\mathsf{C}}:\mathsf{C}\to T(\mathsf{C})$
(in the arrow category of $\mathsf{MonSiMonCat}$) is a functor as
well.

In this paper, we have studied the following MonSiMon categories,
which can be linked by the MonSiMon functors drawn below (where we
write $-/\mathsf{C}$ for the MonSiMon category $(\mathsf{C},-/\mathsf{C})$
for brevity, $\mathbb{F}$ is a finite field, and $0<\rho<1$):\\
\[\begin{tikzcd}[sep=scriptsize]
	{*} & {-/\mathrm{Epi}(\Delta_{+})} & {-/\mathrm{Epi}(\mathsf{FinSet}^{\mathrm{op}})} & {-/\mathrm{Epi}(\mathsf{FinVect}_{\mathbb{F}})} \\
	& {-/\mathrm{Epi}(\mathsf{FinISet})} \\
	{*} && {-/\mathsf{FinProb}} & {-/\mathrm{Epi}(\mathsf{FinAb})} \\
	& {-/\mathsf{LProb}_{\rho}}
	\arrow[from=3-3, to=2-2]
	\arrow[from=1-3, to=1-4]
	\arrow[from=1-4, to=3-4]
	\arrow[from=3-4, to=3-3]
	\arrow[from=1-1, to=1-2]
	\arrow[from=1-2, to=1-3]
	\arrow[from=3-3, to=4-2]
	\arrow[from=2-2, to=3-1]
	\arrow[from=4-2, to=3-1]
\end{tikzcd}\]Each of these functors will be explained separately later. Note that
$*$ denotes the zero object of $\mathsf{MonSiMonCat}$ (given by
the pair $(*,\,*\stackrel{*}{\to}\mathsf{MonCat}_{\ell})$, where
$*\stackrel{*}{\to}\mathsf{MonCat}_{\ell}$ is the functor sending
the unique object to $*\in\mathsf{MonCat}_{\ell}$), and has a universal
$\mathsf{IcOrdAb}$-entropy $(!,h)$, $h_{\bullet}(\bullet)=0$. Considering
the reflective subcategory $\mathsf{IcOrdAb}\hookrightarrow\mathsf{MonSiMonCat}$,
and passing these MonSiMon categories to the monad unit $\phi$, and
passing these MonSiMon functors to the monad $T:\mathsf{MonSiMonCat}\to\mathsf{IcOrdAb}$,
we have the following commutative diagram in $\mathsf{MonSiMonCat}$
linking those MonSiMon categories and their universal $\mathsf{IcOrdAb}$-entropies:\\
\[\begin{tikzcd}
	{\mathsf{C}} && {T(\mathsf{C})} \\
	{*} && {\{0\}} \\
	{-/\mathrm{Epi}(\Delta_{+})} && {\mathbb{Z}} \\
	{-/\mathrm{Epi}(\mathsf{FinSet}^{\mathrm{op}})} && {\mathbb{Z}} \\
	{-/\mathrm{Epi}(\mathsf{FinVect}_{\mathbb{F}})} && {\mathbb{Z}} \\
	{-/\mathrm{Epi}(\mathsf{FinAb})} && {(\mathbb{Z}^{\oplus\mathbb{N}^{2}},\succeq)} \\
	{-/\mathsf{FinProb}} && {(\log\mathbb{Q}_{>0})\times\mathbb{R}} \\
	& {-/\mathsf{LProb}_{\rho}} && {\mathbb{R}} \\
	{-/\mathrm{Epi}(\mathsf{FinISet})} && {\log \mathbb{Q}_{>0}} \\
	& {*} && {\{0\}}
	\arrow[from=7-1, to=9-1]
	\arrow["{\pi_1}"{description, pos=0.7}, from=7-3, to=9-3]
	\arrow["{|\cdot|\circ\mathrm{Cod}}", from=4-1, to=4-3]
	\arrow["{M\circ\mathrm{Cod}}", from=6-1, to=6-3]
	\arrow["{(H_{0},H_{1})\circ\mathrm{Cod}}", from=7-1, to=7-3]
	\arrow["{(\log|\cdot|)\circ\mathrm{Cod}}"{pos=0.2}, from=9-1, to=9-3]
	\arrow[from=9-3, to=10-4]
	\arrow[from=9-1, to=10-2]
	\arrow["{0\circ\mathrm{Cod}}", from=10-2, to=10-4]
	\arrow["{0\circ\mathrm{Cod}}", from=2-1, to=2-3]
	\arrow["{\dim\circ\mathrm{Cod}}", from=5-1, to=5-3]
	\arrow[from=4-1, to=5-1]
	\arrow[Rightarrow, no head, from=4-3, to=5-3]
	\arrow[from=5-1, to=6-1]
	\arrow[from=5-3, to=6-3]
	\arrow[from=6-1, to=7-1]
	\arrow[from=6-3, to=7-3]
	\arrow[from=2-1, to=3-1]
	\arrow[from=3-1, to=4-1]
	\arrow[from=2-3, to=3-3]
	\arrow[Rightarrow, no head, from=3-3, to=4-3]
	\arrow["{([n]\mapsto n+1)\circ\mathrm{Cod}}", from=3-1, to=3-3]
	\arrow[from=7-1, to=8-2]
	\arrow["{\pi_2}", from=7-3, to=8-4]
	\arrow[from=8-4, to=10-4]
	\arrow["{H_1\circ\mathrm{Cod}}"{pos=0.2}, from=8-2, to=8-4]
	\arrow[from=8-2, to=10-2]
	\arrow["{\phi_{\mathsf{C}}}", from=1-1, to=1-3]
\end{tikzcd}\]

The categories on the left are domains of the entropies, that are
MonSiMon categories in the form $(\mathsf{C},-/\mathsf{C})$. The
categories on the right are codomains of the entropies, that are integrally
closed partially ordered abelian groups, which are also MonSiMon categories
via the embedding in Section \ref{sec:subcats}.  Each horizontal
functor above are the universal $\mathsf{IcOrdAb}$-entropies of those
MonSiMon categories. For notational simplicity, a universal $\mathsf{IcOrdAb}$-entropy
is written as $h=H\circ\mathrm{Cod}$, where $H$ is the universal
baseless $\mathsf{IcOrdAb}$-entropy.\footnote{This notation is justified by the one-to-one correspondence $Uh=UH\circ\mathrm{Cod}$
in Proposition \ref{prop:baseless}, where we omit the forgetful functor
$U$ for brevity.} We now explain each of the vertical functors above:
\begin{itemize}
\item $-/\mathrm{Epi}(\Delta_{+})\to-/\mathrm{Epi}(\mathsf{FinSet}^{\mathrm{op}})$
is the MonSiMon functor $(F,\gamma)$, with $F$ sending $[n]$ to
$\{0,\ldots,n\}$, and sending order-preserving surjective function
$f:[m]\to[n]$ to the injective function $\{0,\ldots,n\}\to\{0,\ldots,m\}$,
$x\mapsto\min f^{-1}(\{x\})$ (or $\max f^{-1}(\{x\})$), and $\gamma:-/\mathrm{Epi}(\Delta_{+})\Rightarrow F^{\mathrm{op}}(-)/\mathrm{Epi}(\mathsf{FinSet}^{\mathrm{op}})$,
$\gamma_{P}(X)=F(X)$ for $X:P\to Q$. After passing this MonSiMon
functor to the monad $T$, the ``$\mathbb{Z}=\mathbb{Z}$'' functor
in the diagram on the right is simply the identity MonSiMon functor.
\item $-/\mathrm{Epi}(\mathsf{FinSet}^{\mathrm{op}})\to-/\mathrm{Epi}(\mathsf{FinVect}_{\mathbb{F}})$
is the MonSiMon $(F,\gamma)$, with $F$ sending $A\in\mathsf{FinSet}$
to $\mathbb{F}^{A}\in\mathsf{FinVect}_{\mathbb{F}}$, and sending
injective function $f:A\to B$ to $\mathbb{F}^{B}\to\mathbb{F}^{A}$,
$(x_{i})_{i\in B}\mapsto(x_{f(j)})_{j\in A}$, and $\gamma$ is the
same as above. After passing this MonSiMon functor to the monad $T$,
the ``$\mathbb{Z}=\mathbb{Z}$'' functor in the diagram on the right
is simply the identity MonSiMon functor.
\item $-/\mathrm{Epi}(\mathsf{FinVect}_{\mathbb{F}})\to-/\mathrm{Epi}(\mathsf{FinAb})$
is the inclusion MonSiMon functor. After passing this MonSiMon functor
to the monad $T$, the ``$\mathbb{Z}\to(\mathbb{Z}^{\oplus\mathbb{N}^{2}},\succeq)$''
functor in the diagram on the right is the MonSiMon functor mapping
$x\in\mathbb{R}$ to $(\alpha_{i,k})_{i,k\in\mathbb{N}}$ where $\alpha_{i,k}=x$
if $p_{i}=p$ and $k=m$ (assuming $|\mathbb{F}|=p^{m}$), and $\alpha_{i,k}=0$
otherwise.
\item $-/\mathrm{Epi}(\mathsf{FinAb})\to-/\mathsf{FinProb}$ is the MonSiMon
functor $(F,\gamma)$, with $F$ sending $A\in\mathsf{FinAb}$ to
the uniform distribution over $A$, and sending $f:A\to B$ in $\mathrm{Epi}(\mathsf{FinAb})$
to the measure-preserving function from the uniform distribution over
$A$ to that over $B$ given by $f$, and $\gamma$ is the same as
above. After passing this MonSiMon functor to the monad $T$, the
``$(\mathbb{Z}^{\oplus\mathbb{N}^{2}},\succeq)\to(\log\mathbb{Q}_{>0})\times\mathbb{R}$''
functor in the diagram on the right is the MonSiMon functor mapping
$(\alpha_{i,k})_{i,k\in\mathbb{N}}$ to $(\sum_{i,k}\alpha_{i,k}\log p_{i}^{k},\,\sum_{i,k}\alpha_{i,k}\log p_{i}^{k})$.
\item $-/\mathsf{FinProb}\to-/\mathrm{Epi}(\mathsf{FinISet})$ is the MonSiMon
functor $(F,\gamma)$, with $F$ sending $P\in\mathsf{FinProb}$ to
the support of $P$ (the functor $\mathrm{Supp}$ in Section \ref{sec:set}).
After passing this MonSiMon functor to the monad $T$, the ``$\pi_{1}:(\log\mathbb{Q}_{>0})\times\mathbb{R}\to\log\mathbb{Q}_{>0}$''
functor in the diagram on the right is the MonSiMon functor mapping
$(x_{0},x_{1})$ to $x_{0}$.
\item $-/\mathsf{FinProb}\to-/\mathsf{LProb}_{\rho}$ is the inclusion MonSiMon
functor. After passing this MonSiMon functor to the monad $T$, the
``$\pi_{2}:(\log\mathbb{Q}_{>0})\times\mathbb{R}\to\mathbb{R}$''
functor in the diagram on the right is the MonSiMon functor mapping
$(x_{0},x_{1})$ to $x_{1}$.
\end{itemize}
\medskip{}

Moreover, we can link these MonSiMon categories and their universal
$\mathsf{OrdCMon}$, $\mathsf{IcOrdCMon}$, $\mathsf{IcOrdAb}$ and
$\mathsf{IcOrdVect}_{\mathbb{Q}}$-entropies, placing almost all the
universal entropies discussed in this paper so far in the following
commutative diagram. Also refer to Figure \ref{fig:reflect2} for
the same figure drawn in a way that highlights the objects and morphisms
in different reflective subcategories.\\
{\footnotesize{}
\[\begin{tikzcd}[column sep=scriptsize]
	{\mathsf{MonSiMonCat}} & {\;\mathsf{OrdCMon}\;} & {\mathsf{IcOrdCMon}} & {\mathsf{IcOrdAb}} & {\mathsf{IcOrdVect}_\mathbb{Q}} \\
	\ni & \ni & \ni & \ni & \ni \\
	{-/\mathrm{Epi}(\Delta_{+})} & {\mathbb{Z}_{\ge 0}} & {\mathbb{Z}_{\ge 0}} & {\mathbb{Z}} & {\mathbb{Q}} \\
	{-/\mathrm{Epi}(\mathsf{FinSet}^{\mathrm{op}})} & {\mathbb{Z}_{\ge 0}} & {\mathbb{Z}_{\ge 0}} & {\mathbb{Z}} & {\mathbb{Q}} \\
	{-/\mathrm{Epi}(\mathsf{FinVect}_{\mathbb{F}})} & {\mathbb{Z}_{\ge 0}} & {\mathbb{Z}_{\ge 0}} & {\mathbb{Z}} & {\mathbb{Q}} \\
	{-/\mathrm{Epi}(\mathsf{FinAb})} & {(\mathbb{Z}_{\ge0}^{\oplus\mathbb{N}^{2}},\succeq)} & {(\mathbb{Z}_{\ge0}^{\oplus\mathbb{N}^{2}},\succeq)} & {(\mathbb{Z}^{\oplus\mathbb{N}^{2}},\succeq)} & {(\mathbb{Q}^{\oplus\mathbb{N}^{2}},\succeq)} \\
	\\
	{-/\mathsf{FinProb}} & {(\mathsf{FinProb},\succsim_{01})} & {(H_{0},H_{1})(\mathsf{FinProb})} & {(\log\mathbb{Q}_{>0})\times\mathbb{R}} & {(\mathbb{Q} \log \mathbb{Q}_{>0}) \times \mathbb{R}} \\
	{-/\mathrm{Epi}(\mathsf{FinISet})} & {\log \mathbb{N}} & {\log \mathbb{N}} & {\log \mathbb{Q}_{>0}} & {\mathbb{Q} \log \mathbb{Q}_{>0}} \\
	\\
	{-/\mathsf{LProb}_{\rho}} & {?} & {\mathbb{R}_{\ge 0}} & {\mathbb{R}} & {\mathbb{R}}
	\arrow[curve={height=6pt}, from=8-1, to=9-1]
	\arrow[from=4-1, to=5-1]
	\arrow[Rightarrow, no head, from=4-5, to=5-5]
	\arrow[from=5-1, to=6-1]
	\arrow[from=5-5, to=6-5]
	\arrow[from=6-1, to=8-1]
	\arrow[from=6-5, to=8-5]
	\arrow[from=3-1, to=4-1]
	\arrow[Rightarrow, no head, from=3-5, to=4-5]
	\arrow[curve={height=-24pt}, from=8-1, to=11-1]
	\arrow[Rightarrow, no head, from=3-2, to=3-3]
	\arrow[Rightarrow, no head, from=4-2, to=4-3]
	\arrow[Rightarrow, no head, from=5-2, to=5-3]
	\arrow[Rightarrow, no head, from=6-2, to=6-3]
	\arrow["{(H_{0},H_{1})}", from=8-2, to=8-3]
	\arrow[from=11-2, to=11-3]
	\arrow[Rightarrow, no head, from=9-2, to=9-3]
	\arrow["{([n]\mapsto n+1) \circ \mathrm{Cod}}", from=3-1, to=3-2]
	\arrow["{|\cdot| \circ \mathrm{Cod}}", from=4-1, to=4-2]
	\arrow["{\dim \circ \mathrm{Cod}}", from=5-1, to=5-2]
	\arrow["{M \circ \mathrm{Cod}}", from=6-1, to=6-2]
	\arrow[from=8-1, to=8-2]
	\arrow[from=11-1, to=11-2]
	\arrow["{(\log|\cdot|) \circ \mathrm{Cod}}", from=9-1, to=9-2]
	\arrow[hook, from=3-3, to=3-4]
	\arrow[hook, from=3-4, to=3-5]
	\arrow[hook, from=4-3, to=4-4]
	\arrow[hook, from=4-4, to=4-5]
	\arrow[hook, from=5-3, to=5-4]
	\arrow[hook, from=5-4, to=5-5]
	\arrow[hook, from=6-3, to=6-4]
	\arrow[hook, from=6-4, to=6-5]
	\arrow[hook, from=8-3, to=8-4]
	\arrow[hook, from=8-4, to=8-5]
	\arrow[hook, from=11-3, to=11-4]
	\arrow[Rightarrow, no head, from=11-4, to=11-5]
	\arrow[hook, from=9-3, to=9-4]
	\arrow[hook, from=9-4, to=9-5]
	\arrow[Rightarrow, no head, from=3-4, to=4-4]
	\arrow[Rightarrow, no head, from=4-4, to=5-4]
	\arrow[from=5-4, to=6-4]
	\arrow[from=6-4, to=8-4]
	\arrow[Rightarrow, no head, from=3-2, to=4-2]
	\arrow[Rightarrow, no head, from=4-2, to=5-2]
	\arrow[Rightarrow, no head, from=3-3, to=4-3]
	\arrow[Rightarrow, no head, from=4-3, to=5-3]
	\arrow[from=5-2, to=6-2]
	\arrow[from=5-3, to=6-3]
	\arrow[from=6-2, to=8-2]
	\arrow[from=6-3, to=8-3]
	\arrow[curve={height=6pt}, from=8-2, to=9-2]
	\arrow["{\pi_1}"', curve={height=6pt}, from=8-3, to=9-3]
	\arrow["{\pi_1}"', curve={height=6pt}, from=8-4, to=9-4]
	\arrow["{\pi_1}"', curve={height=6pt}, from=8-5, to=9-5]
	\arrow[curve={height=-24pt}, from=8-2, to=11-2]
	\arrow["{\pi_2}"{description, pos=0.6}, curve={height=-24pt}, from=8-3, to=11-3]
	\arrow["{\pi_2}"{description, pos=0.6}, curve={height=-24pt}, from=8-4, to=11-4]
	\arrow["{\pi_2}"{description, pos=0.6}, curve={height=-24pt}, from=8-5, to=11-5]
	\arrow[hook', from=1-2, to=1-1]
	\arrow[hook', from=1-3, to=1-2]
	\arrow[hook', from=1-4, to=1-3]
	\arrow[hook', from=1-5, to=1-4]
	\arrow["{H_1 \circ \mathrm{Cod}}"{description, pos=0.3}, curve={height=-30pt}, from=11-1, to=11-3]
	\arrow["{(H_0,H_1) \circ \mathrm{Cod}}"{description, pos=0.2}, curve={height=-30pt}, from=8-1, to=8-3]
\end{tikzcd}\]}{\footnotesize\par}

\begin{figure}
\begin{centering}
\includegraphics[scale=0.83]{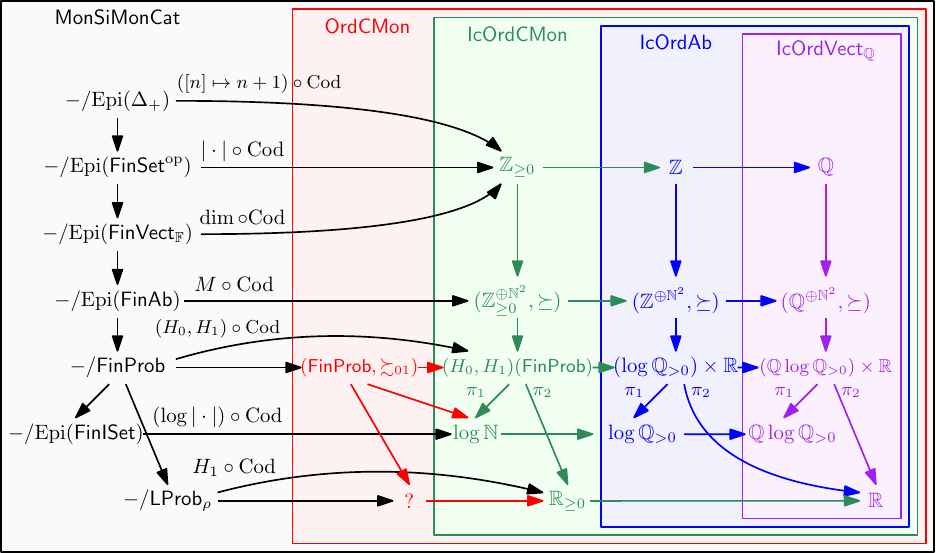}
\par\end{centering}
\caption{\label{fig:reflect2}The universal $\mathsf{V}$-entropies of various
categories for various $\mathsf{V}$'s discussed in this paper, placed
in a single commutative diagram. The \textcolor{purple}{purple} objects
and morphisms in the \textcolor{purple}{purple} rectangle are in the
category \textcolor{purple}{$\mathsf{IcOrdVect}_{\mathbb{Q}}$}. The
\textcolor{blue}{blue} and \textcolor{purple}{purple} objects and
morphisms are in \textcolor{blue}{$\mathsf{IcOrdAb}$}. The \textcolor{teal}{green},
\textcolor{blue}{blue} and \textcolor{purple}{purple} objects and
morphisms are in \textcolor{teal}{$\mathsf{IcOrdCMon}$}. The \textcolor{red}{red},
\textcolor{teal}{green}, \textcolor{blue}{blue} and \textcolor{purple}{purple}
objects and morphisms are in \textcolor{red}{$\mathsf{OrdCMon}$}.
All objects and morphisms are in $\mathsf{MonSiMonCat}$. The universal
$\mathsf{V}$-entropy of $\mathsf{C}\in\mathsf{MonSiMonCat}$ (where
$\mathsf{V}$ is $\mathsf{OrdCMon}$, $\mathsf{IcOrdCMon}$, ...)
is the arrow from $\mathsf{C}$ to its reflection in $\mathsf{V}\subseteq\mathsf{MonSiMonCat}$.}
\end{figure}

\medskip{}

\section{Conditional Entropy\label{sec:conditional}}

\subsection{Discretely-Indexed Monoidal Strictly-Indexed Monoidal Categories}

For $Q\in\mathsf{FinProb}$, the over category $\mathsf{FinProb}/Q$
can be regarded as ``the category of finite conditional probability
spaces conditioned on $Q$'', where a ``conditional probability
space'' is in the form $f:P\to Q$.  It was noted in \cite{baez2011entropy}
that $\mathsf{FinProb}/Q$ is monoidal, with tensor product given
by the conditional product of distributions \cite{dawid1999conditional},
where $f_{1}\otimes f_{2}$ for $f_{1}:P_{1}\to Q$ (where $P_{1}$
is a distribution over $A_{1}$, and $Q$ is a distribution over $B$)
and $f_{2}:P_{2}\to Q$ (where $P_{2}$ is a distribution over $A_{2}$)
is given by $f:P\to Q$, where $P$ is a distribution over $\{(x_{1},x_{2})\in A_{1}\times A_{2}:\,f_{1}(x_{1})=f_{2}(x_{2})\}$
with 
\[
P(x_{1},x_{2})=\frac{P_{1}(x_{1})P_{2}(x_{2})}{Q(f_{1}(x_{1}))},
\]
and $f$ is the measure-preserving function that maps $(x_{1},x_{2})$
to $f_{1}(x_{1})$. It was observed in \cite{baez2011entropy} that
under categories of $\mathsf{FinProb}/Q$ admits categorical products.
For $f:P\to Q$, the under category $f/(\mathsf{FinProb}/Q)$ consists
of objects in the form $P\stackrel{x}{\to}A\stackrel{a}{\to}Q$ where
$ax=f$, i.e., objects are commuting triangles
\[
\begin{array}{ccccc}
 &  & A\\
 & \!\!\!\stackrel{x}{\nearrow}\!\!\! &  & \!\!\!\stackrel{a}{\searrow}\!\!\!\\
P\! &  & \!\!\stackrel{f}{\longrightarrow}\!\! &  & \!Q
\end{array}
\]
which can be regarded as a ``conditional random variable conditioned
on $f$'', i.e., $x$ is a random variable over the probability space
$P$, where the random variable $f$ must be a function of $x$. The
categorical product over $f/(\mathsf{FinProb}/Q)$ corresponds to
taking the joint random variable. The ``conditional Shannon entropy
functor'' $\tilde{H}_{1}:f/(\mathsf{FinProb}/Q)\to[0,\infty)$ where
\[
\tilde{H}_{1}\left(\begin{array}{ccccc}
 &  & A\\
 & \!\!\!\stackrel{x}{\nearrow}\!\!\! &  & \!\!\!\stackrel{a}{\searrow}\!\!\!\\
P\! &  & \!\!\stackrel{f}{\longrightarrow}\!\! &  & \!Q
\end{array}\right)=H_{1}(x|f)=H_{1}(A)-H_{1}(Q)
\]
is the conditional Shannon entropy of the random variable $x$ conditional
on the random variable $f$, was observed to be a lax monoidal functor
in \cite{baez2011entropy}. This is a consequence of the subadditivity
of conditional Shannon entropy:
\begin{equation}
H_{1}(X|Z)+H_{1}(Y|Z)\ge H_{1}(X,Y|Z)\label{eq:subadd}
\end{equation}
for any finite jointly distributed random variables $X,Y,Z$. 

The above observations show that $\mathsf{FinProb}/Q$ shares many
properties as $\mathsf{FinProb}$. In fact, we can check that $(\mathsf{FinProb}/Q,\,-/(\mathsf{FinProb}/Q))$
is a MonSiMon category as well, with a universal baseless $\mathsf{IcOrdAb}$-entropy
that maps $f:P\to Q$ to 
\[
(H_{0}(X|Y=y),\,H_{1}(X|Y=y))_{y\in\mathrm{supp}(Q)}\in((\log\mathbb{Q}_{>0})\times\mathbb{R})^{|\mathrm{supp}(Q)|},
\]
where $X$ is a random variable that follows $P$, $Y=f(X)$, $\mathrm{supp}(Q)$
denotes the support of the distribution $Q$, and $H_{1}(X|Y=y)=H_{1}(P_{X|Y=y})$
is the entropy of the conditional distribution $P_{X|Y=y}$. Here
we simply concatenate the Shannon and Hartley entropies conditional
on each value of $Y$. Different values of $Y$ are treated separately
and do not interact with each other. Although we can obtain the conditional
entropy $H_{1}(X|Y)=\sum_{y}Q(y)H_{1}(X|Y=y)$ as a linear function
of $(H_{0}(X|Y=y),\,H_{1}(X|Y=y))_{y}$, we are interested in a characterization
of the (conditional) Shannon entropy as the only entropy satisfying
a certain universal property, excluding other linear functions of
$(H_{0}(X|Y=y),\,H_{1}(X|Y=y))_{y}$. To this end, we should not find
the entropy of each $\mathsf{FinProb}/Q$ separately, but find a way
to link the entropies of $\mathsf{FinProb}/Q$ for different $Q$
together. This can be done via the chain rule
\begin{equation}
H_{1}(X|Y)+H_{1}(Y|Z)=H_{1}(X|Z)\label{eq:chain}
\end{equation}
for any random variables $X,Y,Z$ such that $Y$ is a function of
$X$, and $Z$ is a function of $Y$. This is referred as the functoriality
property in \cite{baez2011characterization}. 

One might treat $\mathsf{FinProb}$ as an ``indexed MonSiMon category''
since each of its over category is a MonSiMon category. Unfortunately,
although the over category functor $\mathsf{FinProb}/-:\,\mathsf{FinProb}\to\mathsf{Cat}$
is a functor, it cannot be upgraded to a $\mathsf{FinProb}\to\mathsf{MonSiMonCat}$
functor since there is generally no monoidal functor from $\mathsf{FinProb}/X$
to $\mathsf{FinProb}/Y$ that can be obtained from the morphism $f:X\to Y$
(the conditional product does not interact well with the change of
the conditioned distribution). Though it can be considered as a $\mathsf{FinProb}\to\mathsf{SiCat}$
functor, where $\mathsf{SiCat}=\mathbb{U}\mathsf{Cat}/\!/\mathsf{Cat}$
is like $\mathsf{MonSiMonCat}$ but with all monoidal structures removed.
Therefore, we require two functors, one of them is $\mathsf{FinProb}\to\mathsf{SiCat}$,
and the other is $\mathrm{Dis}(\mathsf{C})\to\mathsf{MonSiMonCat}$
(where $\mathrm{Dis}(\mathsf{C})$ is the discrete category with the
same objects as $\mathsf{C}$). We now give the definition of ``discretely-indexed
MonSiMon categories'' by linking these two functors.

\medskip{}

\begin{defn}
[Discretely-indexed monoidal strictly-indexed monoidal categories]
\label{def:gcmcat-1-1}Assume $\mathsf{Cat}$ contains only small
categories with respect to a Grothendieck universe. Let $\mathbb{U}$
be a Grothendieck universe large enough such that $\mathbb{U}\mathsf{Cat}$,
the 2-category of $\mathbb{U}$-small categories, contains $\mathsf{Cat}$
as an object. Let $\mathsf{SiCat}=\mathbb{U}\mathsf{Cat}/\!/\mathsf{Cat}$
(the \emph{category of strictly-indexed categories}; see Definition
\ref{def:gcmcat_succinct}), with the obvious forgetful functor $U_{\mathrm{MM}}:\mathsf{MonSiMonCat}\to\mathsf{SiCat}$
that forgets the two monoid structures. Let $\mathbb{U}_{2}$ be a
Grothendieck universe large enough such that $\mathbb{U}_{2}\mathsf{Cat}$,
the 2-category of $\mathbb{U}_{2}$-small categories, contains $\mathsf{SiCat}$
and $\mathsf{MonSiMonCat}$ as objects. Consider the strict-strict
arrow $2$-category $2\mathrm{Arr}(\mathbb{U}_{2}\mathsf{Cat})$.\footnote{$2\mathrm{Arr}(\mathbb{U}_{2}\mathsf{Cat})$ is the strict-strict
functor $2$-category from $\mathsf{I}$ to $\mathbb{U}_{2}\mathsf{Cat}$
(where $\mathsf{I}=\{0\to1\}$ is the interval category as a strict
2-category), where objects are strict $2$-functors $\mathsf{I}\to\mathbb{U}_{2}\mathsf{Cat}$;
$1$-cells are strict natural transformations; and $2$-cells are
modifications.} Then the \emph{category of discretely-indexed monoidal strictly-indexed
monoidal (DisiMonSiMon) categories}, denoted as $\mathsf{DisiMonSiMonCat}$,
is given as the lax comma category\footnote{Given a pair of functors $\mathsf{A}\stackrel{F}{\to}\mathsf{C}\stackrel{G}{\leftarrow}\mathsf{B}$
where $\mathsf{A},\mathsf{B},\mathsf{C}$ are strict 2-categories
and $F,G$ are strict 2-functors, the lax comma category is a category
where objects are triples $(A,B,t)$, $A\in\mathsf{A}$, $B\in\mathsf{B}$,
and $t:F(A)\to G(B)$ in $\mathsf{C}$; and a morphism from $(A,B,t)$
to $(A',B',t')$ is a triple $(a,b,\gamma)$, $a:A\to A'$ in $\mathsf{A}$,
$b:B\to B'$ in $\mathsf{B}$, and $\gamma:G(b)t\Rightarrow t'F(a)$
in $\mathsf{C}$.} of the following pair of functors
\[
\mathsf{Cat}\stackrel{\mathrm{Dis}(-)\hookrightarrow-}{\longrightarrow}2\mathrm{Arr}(\mathbb{U}_{2}\mathsf{Cat})\stackrel{U_{\mathrm{MM}}}{\longleftarrow}*
\]
where $\mathrm{Dis}:\mathsf{Cat}\to\mathsf{Cat}$ is the comonad sending
a category to its discrete subcategory (with the same objects and
only the identity morphisms), $\mathrm{Dis}(\mathsf{C})\hookrightarrow\mathsf{C}$
is the obvious embedding, $\mathrm{Dis}(-)\hookrightarrow-$ is the
functor that sends $\mathsf{C}\in\mathsf{Cat}$ to the arrow $\mathrm{Dis}(\mathsf{C})\hookrightarrow\mathsf{C}$
in $2\mathrm{Arr}(\mathbb{U}_{2}\mathsf{Cat})$, and $*\stackrel{U_{\mathrm{MM}}}{\to}2\mathrm{Arr}(\mathbb{U}_{2}\mathsf{Cat})$
is the functor that maps the unique object to the arrow $U_{\mathrm{MM}}$
in $2\mathrm{Arr}(\mathbb{U}_{2}\mathsf{Cat})$. 

\end{defn}

\medskip{}

We now unpack the above definition. An object of $\mathsf{DisiMonSiMonCat}$,
called a \emph{discretely-indexed monoidal strictly-indexed monoidal
(DisiMonSiMon) category}, is a tuple $(\mathsf{C},O,\tilde{O})$,
where $\mathsf{C}\in\mathsf{Cat}$, $O:\mathrm{Dis}(\mathsf{C})\to\mathsf{MonSiMonCat}$
(a morphism in $\mathbb{U}\mathsf{Cat}$), and $\tilde{O}:\mathsf{C}\to\mathsf{SiCat}$
(a morphism in $\mathbb{U}_{2}\mathsf{Cat}$) such that the following
commutes:\\
\[\begin{tikzcd}
	{\mathrm{Dis}(\mathsf{C})} & {\mathsf{C}} \\
	{\mathsf{MonSiMonCat}} & {\mathsf{SiCat}}
	\arrow[hook, from=1-1, to=1-2]
	\arrow["{U_{\mathrm{MM}}}", from=2-1, to=2-2]
	\arrow["O"', from=1-1, to=2-1]
	\arrow["{\tilde{O}}", from=1-2, to=2-2]
\end{tikzcd}\]Given DisiMonSiMon categories $(\mathsf{C}_{1},O_{1},\tilde{O}_{1})$,
$(\mathsf{C}_{2},O_{2},\tilde{O}_{2})$, a morphism from the former
to the later in $\mathsf{DisiMonSiMonCat}$ (called a \emph{DisiMonSiMon
functor)} is given by a tuple $(F,\omega,\tilde{\omega})$, where
$F:\mathsf{C}_{1}\to\mathsf{C}_{2}$, and $\omega:O_{1}\Rightarrow O_{2}\mathrm{Dis}(F)$,
$\tilde{\omega}:\tilde{O}_{1}\Rightarrow\tilde{O}_{2}F$ are natural
transformation such that the following two compositions are the same:\\
\[\begin{tikzcd}[column sep=small]
	{\mathrm{Dis}(\mathsf{C}_1)} & {\mathrm{Dis}(\mathsf{C}_2)} & {\mathsf{C}_2} && {\mathrm{Dis}(\mathsf{C}_1)} & {\mathsf{C}_1} & {\mathsf{C}_2} \\
	&&& {=} \\
	& {\mathsf{MonSiMonCat}} & {\mathsf{SiCat}} && {\mathsf{MonSiMonCat}} & {\mathsf{SiCat}}
	\arrow["{U_{\mathrm{MM}}}", from=3-2, to=3-3]
	\arrow[""{name=0, anchor=center, inner sep=0}, "{O_1}"', from=1-1, to=3-2]
	\arrow["{\tilde{O}_2}", from=1-3, to=3-3]
	\arrow["{\mathrm{Dis}(F)}", from=1-1, to=1-2]
	\arrow[hook, from=1-2, to=1-3]
	\arrow["{O_2}", from=1-2, to=3-2]
	\arrow[hook, from=1-5, to=1-6]
	\arrow["{U_{\mathrm{MM}}}", from=3-5, to=3-6]
	\arrow["{O_1}"', from=1-5, to=3-5]
	\arrow[""{name=1, anchor=center, inner sep=0}, "{\tilde{O}_1}"', from=1-6, to=3-6]
	\arrow["F", from=1-6, to=1-7]
	\arrow["{\tilde{O}_2}", from=1-7, to=3-6]
	\arrow["{\mathrm{id}}"{description}, shorten <=9pt, shorten >=9pt, Rightarrow, from=3-2, to=1-3]
	\arrow["{\mathrm{id}}"{description}, shorten <=9pt, shorten >=9pt, Rightarrow, from=3-5, to=1-6]
	\arrow["\omega", shorten <=5pt, Rightarrow, from=0, to=1-2]
	\arrow["{\tilde{\omega}}", shorten <=6pt, Rightarrow, from=1, to=1-7]
\end{tikzcd}\]Composition of DisiMonSiMon functors $(\mathsf{C}_{1},O_{1},\tilde{O}_{1})\stackrel{(F_{1},\omega_{1},\tilde{\omega}_{1})}{\longrightarrow}(\mathsf{C}_{2},O_{2},\tilde{O}_{2})\stackrel{(F_{2},\omega_{2},\tilde{\omega}_{2})}{\longrightarrow}(\mathsf{C}_{3},O_{3},\tilde{O}_{3})$
is given by $(F_{2}F_{1},\,\omega_{2}\mathrm{Dis}(F_{1})\circ\omega_{1},\,\tilde{\omega}_{2}F_{1}\circ\tilde{\omega}_{1})$.
\[\begin{tikzcd}[column sep=small]
	{\mathrm{Dis}(\mathsf{C}_1)} & {\mathrm{Dis}(\mathsf{C}_2)} & {\mathrm{Dis}(\mathsf{C}_3)} && {\mathsf{C}_1} & {\mathsf{C}_2} & {\mathsf{C}_3} \\
	&&&& \qquad\qquad && \qquad\qquad \\
	& {\mathsf{MonSiMonCat}} &&&& {\mathsf{SiCat}}
	\arrow["{F_1}", from=1-5, to=1-6]
	\arrow[""{name=0, anchor=center, inner sep=0}, "{\tilde{O}_{1}}"', from=1-5, to=3-6]
	\arrow[""{name=1, anchor=center, inner sep=0}, "{\tilde{O}_{2}}"{description}, from=1-6, to=3-6]
	\arrow["{F_2}", from=1-6, to=1-7]
	\arrow["{\tilde{O}_{3}}", from=1-7, to=3-6]
	\arrow["{\mathrm{Dis}(F_1)}", from=1-1, to=1-2]
	\arrow["{\mathrm{Dis}(F_2)}", from=1-2, to=1-3]
	\arrow[""{name=2, anchor=center, inner sep=0}, "{O_{1}}"', from=1-1, to=3-2]
	\arrow[""{name=3, anchor=center, inner sep=0}, "{O_{2}}"{description}, from=1-2, to=3-2]
	\arrow["{O_{3}}", from=1-3, to=3-2]
	\arrow["{\tilde{\omega}_2}", shorten <=7pt, shorten >=3pt, Rightarrow, from=1, to=1-7]
	\arrow["{\tilde{\omega}_1}", shorten <=4pt, shorten >=2pt, Rightarrow, from=0, to=1-6]
	\arrow["{\omega_1}", shorten <=5pt, shorten >=2pt, Rightarrow, from=2, to=1-2]
	\arrow["{\omega_2}", shorten <=8pt, shorten >=4pt, Rightarrow, from=3, to=1-3]
\end{tikzcd}\]

\medskip{}

An example of DisiMonSiMon category is the ``DisiMonSiMon category
of finite conditional probability spaces'', denoted as $\mathsf{FinCondProb}$,
with: 
\begin{itemize}
\item $\mathsf{C}=\mathsf{FinProb}$.
\item $O:\mathrm{Dis}(\mathsf{FinProb})\to\mathsf{MonSiMonCat}$ sends $Q\in\mathsf{FinProb}$
to $O(Q)=(\mathsf{FinProb}/Q,\,-/(\mathsf{FinProb}/Q))$, the ``MonSiMon
category of conditional probability spaces conditioned on $Q$'',
with tensor product given by the conditional product.
\item $\tilde{O}:\mathsf{FinProb}\to\mathsf{SiCat}$ defined as:
\begin{itemize}
\item For object $Q\in\mathsf{FinProb}$, $\tilde{O}(Q)=U_{\mathrm{MM}}(O(Q))$.
\item For morphism $f:Q\to Q'$, $\tilde{O}(f)=(F,\gamma)$ is a GC functor
(morphism in $\mathsf{SiCat}$, defined in a similar way as MonSiMon
functor) where: 
\begin{itemize}
\item $F:\mathsf{FinProb}/Q\to\mathsf{FinProb}/Q'$ is the postcomposition
functor by $f$, a ``functor sending a conditional probability space
conditioned on $Q$ to that conditioned on $Q'$''.
\item The natural transformation 
\[
\gamma:-/(\mathsf{FinProb}/Q)\Rightarrow F^{\mathrm{op}}(-)/(\mathsf{FinProb}/Q')
\]
 has component 
\[
\gamma_{g}:g/(\mathsf{FinProb}/Q)\to fg/(\mathsf{FinProb}/Q')
\]
 for $g:P\to Q$ given by the postcomposition by $f$:
\begin{align*}
\gamma_{g}\left(\begin{array}{ccccc}
 &  & A\\
 & \!\!\!\stackrel{x}{\nearrow}\!\!\! &  & \!\!\!\stackrel{a}{\searrow}\!\!\!\\
P\! &  & \!\!\stackrel{g}{\longrightarrow}\!\! &  & \!Q
\end{array}\right) & \;=\;\left(\begin{array}{ccccc}
 &  & A\\
 & \!\!\!\stackrel{x}{\nearrow}\!\!\! &  & \!\!\!\stackrel{fa}{\searrow}\!\!\!\\
P\! &  & \!\!\stackrel{fg}{\longrightarrow}\!\! &  & \!Q'
\end{array}\right),
\end{align*}
which is a ``functor sending a conditional random variable conditioned
on $Q$ to that conditioned on $Q'$''.
\end{itemize}
\end{itemize}
\end{itemize}
\medskip{}

Note that $\gamma_{g}$ is not a monoidal functor, and $\gamma$ is
not a monoidal natural transformation, which is fine since functoriality
is only required for $\tilde{O}:\mathsf{FinProb}\to\mathsf{SiCat}$
(where $\mathsf{SiCat}$ is like $\mathsf{MonSiMonCat}$ but without
any monoidal structure), not for $O:\mathrm{Dis}(\mathsf{FinProb})\to\mathsf{MonSiMonCat}$
(since the domain is discrete). 

\medskip{}

\subsection{Embedding $\mathsf{OrdCMon}$ into $\mathsf{DisiMonSiMonCat}$}

Similar to how we embed $\mathsf{OrdCMon}$ into $\mathsf{MonSiMonCat}$,
we can also embed it into $\mathsf{DisiMonSiMonCat}$.

\medskip{}

\begin{defn}
[Embedding $\mathsf{OrdCMon}$ into $\mathsf{DisiMonSiMonCat}$]
\label{def:gcmcat-1-1-1-1}Define a functor $R_{\mathrm{DiM}}:\mathsf{OrdCMon}\to\mathsf{DisiMonSiMonCat}$
which maps the object $\mathsf{W}\in\mathsf{OrdCMon}$ to $(\mathbf{B}(\mathsf{W}),\,R(\mathsf{W}),\,\mathbf{A}_{\mathsf{W}})\in\mathsf{DisiMonSiMonCat}$,
\[\begin{tikzcd}
	{*} & {\mathbf{B}(\mathsf{W})} \\
	{\mathsf{MonSiMonCat}} & {\mathsf{SiCat}}
	\arrow[hook, from=1-1, to=1-2]
	\arrow["{U_{\mathrm{MM}}}", from=2-1, to=2-2]
	\arrow["{R(\mathsf{W})}"', from=1-1, to=2-1]
	\arrow["{\mathbf{A}_{\mathsf{W}}}", from=1-2, to=2-2]
\end{tikzcd}\]where 
\begin{itemize}
\item $\mathbf{B}:\mathsf{CMon}\to\mathsf{Cat}$ is the delooping functor
sending the commutative monoid $M\in\mathsf{CMon}$ to the category
$\mathbf{B}(M)\in\mathsf{Cat}$ with one object $\bullet$ with a
monoid of endomorphisms $\mathrm{hom}_{\mathbf{B}(M)}(\bullet,\bullet)=M$,
and sending a monoid homomorphism $f:M_{1}\to M_{2}$ to the functor
sending the endomorphism $x\in\mathrm{hom}_{\mathbf{B}(M_{1})}(\bullet,\bullet)=M_{1}$
to $f(x)\in\mathrm{hom}_{\mathbf{B}(M_{2})}(\bullet,\bullet)=M_{2}$.\footnote{We write $\mathbf{B}(\mathsf{W})=\mathbf{B}(U(\mathsf{W}))$ where
$U:\mathsf{OrdCMon}\to\mathsf{CMon}$ is the forgetful functor. We
omit $U$ for brevity.} 
\item $R(\mathsf{W})$ denotes the functor $*\to\mathsf{MonSiMonCat}$ sending
the unique object to $R(\mathsf{W})\in\mathsf{MonSiMonCat}$ (see
Lemma \ref{lem:lgcm}).
\item $\mathbf{A}_{\mathsf{W}}:\mathbf{B}(\mathsf{W})\to\mathsf{SiCat}$
is the ``multiplication action'' which sends the only object to
$U_{\mathrm{MM}}(R(\mathsf{W}))$, and the morphism $x\in\mathsf{W}$
(recall that $\mathrm{hom}_{\mathbf{B}(M)}(\bullet,\bullet)$ is $\mathsf{W}$
as a monoid) to the morphism $(!,\gamma)$ in $\mathsf{SiCat}$, where
$!:*\to*$ is the unique functor, and $\gamma:(*\stackrel{\mathsf{W}}{\to}\mathsf{Cat})\Rightarrow(*\stackrel{\mathsf{W}}{\to}\mathsf{Cat})$
has component $\gamma_{\bullet}:\mathsf{W}\to\mathsf{W}$ given by
left multiplication $\gamma_{\bullet}=x\otimes-$.
\end{itemize}
For a morphism $F:\mathsf{W}_{1}\to\mathsf{W}_{2}$ in $\mathsf{OrdCMon}$,
we take $R_{\mathrm{DiM}}(F)=(\mathbf{B}(F),R(F),\tilde{\omega})$:\\
\[\begin{tikzcd}[column sep=small]
	{*} & {*} & {\mathbf{B}(\mathsf{W}_2)} && {*} & {\mathbf{B}(\mathsf{W}_1)} & {\mathbf{B}(\mathsf{W}_2)} \\
	&&& {=} \\
	& {\mathsf{MonSiMonCat}} & {\mathsf{SiCat}} && {\mathsf{MonSiMonCat}} & {\mathsf{SiCat}}
	\arrow["{U_{\mathrm{MM}}}", from=3-2, to=3-3]
	\arrow[""{name=0, anchor=center, inner sep=0}, "{R(\mathsf{W}_1)}"', from=1-1, to=3-2]
	\arrow["{\mathbf{A}_{\mathsf{W}_2}}", from=1-3, to=3-3]
	\arrow[from=1-1, to=1-2]
	\arrow[hook, from=1-2, to=1-3]
	\arrow["{R(\mathsf{W}_2)}"{description}, from=1-2, to=3-2]
	\arrow[hook, from=1-5, to=1-6]
	\arrow["{U_{\mathrm{MM}}}", from=3-5, to=3-6]
	\arrow["{R(\mathsf{W}_1)}"', from=1-5, to=3-5]
	\arrow[""{name=1, anchor=center, inner sep=0}, "{\mathbf{A}_{\mathsf{W}_1}}"{description}, from=1-6, to=3-6]
	\arrow["{\mathbf{B}(F)}", from=1-6, to=1-7]
	\arrow["{\mathbf{A}_{\mathsf{W}_2}}", from=1-7, to=3-6]
	\arrow["{\mathrm{id}}"{description}, shorten <=9pt, shorten >=9pt, Rightarrow, from=3-2, to=1-3]
	\arrow["{\mathrm{id}}"{description}, shorten <=9pt, shorten >=9pt, Rightarrow, from=3-5, to=1-6]
	\arrow["{R(F)}"{description}, shorten <=5pt, Rightarrow, from=0, to=1-2]
	\arrow["{\tilde{\omega}}", shorten <=6pt, Rightarrow, from=1, to=1-7]
\end{tikzcd}\]where $\tilde{\omega}:\mathbf{A}_{\mathsf{W}_{1}}\Rightarrow\mathbf{A}_{\mathsf{W}_{2}}\mathbf{B}(F)$
with component $\tilde{\omega}_{\bullet}:U_{\mathrm{MM}}(R(\mathsf{W}_{1}))\to U_{\mathrm{MM}}(R(\mathsf{W}_{2}))$
being a morphism $(!,\gamma)$ in $\mathsf{SiCat}$, where $\gamma:(*\stackrel{\mathsf{W}_{1}}{\to}\mathsf{Cat})\Rightarrow(*\stackrel{\mathsf{W}_{2}}{\to}\mathsf{Cat})$
has component $\gamma_{\bullet}=F:\mathsf{W}_{1}\to\mathsf{W}_{2}$.
\end{defn}

\medskip{}

\subsection{Chain Rule}

Similar to how Section \ref{sec:univ} showed that a MonSiMon functor
from $\mathsf{C}\in\mathsf{MonSiMonCat}$ to $\mathsf{W}\in\mathsf{OrdCMon}$
must satisfy the monotonicity, additivity and subadditivity properties,
we now show that a DisiMonSiMon functor $(F,\omega,\tilde{\omega})$
from $(\mathsf{C},O,\tilde{O})\in\mathsf{DisiMonSiMonCat}$ to $R_{\mathrm{DiM}}(\mathsf{W})$
(where $\mathsf{W}\in\mathsf{OrdCMon}$) must satisfy the chain rule
\eqref{eq:chain}. Denote the tensor product and the tensor unit of
$\mathsf{W}$ by $+$ and $0$ respectively. By the definition of
DisiMonSiMon functor, the following two compositions are the same:\\
\[\begin{tikzcd}[column sep=small]
	{\mathrm{Dis}(\mathsf{C})} & {*} & {\mathbf{B}(\mathsf{W})} && {\mathrm{Dis}(\mathsf{C})} & {\mathsf{C}} & {\mathbf{B}(\mathsf{W})} \\
	&&& {=} \\
	& {\mathsf{MonSiMonCat}} & {\mathsf{SiCat}} && {\mathsf{MonSiMonCat}} & {\mathsf{SiCat}}
	\arrow["{U_{\mathrm{MM}}}", from=3-2, to=3-3]
	\arrow[""{name=0, anchor=center, inner sep=0}, "O"', from=1-1, to=3-2]
	\arrow["{\mathbf{A}_\mathsf{W}}", from=1-3, to=3-3]
	\arrow["{\mathrm{Dis}(F)}", from=1-1, to=1-2]
	\arrow[hook, from=1-2, to=1-3]
	\arrow["{R(\mathsf{W})}"{description}, from=1-2, to=3-2]
	\arrow[hook, from=1-5, to=1-6]
	\arrow["{U_{\mathrm{MM}}}", from=3-5, to=3-6]
	\arrow["O"', from=1-5, to=3-5]
	\arrow[""{name=1, anchor=center, inner sep=0}, "{\tilde{O}}"', from=1-6, to=3-6]
	\arrow["F", from=1-6, to=1-7]
	\arrow["{\mathbf{A}_\mathsf{W}}", from=1-7, to=3-6]
	\arrow["{\mathrm{id}}"{description}, shorten <=9pt, shorten >=9pt, Rightarrow, from=3-2, to=1-3]
	\arrow["{\mathrm{id}}"{description}, shorten <=9pt, shorten >=9pt, Rightarrow, from=3-5, to=1-6]
	\arrow["\omega", shorten <=5pt, Rightarrow, from=0, to=1-2]
	\arrow["{\tilde{\omega}}", shorten <=6pt, Rightarrow, from=1, to=1-7]
\end{tikzcd}\] Hence, for $P\in\mathsf{C}$,
\begin{equation}
U_{\mathrm{MM}}(\omega_{P})=\tilde{\omega}_{P}.\label{eq:omegatilde}
\end{equation}
By the naturality of $\tilde{\omega}$, for $g:P\to Q$ in $\mathsf{C}$,
we have 
\begin{equation}
\mathbf{A}_{\mathsf{W}}(F(g))\tilde{\omega}_{P}=\tilde{\omega}_{Q}\tilde{O}(g).\label{eq:condent_nat}
\end{equation}
Let the MonSiMon category $O(P)$ be $(\mathsf{D}_{P},N_{P})$ (same
for $O(Q)$). Let the GC functor $\tilde{\omega}_{P}:\tilde{O}(P)\to\mathbf{A}_{\mathsf{W}}(F(P))$
be $(!,\tilde{h}_{P})$, where $\tilde{h}_{P}:N_{P}\Rightarrow\Delta_{\mathsf{W}}$.
Let the GC functor $\tilde{O}(g)$ be $(G,\gamma)$, $G:\mathsf{D}_{P}\to\mathsf{D}_{Q}$,
$\gamma:N_{P}\Rightarrow N_{Q}G^{\mathrm{op}}$. Equation \eqref{eq:condent_nat}
becomes\\
\[\begin{tikzcd}
	{\mathsf{D}_{P}^\mathrm{op}} & {*} &&& {*} && {\mathsf{D}_{P}^\mathrm{op}} & {\mathsf{D}_{Q}^\mathrm{op}} & {*} \\
	&&&&& {=} \\
	& {\mathsf{Cat}} &&&&&& {\mathsf{Cat}}
	\arrow[""{name=0, anchor=center, inner sep=0}, "{N_P}"', from=1-7, to=3-8]
	\arrow[""{name=1, anchor=center, inner sep=0}, "{N_Q}"{description}, from=1-8, to=3-8]
	\arrow["{G^\mathrm{op}}", from=1-7, to=1-8]
	\arrow["{!}", from=1-8, to=1-9]
	\arrow["{\mathsf{W}}", from=1-9, to=3-8]
	\arrow[""{name=2, anchor=center, inner sep=0}, "{N_P}"', from=1-1, to=3-2]
	\arrow["{!}", from=1-1, to=1-2]
	\arrow["{!}", from=1-2, to=1-5]
	\arrow[""{name=3, anchor=center, inner sep=0}, "{\mathsf{W}}"{description}, from=1-2, to=3-2]
	\arrow["{\mathsf{W}}", from=1-5, to=3-2]
	\arrow["\gamma", shorten <=4pt, shorten >=4pt, Rightarrow, from=0, to=1-8]
	\arrow["{h_Q}", shorten <=6pt, shorten >=6pt, Rightarrow, from=1, to=1-9]
	\arrow["{\tilde{h}_{P}}", shorten <=4pt, shorten >=4pt, Rightarrow, from=2, to=1-2]
	\arrow["{F(g)+-}"{pos=0.6}, shorten <=17pt, shorten >=17pt, Rightarrow, from=3, to=1-5]
\end{tikzcd}\]

For $E\in\mathsf{D}_{P}$, $X\in N_{P}(E)$, we have
\[
F(g)+\tilde{h}_{P,E}(X)=\tilde{h}_{Q,G(E)}(\gamma_{E}(X)).
\]
Let the MonSiMon functor $\omega_{P}:O(P)\to R(\mathsf{W})$ be $(!,h_{P})$,
where $h_{P}:N_{P}\Rightarrow\Delta_{\mathsf{W}}$ is a monoidal natural
transformation. By \eqref{eq:omegatilde}, we also have
\begin{equation}
F(g)+h_{P,E}(X)=h_{Q,G(E)}(\gamma_{E}(X)).\label{eq:chain_gen}
\end{equation}

Taking $E=I_{\mathsf{D}_{P}}$ and $X=\epsilon_{P,\bullet}$ where
$\epsilon_{P}:*\to N(I_{\mathsf{D}_{P}})$ is one of the coherence
maps of the lax monoidal functor $N$, we have $h_{P,E}(X)=0$ since
$h_{P}$ is a monoidal natural transformation, and hence by \eqref{eq:chain_gen},
$F(g)=h_{Q,G(I_{\mathsf{D}_{P}})}(\gamma_{I_{\mathsf{D}_{P}}}(\epsilon_{P,\bullet}))$.
Substituting back to \eqref{eq:chain_gen},
\begin{equation}
h_{Q,G(I_{\mathsf{D}_{P}})}(\gamma_{I_{\mathsf{D}_{P}}}(\epsilon_{P,\bullet}))+h_{P,E}(X)=h_{Q,G(E)}(\gamma_{E}(X)).\label{eq:chain_sub}
\end{equation}
Here $h_{P,E}(X)$ should be interpreted as ``the conditional entropy
of the random variable $X\in N_{P}(E)$ conditioned on $P$''.

We now consider the DisiMonSiMon category $\mathsf{FinCondProb}$
discussed earlier. For a chain of morphisms $K\stackrel{f}{\to}P\stackrel{g}{\to}Q$
in $\mathsf{FinProb}$, taking $X=\mathrm{id}_{K}$, \eqref{eq:chain_sub}
gives
\begin{equation}
h_{Q,g}(\mathrm{id}_{P})+h_{P,f}(\mathrm{id}_{K})=h_{Q,gf}(\mathrm{id}_{K}),\label{eq:chain_finprob}
\end{equation}
or informally, ``$H(P|Q)+H(K|P)=H(K|Q)$'', which is basically the
chain rule \eqref{eq:chain}. 

\medskip{}

\subsection{$\mathsf{OrdCMon}$ as a Reflective Subcategory of $\mathsf{DisiMonSiMonCat}$}

Similar to Lemma \ref{lem:lgcm}, we can show that $R_{\mathrm{DiM}}$
is fully faithful and has a left adjoint. Therefore, $\mathsf{OrdCMon}$
is a reflective subcategory of $\mathsf{DisiMonSiMonCat}$. \medskip{}

\begin{lem}
\label{lem:lgcm-1}$R_{\mathrm{DiM}}:\mathsf{OrdCMon}\to\mathsf{DisiMonSiMonCat}$
is fully faithful, and has a left adjoint $L_{\mathrm{DiM}}:\mathsf{DisiMonSiMonCat}\to\mathsf{OrdCMon}$.
Hence, $\mathsf{OrdCMon}$ is a reflective subcategory of $\mathsf{DisiMonSiMonCat}$.
\end{lem}

\begin{proof}
First check that $R_{\mathrm{DiM}}$ is fully faithful. Consider
a DisiMonSiMon functor $(F,\omega,\tilde{\omega}):R_{\mathrm{DiM}}(\mathsf{W}_{1})\to R_{\mathrm{DiM}}(\mathsf{W}_{2})$,
where $\mathsf{W}_{1},\mathsf{W}_{2}\in\mathsf{OrdCMon}$. We have
the following (omitting obvious embedding functors):\\
\[\begin{tikzcd}[column sep=small]
	{*} & {*} & {\mathbf{B}(\mathsf{W}_2)} && {*} & {\mathbf{B}(\mathsf{W}_1)} & {\mathbf{B}(\mathsf{W}_2)} \\
	&&& {=} \\
	& {\mathsf{MonSiMonCat}} & {\mathsf{SiCat}} && {\mathsf{MonSiMonCat}} & {\mathsf{SiCat}}
	\arrow["{U_{\mathrm{MM}}}", from=3-2, to=3-3]
	\arrow[""{name=0, anchor=center, inner sep=0}, "{\mathsf{W}_1}"', from=1-1, to=3-2]
	\arrow["{\mathbf{A}_{\mathsf{W}_2}}", from=1-3, to=3-3]
	\arrow[from=1-1, to=1-2]
	\arrow[hook, from=1-2, to=1-3]
	\arrow["{\mathsf{W}_2}", from=1-2, to=3-2]
	\arrow[hook, from=1-5, to=1-6]
	\arrow["{U_{\mathrm{MM}}}", from=3-5, to=3-6]
	\arrow["{\mathsf{W}_1}"', from=1-5, to=3-5]
	\arrow[""{name=1, anchor=center, inner sep=0}, "{\mathbf{A}_{\mathsf{W}_1}}"', from=1-6, to=3-6]
	\arrow["{\mathbf{B}(F)}", from=1-6, to=1-7]
	\arrow["{\mathbf{A}_{\mathsf{W}_2}}", from=1-7, to=3-6]
	\arrow["{\mathrm{id}}"{description}, shorten <=9pt, shorten >=9pt, Rightarrow, from=3-2, to=1-3]
	\arrow["{\mathrm{id}}"{description}, shorten <=9pt, shorten >=9pt, Rightarrow, from=3-5, to=1-6]
	\arrow["\omega", shorten <=5pt, Rightarrow, from=0, to=1-2]
	\arrow["{\tilde{\omega}}", shorten <=6pt, Rightarrow, from=1, to=1-7]
\end{tikzcd}\]

Hence, $U_{\mathrm{MM}}(\omega_{\bullet})=\tilde{\omega}_{\bullet}$.
Let the GC functor $\tilde{\omega}_{\bullet}:U_{\mathrm{MM}}(R(\mathsf{W}_{1}))\to U_{\mathrm{MM}}(R(\mathsf{W}_{2}))$
be $(!,\tilde{\gamma})$, where $\tilde{\gamma}:\Delta_{\mathsf{W}_{1}}\Rightarrow\Delta_{\mathsf{W}_{2}}$.
By the naturality of $\tilde{\omega}$,  $F(x)\otimes\tilde{\gamma}_{\bullet}(y)=\tilde{\gamma}_{\bullet}(x\otimes y)$
for all $x,y\in\mathsf{W}_{1}$. Taking $y$ to be the unit, we obtain
$F(x)=\tilde{\gamma}_{\bullet}(x)$. Hence $(F,\omega,\tilde{\omega})$
is determined by $\omega$.

Consider a DisiMonSiMon category $(\mathsf{C},O,\tilde{O})$. For
$P\in\mathsf{C}$, consider $\mathrm{Ob}(L(O(P)))$, which is the
codomain of the universal $\mathsf{OrdCMon}$ entropy of $O(P)$ as
a set. Let $(!,h_{P})$ be the universal $\mathsf{OrdCMon}$ entropies
of $O(P)$, where $h_{P}:N_{P}\Rightarrow\Delta_{L(O(P))}$ is a monoidal
natural transformation. Let 
\[
S=\mathrm{Ob}(\mathrm{Arr}(\mathsf{C}))\sqcup\bigsqcup_{P\in\mathsf{C}}\mathrm{Ob}(L(O(P)))\in\mathsf{Set}
\]
be the disjoint set union of the set of morphisms in $\mathsf{C}$
and the aforementioned codomains. Let $\mathrm{F}(S)$ be the commutative
monoid generated by the set $S$, with product denoted as $+$. Consider
a binary relation $\mathcal{R}$ over $\mathrm{F}(S)$, where the
following pairs are related for every $P\in\mathsf{C}$:
\begin{itemize}
\item $(gf,\,g+f)$ and $(g+f,\,gf)$ for $P_{1}\stackrel{f}{\to}P_{2}\stackrel{g}{\to}P_{3}$
in $\mathsf{C}$;
\item $(A,B)$ for every morphism $g:A\to B$ in $L(O(P))$;
\item $(A\otimes B,\,A+B)$ and $(A+B,\,A\otimes B)$ for every $A,B\in L(O(P))$;
\item $(g+h_{P,E}(X),\,h_{Q,G(E)}(\gamma_{E}(X)))$ and $(h_{Q,G(E)}(\gamma_{E}(X)),\,g+h_{P,E}(X))$
for $g:P\to Q$ in $\mathsf{C}$, where we let $O(P)$ be $(\mathsf{D}_{P},N_{P})$,
$\tilde{O}(g)$ be $(G,\gamma)$, and consider every $E\in\mathsf{D}_{P}$,
$X\in N_{P}(E)$ (see \eqref{eq:chain_gen});
\item $(I_{L(O(P))},\,I)$ and $(I,\,I_{L(O(P))})$, where $I$ and $I_{L(O(P))}$
are the monoidal units of $\mathrm{F}(S)$ and $L(O(P))$ respectively.
\end{itemize}
Consider the transitive monoidal closure $\mathrm{Tra}(\mathrm{Mon}(\mathcal{R}))$
of $\mathcal{R}$, which is the smallest monoidal preorder containing
$\mathcal{R}$. Let $L_{\mathrm{DiM}}((\mathsf{C},O,\tilde{O}))\in\mathsf{OrdCMon}$
be the symmetric monoidal posetal category where objects are equivalent
classes in $\mathrm{Tra}(\mathrm{Mon}(\mathcal{R}))$, tensor product
is given by $+$, and $X\ge Y$ (there is a morphism $X\to Y$) if
$(X,Y)\in\mathrm{Tra}(\mathrm{Mon}(\mathcal{R}))$. It is straightforward
to check that $L_{\mathrm{DiM}}$ is left adjoint to $R_{\mathrm{DiM}}$.
The rest of the proof is straightforward and omitted.
\end{proof}
\medskip{}

\subsection{Universal Conditional Entropy}

We can define the universal conditional $\mathsf{V}$-entropy in the
same manner as Definition \ref{def:gcm_entropy}.

\medskip{}

\begin{defn}
[Universal conditional $\mathsf{V}$-entropy]\label{def:gcm_condentropy}Consider
a DisiMonSiMon category $\mathsf{C}\in\mathsf{DisiMonSiMonCat}$,
and a reflective subcategory $\mathsf{V}$ of $\mathsf{OrdCMon}$,
the \emph{universal conditional} $\mathsf{V}$\emph{-entropy} of $\mathsf{C}$
is the reflection morphism from $\mathsf{C}$ to $\mathsf{V}$. 
\end{defn}

\medskip{}

We now prove that the conditional Shannon entropy is the universal
conditional $\mathsf{IcOrdCMon}$-entropy of $\mathsf{FinCondProb}$,
the category of finite conditional probability spaces.

\medskip{}

\begin{thm}
\label{thm:cond_ent}The universal conditional $\mathsf{IcOrdCMon}$-entropy
of $\mathsf{FinCondProb}$ is given by the conditional Shannon entropy
as a DisiMonSiMon functor $(F,\omega,\tilde{\omega}):\mathsf{FinCondProb}\to R_{\mathrm{DiM}}(\mathbb{R}_{\ge0})$,
where $F:\mathsf{FinProb}\to\mathbf{B}(\mathbb{R}_{\ge0})$ maps the
morphism $f:P\to Q$ to $H_{1}(P)-H_{1}(Q)$, the MonSiMon functor
$\omega_{Q}:(-/(\mathsf{FinProb}/Q))\to R(\mathbb{R}_{\ge0})$ is
given by $(!,h_{Q})$ with a component $h_{Q,f}$ for $f\in\mathsf{FinProb}/Q$
given as
\[
h_{Q,f}\left(\begin{array}{ccccc}
 &  & A\\
 & \!\!\!\stackrel{x}{\nearrow}\!\!\! &  & \!\!\!\stackrel{a}{\searrow}\!\!\!\\
P\! &  & \!\!\stackrel{f}{\longrightarrow}\!\! &  & \!Q
\end{array}\right)=H_{1}(A)-H_{1}(Q),
\]
and $\tilde{\omega}_{Q}=U_{\mathrm{MM}}(\omega_{Q})$.
\end{thm}

\begin{proof}
It is straightforward to check that $(F,\omega,\tilde{\omega})$ is
a DisiMonSiMon functor. To check the universal property, we now consider
any DisiMonSiMon functor $(F,\omega,\tilde{\omega}):\mathsf{FinProb}\to R_{\mathrm{DiM}}(\mathsf{W})$
where $\mathsf{W}\in\mathsf{IcOrdCMon}$. Assume the MonSiMon functor
$\omega_{Q}:(-/(\mathsf{FinProb}/Q))\to R(\mathsf{W})$ is given by
$(!,h_{Q})$, $h_{Q}:-/(\mathsf{FinProb}/Q)\Rightarrow\Delta_{\mathsf{W}}$.
Consider the degenerate distribution $1\in\mathsf{FinProb}$. Note
that $O(1)=-/(\mathsf{FinProb}/1)$ is isomorphic to $-/\mathsf{FinProb}$
as a MonSiMon category. By Theorem \ref{thm:finprob_ic}, we can let
$\omega_{1}$ be the composition
\[
O(1)\stackrel{(!,(s_{0},s_{1}))}{\longrightarrow}R(\mathsf{W}_{1})\stackrel{R(t)}{\longrightarrow}R(\mathsf{W}),
\]
where $(!,(s_{0},s_{1}))$ is the universal $\mathsf{IcOrdCMon}$-entropy
(pairing of Hartley and Shannon entropies) in \ref{thm:finprob_ic},
$\mathsf{W}_{1}:=(H_{0},H_{1})(\mathsf{FinProb})\subseteq(\log\mathbb{N})\times\mathbb{R}_{\ge0}$,
and $t:\mathsf{W}_{1}\to\mathsf{W}$ is a morphism in $\mathsf{IcOrdCMon}$. 

For $Q\stackrel{!}{\to}1$ in $\mathsf{FinProb}$, let 
\begin{equation}
H(Q):=h_{1,!}(\mathrm{id}_{Q}).\label{eq:hQ_def}
\end{equation}
By \eqref{eq:chain_finprob}, for every $P\stackrel{f}{\to}Q\stackrel{!}{\to}1$
in $\mathsf{FinProb}$, 
\[
h_{1,!}(\mathrm{id}_{Q})+h_{Q,f}(\mathrm{id}_{P})=h_{1,!f}(\mathrm{id}_{P}).
\]
Therefore,
\begin{equation}
h_{Q,f}(\mathrm{id}_{P})=H(P)-H(Q).\label{eq:thm_chain}
\end{equation}
Let $P$ be a probability distribution over $\{0,1\}^{3}$ where 
\[
P(0,0,0)=P(0,1,0)=P(1,0,0)=1/4,
\]
\[
P(1,1,0)=1/4-\epsilon,
\]
\[
P(1,1,1)=\epsilon,
\]
where $0<\epsilon<1/4$. Let $Q$ be the uniform distribution over
$\{0,1\}$. Let $f:P\to Q$ be given by the measure-preserving mapping
$(x,y,z)\mapsto z$. Consider $a,b\in f/(\mathsf{FinProb}/Q)$, $a:P\to A$,
$b:P\to B$, given by the measure-preserving mappings $(x,y,z)\mapsto(z,x)$
and $(x,y,z)\mapsto(y,z)$ respectively. Note that the categorical
product $a\times b$ over the category $f/(\mathsf{FinProb}/Q)$ is
$\mathrm{id}_{P}$. We have the following diagram:\\
\[\begin{tikzcd}
	A & P & B \\
	& Q
	\arrow["f", from=1-2, to=2-2]
	\arrow["a"', from=1-2, to=1-1]
	\arrow["b", from=1-2, to=1-3]
	\arrow["{\pi_1}"', from=1-1, to=2-2]
	\arrow["{\pi_2}", from=1-3, to=2-2]
\end{tikzcd}\]where $\pi_{1},\pi_{2}$ are given by the measure-preserving mappings
$(z,x)\mapsto z$ and $(y,z)\mapsto z$ respectively. By the subadditivity
property of the MonSiMon functor $\omega_{Q}:\mathsf{FinProb}/Q\to\mathsf{W}$,
we have
\[
h_{Q,f}(a)+h_{Q,f}(b)\ge h_{Q,f}(\mathrm{id}_{P}).
\]
By \eqref{eq:invar_extend},
\[
h_{Q,\pi_{1}}(\mathrm{id}_{A})+h_{Q,\pi_{2}}(\mathrm{id}_{B})\ge h_{Q,f}(\mathrm{id}_{P}).
\]
Applying \eqref{eq:thm_chain},
\begin{equation}
H(A)+H(B)\ge H(P)+H(Q),\label{eq:thm_submod}
\end{equation}
which is the submodularity inequality \cite{fujishige1978polymatroidal}.
Note that the Hartley and Shannon entropies of $P,A,B,Q$ are 
\[
H_{0}(P)=\log5,\;H_{1}(P)=\log4+\delta_{P}(\epsilon),
\]
\[
H_{0}(A)=\log3,\;H_{1}(A)=\log2+\delta_{A}(\epsilon),
\]
\[
H_{0}(B)=\log3,\;H_{1}(B)=\log2+\delta_{A}(\epsilon),
\]
\[
H_{0}(Q)=\log2,\;H_{1}(Q)=\delta_{Q}(\epsilon),
\]
where $\delta_{P}(\epsilon),\delta_{A}(\epsilon),\delta_{Q}(\epsilon)\to0$
as $\epsilon\to0$. Hence, \eqref{eq:thm_submod} gives
\[
t(\log9,\log4+2\delta_{A}(\epsilon))\ge t(\log10,\log4+\delta_{P}(\epsilon)+\delta_{Q}(\epsilon)).
\]
Writing $\delta(\epsilon):=2\delta_{A}(\epsilon)$, $\delta'(\epsilon):=\delta_{P}(\epsilon)+\delta_{Q}(\epsilon)$,
we have
\[
t(\log9,\,\log4+\delta(\epsilon))\ge t(\log10,\,\log4+\delta'(\epsilon)).
\]
We then prove that for $(a_{1},b_{1}),(a_{2},b_{2})\in\mathsf{W}_{1}$,
if $b_{1}>b_{2}$, then we have $t(a_{1},b_{1})\ge t(a_{2},b_{2})$.
To prove this, consider $n\in\mathbb{N}$ large enough such that $n\log10+a_{1}>n\log9+a_{2}$.
Let $\epsilon$ be small enough such that $n\delta'(\epsilon)+b_{1}>n\delta(\epsilon)+b_{2}$.
Since $t$ is an order-preserving homomorphism,
\begin{align*}
 & nt(\log9,\,\log4+\delta(\epsilon))+t(a_{1},b_{1})\\
 & \ge nt(\log10,\,\log4+\delta'(\epsilon))+t(a_{1},b_{1})\\
 & =t(n\log10+a_{1},\,n\log4+n\delta'(\epsilon)+b_{1})\\
 & \ge t(n\log9+a_{2},\,n\log4+n\delta(\epsilon)+b_{2})\\
 & =nt(\log9,\,\log4+\delta(\epsilon))+t(a_{2},b_{2}).
\end{align*}
Since $\mathsf{W}$ is cancellative, $t(a_{1},b_{1})\ge t(a_{2},b_{2})$.

We now prove that $t(a,b)$ does not depend on $a$. Fix $a_{1},a_{2},b$
such that $(a_{1},b),(a_{2},b)\in\mathsf{W}_{1}$. For any $n\in\mathbb{N}$,
we have $nb+\log2>nb$, and hence $nt(a_{1},b)+t(\log2,\log2)=t(na_{1}+\log2,nb+\log2)\ge t(na_{2},nb)=nt(a_{2},b)$.
By integral closedness, $t(a_{1},b)\ge t(a_{2},b)$. We also have
$t(a_{2},b)\ge t(a_{1},b)$, and hence $t(a_{1},b)=t(a_{2},b)$. Hence,
$t$ can be factorized as
\[
\mathsf{W}_{1}\stackrel{(!,(s_{0},s_{1}))}{\longrightarrow}\mathbb{R}_{\ge0}\stackrel{t'}{\longrightarrow}\mathsf{W}.
\]
Such $t'$ is clearly an order-preserving homomorphism and is unique.
Therefore, the conditional Shannon entropy functor is the universal
DisiMonSiMon functor from $\mathsf{FinCondProb}$ to $\mathsf{SOrdVect}_{\mathbb{Q}}$.
\end{proof}
\medskip{}

We can also show the same result for $\mathsf{LCondProb}_{\rho}$
and $\mathsf{HCondProb}$ (the conditional versions of $\mathsf{LProb}_{\rho}$
and $\mathsf{HProb}$ defined in the same manner as $\mathsf{FinCondProb}$;
see Section \ref{sec:lrho} for the definitions of $\mathsf{LProb}_{\rho}$
and $\mathsf{HProb}$). The proof is the same as Theorem \ref{thm:cond_ent}
and is omitted.

\medskip{}

\begin{prop}
The universal conditional $\mathsf{IcOrdCMon}$-entropy of $\mathsf{LCondProb}_{\rho}$
(where $0\le\rho<1$) is given by the conditional Shannon entropy
as a DisiMonSiMon functor $(F,\omega,\tilde{\omega}):\mathsf{LCondProb}_{\rho}\to R_{\mathrm{DiM}}(\mathbb{R}_{\ge0})$,
where $F:\mathsf{LProb}_{\rho}\to\mathbf{B}(\mathbb{R}_{\ge0})$ maps
the morphism $f:P\to Q$ to $H_{1}(P)-H_{1}(Q)$, the MonSiMon functor
$\omega_{Q}:(-/(\mathsf{LProb}_{\rho}/Q))\to R(\mathbb{R}_{\ge0})$
is given by $(!,h_{Q})$ with a component $h_{Q,f}$ for $f\in\mathsf{LProb}_{\rho}/Q$
given as
\[
h_{Q,f}\left(\begin{array}{ccccc}
 &  & A\\
 & \!\!\!\stackrel{x}{\nearrow}\!\!\! &  & \!\!\!\stackrel{a}{\searrow}\!\!\!\\
P\! &  & \!\!\stackrel{f}{\longrightarrow}\!\! &  & \!Q
\end{array}\right)=H_{1}(A)-H_{1}(Q),
\]
and $\tilde{\omega}_{Q}=U_{\mathrm{MM}}(\omega_{Q})$. The same is
true for $\mathsf{HCondProb}$ in place of $\mathsf{LCondProb}_{\rho}$.
\end{prop}

\medskip{}

\section{Conclusion and Discussions}

In this paper, we studied the notion of monoidal strictly-indexed
monoidal (MonSiMon) categories that captures the notions of product
distributions and joint random variables. The universal entropy of
a MonSiMon category is then defined as the universal MonSiMon functor
to some reflective subcategory of $\mathsf{MonSiMonCat}$ (the category
of MonSiMon categories) such as $\mathsf{IcOrdCMon}$ and $\mathsf{IcOrdAb}$,
which can also be given as a monoidal natural transformation satisfying
a universal property. For example, the pairing of the Hartley and
Shannon entropies is the univeral entropy $\mathsf{FinProb}$, and
the Shannon entropy is the univeral entropy $\mathsf{LProb}_{\rho}$
($0<\rho<1$). We also introduced the discretely-indexed monoidal
strictly-indexed monoidal (DisiMonSiMon) categories, where ``over
categories'' are MonSiMon categories, to capture the notion of conditional
product distributions. We show that the conditional Shannon entropy
is the universal conditional entropy of $\mathsf{FinProb}$ as a DisiMonSiMon
category. Moreover, the definition of universal entropy allows us
to link the universal entropies of different MonSiMon categories together
in a natural manner.

Although we have shown that the universal $\mathsf{IcOrdCMon}$ and
$\mathsf{IcOrdAb}$-entropy always exists for all MonSiMon category,
finding these entropies is often nontrivial. For future studies, it
may be of interest to find the universal entropies of some categories
not covered in this paper, such as category of finite groups, the
category of finite partial orders, the category of finite simplicial
sets, and the category of finite categories.

Linear entropy inequalities are linear inequalities on the Shannon
entropy over $\mathsf{FinProb}$ \cite{yeung1997framework,zhang1998characterization}.
One potential research direction is to investigate the sets of true
linear inequalities over the universal entropies of different DisiMonSiMon
categories, and how they depend on the category concerned. While some
inequalities are true for every DisiMonSiMon category (e.g., the subadditivity
inequality $H(X,Y)\le H(X)+H(Y)$ which follows from the definition
of universal entropy), it is possible that there are some inequalities
that are specific to $\mathsf{FinProb}$, which are necessarily non-Shannon
inequalities \cite{zhang1998characterization,yeung2012first}. Also,
while conditional linear entropy inequalities over $\mathsf{FinProb}$
are undecidable \cite{li2022undecidabilityaffine,li2023undecidability,kuhne2022entropic},
it is uncertain whether the same is true for other DisiMonSiMon categories.
The category-theoretic framework in this paper, which allows us to
keep track of relations and operations on probability distributions
and random variables, may also be useful for symbolic algorithms for
the automated proof of entropy inequalities \cite{yeung1996itip,yeung1997framework,li2021automatedisit,yeung2021machine}.

We may also investigate whether the universal property, or a modified
version of it, can be applied to von Neumann entropy of quantum states.
An obstacle is the no-broadcast theorem \cite{barnum1996noncommuting},
which forbids us from combining two quantum operations, which is necessary
to discuss the subadditivity of entropy in the sense in this paper.
It would be interesting to study whether we can use notions such as
graded monoidal categories \cite{parzygnat2020inverses} to modify
the construction of the DisiMonSiMon category in this paper to accommodate
quantum states.

We may also study whether the differential entropy satisfies a universal
property. The differential entropy lacks a number of important properties
that the discrete entropy satisfies. For example, differential entropy
is not invariant under bijective transformations. Therefore, it is
not immediately clear how a MonSiMon category of continuous probability
distributions can be defined so that the differential entropy is given
by a MonSiMon functor.

\medskip{}

\section{Acknowledgement}

This work was partially supported by two grants from the Research
Grants Council of the Hong Kong Special Administrative Region, China
{[}Project No.s: CUHK 24205621 (ECS), CUHK 14209823 (GRF){]}. The
author would like to thank Tobias Fritz for his invaluable comments.
Most diagrams in this paper are drawn using Quiver \cite{Arkor_quiver_2024}
or Ipe \cite{ipe_2023}.

\medskip{}

\bibliographystyle{IEEEtran}
\bibliography{ref}

\end{document}